\documentclass[11pt,reqno,twoside]{amsart}

\usepackage{amsfonts,amsmath,amssymb}
\usepackage{mathrsfs,mathtools}
\usepackage{enumerate}
\usepackage[colorlinks=true, citecolor=black, urlcolor=black, linkcolor=black, pdfborder={0,0,0}]{hyperref}
\usepackage{cleveref}
\usepackage{xcolor}
\usepackage{esint}
\usepackage{graphicx}
\usepackage{bm}
\usepackage{commath}
\usepackage{esint}
\usepackage{cases}
\DeclareMathAlphabet{\mathpzc}{OT1}{pzc}{m}{it}
\usepackage[colorlinks=true]{hyperref}

\usepackage{placeins} 
\usepackage{flafter}

\usepackage[textsize=small]{todonotes}
\definecolor{darkgreen}{RGB}{0,95,0}

\usepackage{lineno}
\numberwithin{equation}{section}

\usepackage[dvips,bottom=1.4in,right=1in,top=1in, left=1in]{geometry}

\newtheorem{theorem}{Theorem}[section]

\newtheorem{definition}[theorem]{Definition}
\newtheorem{lemma}[theorem]{Lemma}

\newtheorem{proposition}[theorem]{Proposition}

\newtheorem{remark}[theorem]{Remark}
\newtheorem{example}{Example}
\crefformat{equation}{(#2#1#3)}

\renewcommand{\div}{\mathrm{div}}
\newcommand{\supp}{\mathrm{supp}}

\newcommand{\dx}{\,\textup{d}x}
\newcommand{\da}{\,\textup{d}\alpha}

\newcommand{\BV}{\textup{BV}}

\def\RR{{\mathbb{R}}}
\def\RRR{\RR \cup {\{\infty\}}}

\def\Om{\Omega}

\def\al{\alpha}

\def\ga{\gamma}
\def\Ga{\Gamma}

\newcommand{\V}{V}
\newcommand{\Va}{{\V_{\Gamma_{N}}(\Omega)}}

\usepackage{graphicx}

\thanks{
This work is partially supported by NSF grants DMS-1818772, DMS-2012391, DMS-1913004, the Air Force Office of Scientific Research (AFOSR) under Award NO: FA9550-19-1-0036, and Department of Navy, Naval Postgraduate School
under Award NO: N00244-20-1-0005.
}

\begin{document}
	%%%%%%%%%%%%%%%%%%%%%%%%%%%%%%%%%%%%%%%%%%%%%%%%%%%%%%%

	%%%%%%%%%%%%%%%%%%%%%%%%%%%%%%%%%%%%%%%%%%%%%%%%%%%%%%%
	\title[Variational Problems with Distributional and Weak Gradient Constraints]{{Non-diffusive Variational Problems \\ with Distributional and Weak Gradient Constraints}
	%: \\Analysis, Algorithms, and Applications
	}
	
	%%%%%%%%%%%%%%%%%%%%%%%%%%%%%%%%%%%%%%%%%%%%%%%%%%%%%%%
	\author{Harbir Antil, Rafael Arndt, Carlos N. Rautenberg, Deepanshu Verma}
	\address{H. Antil, R. Arndt, C.N. Rautenberg, D. Verma. Department of Mathematical Sciences and the Center for Mathematics and Artificial Intelligence (CMAI), George Mason University, Fairfax, VA 22030, USA.}
	\email{hantil@gmu.edu, tarndt@gmu.edu, crautenb@gmu.edu, dverma2@gmu.edu}
	%%%%%%%%%%%%%%%%%%%%%%%%%%%%%%%%%%%%%%%%%%%%%%%%%%%%%%%
	
	\begin{abstract}
		In this paper, we consider non-diffusive variational problems with mixed boundary conditions and (distributional and weak) gradient constraints. The upper bound in the constraint is either a function or a Borel measure, leading to the state space being a Sobolev one or the space of functions of bounded variation. We address existence and uniqueness of the model under low regularity assumptions, and rigorously identify its Fenchel pre-dual problem. The latter in some cases  {is posed on a} non-standard space of Borel measures with square integrable divergences.  {We also establish existence and uniqueness of solutions to this pre-dual problem under some assumptions. We conclude the paper by introducing a mixed finite-element method to solve the primal-dual system. The numerical examples confirm our theoretical findings.}		 
	\end{abstract}
	
		
	\maketitle

	\tableofcontents
	
	%%%%%%%%%%%%%%%%%%%%%%%%%%%%%%%%%%%%%%%%%%%%%%%%%%%%%%%%
	\section{Introduction} \label{Sec:1}
	%%%%%%%%%%%%%%%%%%%%%%%%%%%%%%%%%%%%%%%%%%%%%%%%%%%%%%%%

	We begin by considering an evolutionary problem whose semi-discretization (in time) gives rise to the class of stationary problems of interest in this paper. {Suppose that} $f:(0,T)\times \Omega \to \mathbb{R}$ together with $u_0:\Omega\to \mathbb{R}$ {are} given, where $\Omega\subset \mathbb{R}^\mathrm{d}$ is {a bounded domain  with a Lipschitz boundary}. Furthermore, let $\alpha$ be a given nonnegative function (possibly only  integrable), or a nonnegative Borel measure in $\Omega$. Suppose that $u:(0,T)\times\Omega\to\mathbb{R} $, such that $u(0)=u_0$, is a solution to the following problem
	 \begin{equation}\label{eq:Evo_u_1st}%\tag{$\tilde{\mathbb{P}}$}
\text{Find }u\in \mathcal{K} \text{ such that } \:\int_0^T\left(\partial _tu(t) - f(t), v(t)-u(t)\right)_{{L^2(\Omega)}}\,\textup{d}t \geq 0,\: \text{ for all } v\in \mathcal{K},
% W_{\Gamma_0}^{1,1}(\Omega)\right)
\end{equation}
where the set $\mathcal{K}$ is given by
\begin{equation}\label{eq:K}
	\mathcal{K} :=\tilde{U}(0,T)\cap \{w\::\: w(t)\in K \text{ almost everywhere} \}.
\end{equation}
The choice of $\tilde{U}(0,T)$ and $K$  in \eqref{eq:K} hinges on the type  of the boundary conditions and the regularity of  $\alpha$. We assume that the boundary $\partial\Omega$ is partitioned
	into a Dirichlet boundary part $\Gamma_D$ and a non-Dirichlet boundary part $\Gamma_N$, both composed of a finite number of connected parts, such that 
	\begin{align*}
	{\overline{\Gamma}_D\cup\overline{\Gamma}_N} = \partial\Omega,\qquad\text{and}\qquad
	\Gamma_D\cap\Gamma_N = \emptyset.
	\end{align*}
%Note that we use the subscript $N$ {on} $\Gamma_N$, but no Neumann boundary condition is assumed in place.  
Notice that on $\Gamma_N$, we do not necessarily prescribe Neumann boundary conditions, as it is later clarified. 
However, %a quantity, 
a conservation law of material is in place in the case $\Gamma_D=\emptyset$; specifically, it can be inferred  from \eqref{eq:Evo_u_1st}  that $\int_{\Omega}(u(T)-u_0)\dif x =\int_0^T\int_{\Omega}f\dif x \dif t$ given that $v=u \pm 1$ are admissible test functions  as we see next.
The restriction of $u$ to the $\Gamma_D$ part of the boundary is assumed to be zero, and no restrictions are assumed on $\Gamma_N$. 

The set $K$ is convex and it arises by a nonlinear law {with a bound on the} first order derivative terms. In the most general form $K$ is given by 
\begin{equation}\label{eq:SetK}
	K= \{v\in U_{\Gamma_D}(\Omega) : |G v |_{p}\leq\alpha\},
\end{equation}
with $1\leq p\leq +\infty$.  {We briefly discuss the two possible scenarios that we consider: }
\begin{enumerate}[\upshape(I)]
  \item\label{itm:1ini} If $\alpha$ is a {\bf nonnegative measurable function}, then $U_{\Gamma_D}(\Omega)$ is a Sobolev-type space and $G=\nabla$ is the weak  gradient, so that $|\nabla v|_{{p}}$ is the $\ell_{p}$-norm of the weak gradient {of $v$}. Hence, $|\nabla v|_{{p}}\leq\alpha$  in \eqref{eq:SetK}  is considered in the almost everywhere (a.e.) in $\Omega$ sense.
    \item\label{itm:2ini} If $\alpha$ is a {\bf nonnegative Borel measure}, then $U_{\Gamma_D}(\Omega)$ is a subset of functions of bounded variation  $\BV(\Omega)$. In this case, $G=\mathrm{D}$ is the distributional gradient, and $|\mathrm{D} v|_{{p}}$ the total variation measure of $\mathrm{D}v$ associated to the $\ell_{p}$-norm, and the  constraint  $|\mathrm{D} v|_{p}\leq\alpha$ is understood in the measure sense.
  \end{enumerate}
  Both instances, \eqref{itm:1ini} and \eqref{itm:2ini} are related, in fact \eqref{itm:1ini} may be considered  as a  special case of \eqref{itm:2ini}. Furthermore,  {letting $\alpha\in \mathrm{M}^+(\Omega)$ in case \eqref{itm:2ini}, where $\mathrm{M}^+(\Omega)$ denotes the set of nonnegative Borel measures, enables us} to handle the delicate case $\alpha\in L^1(\Omega)^+$ in \eqref{itm:1ini}.  
Next we shall provide a brief description of modeling capabilities of \eqref{itm:1ini} and \eqref{itm:2ini} in the context of a particular application.

    \begin{figure}\label{fig:image}
  	\includegraphics[scale=0.35]{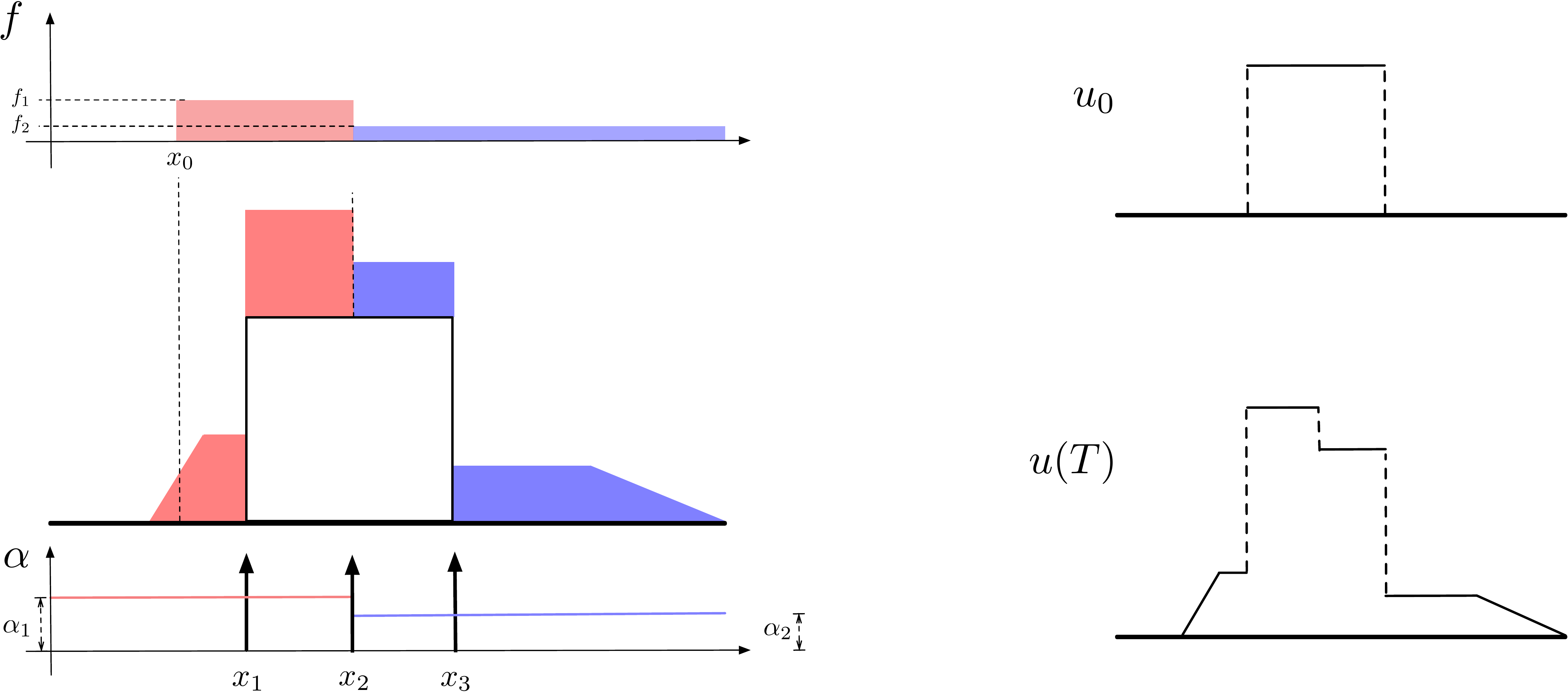}
  	\caption{Accumulation of two kinds (magenta and blue) of granular materials on discontinuous surface. (\textbf{LEFT}) Depiction of $f(t,x)=f_1\chi_{(x_0,x_2)}(x)+f_2\chi_{(x_2,1)}(x)$, the accumulation of both materials, and $\dif \alpha=\alpha_1\chi_{(x_1,x_2)}(x)\dif x+\alpha_2\chi_{(x_2,1)}(x)\dif x+ \sum_{i=1}^3 \delta(x-x_i)$. (\textbf{RIGHT}) The value of the initial supporting structure $u_0$ and the the final distribution $u(T)$. }
  \end{figure}
  
  {A possible motivation for } the above class of problems is based on the study of accumulation of granular heterogeneous material on possibly discontinuous structures. This approach was pioneered by Prigozhin \cite{Prigozhin1994,Prigozhin1996,Prigozhin1996a} in the case of homogeneous materials and a continuous support structure. In this vein, $f:(0,T)\times \Omega \to \mathbb{R}$ represents the {(density)}  rate of a granular material {being} deposited on a supporting structure $u_0:\Omega\to \mathbb{R}$.  {Moreover}, $\int_0^T\int_{\Omega}f\dif x \dif t$ is the total amount of material deposited on $\Omega$  over the time interval~$[0,T]$. {In case that} $\alpha>0$ is a real number, {this} corresponds to the classical case of a granular cohesionless material where homogeneous piles are generated. If $\alpha : \Omega \rightarrow \mathbb{R}$ {is not constant zero}, the value of $\alpha$ at a point determines the angle of repose of the material at that point, i.e., the steepness of a cone generated from a point source of material. This is the case  for heterogenous  sandpiles \cite{bocea2012models} and also a restricted case of the quasi-variational sandpile model; see \cite{BarrettPrigozhinSandpile,MR3231973,Prigozhin1986,MR3335194}. {In a more general setting, where $\alpha$ is a measure, using the approach in this paper, it is possible to generate discontinuous structures such as cliffs by preserving discontinuities in the initial supporting structure~$u_0$ and/or of $f$. Such an approach has not  yet been considered by the literature to the best of our knowledge. 
  
  	A description of the qualitative behavior of Problem \eqref{eq:Evo_u_1st} is displayed in Figure \ref{fig:image}.
  	We assume two materials with different angles of repose~$\alpha_1$ and $\alpha_2$ with $\alpha_1> \alpha_2$ are poured on the discontinuous structure $u_0(x):=\chi_{(x_1,x_2)}(x)$ for $x\in \Omega:=(0,1)$ and $0<x_1<x_2<1$. The intensity of the material being deposited is given by $f(t,x)=f_1\chi_{(x_0,x_2)}(x)+f_2\chi_{(x_2,1)}(x)$ for some points $x_0$, and $x_2$, and some $f_1,f_2>0$, i.e., the first and second materials are poured with density rates~$f_1$ and $f_2$, respectively, during the entire time interval~$(0,T)$. We further assume that a sharp edge can form at $x_2$ with maximum height of $1$, and in addition discontinuities of maximum size $1$ can be preserved at the locations of the discontinuities of $u_0$. Finally, the the gradient constraint $\alpha$ is then given by $\dif \alpha=\alpha_1\chi_{(0,x_2)}(x)\dif x+\alpha_2\chi_{(x_2,1)}(x)\dif x+ \sum_{i=1}^3 \delta(x-x_i)$, and the material is assumed to escape freely at the boundary points of $\Omega$. On the right side of Figure \ref{fig:image}, we see the comparison between $u_0$ and $u(T)$, the solution at time $T>0$; on the left we see the depiction of $f$, $\alpha$, and the accumulation regions of both materials.

  The study of solutions to \eqref{eq:Evo_u_1st} usually makes use of the semi-discretization (in time) of the problem via an implicit Euler method. In particular, we approximate the partial time derivative $\partial_t u$ by $(u^{n}-u^{n-1})/k$ for some  {time-step} $k>0$. The class of problems of interest in this paper is then given by
  	\begin{equation}\label{eq:original_problem_ini}\tag{$\mathbb{P}$}
  	 \begin{split}
	&\text{Minimize (min) }\quad \frac12\int_\Om|u(x)|^2 \dif x - \int_\Omega f(x) u(x) \dif x \quad \text{ over } u\in U_{\Gamma_D}(\Omega), \\
	&\text{subject to (s.t.) } \quad u\in K . 
	\end{split}
	\end{equation}
Notice that \eqref{eq:original_problem_ini} can be seen as the time-discrete version of \eqref{eq:Evo_u_1st} where 
the solution $u$ to \eqref{eq:original_problem_ini} is equal to $u^n$ when $f$ corresponds to $\int_{(n-1)k}^{nk}f(\tau)\dif \tau+ku^{n-1}$ {with $u^{n-1}$ given}. Closely related to the problem above, we   consider the following class of problems 
  	\begin{equation}\label{eq:predual_problem_ini}\tag{$\mathbb{P}^*$}
		\min\:\frac12\int_\Omega|\div\:\bm p(x) - f(x)|^2\dif x + J(\bm p) \quad \text{ over } {\bm p\in V_{\Gamma_{N}}(\Omega) }. 
	\end{equation}
	 {We prove that \eqref{eq:predual_problem_ini} is the \emph{Fenchel pre-dual} of problem \eqref{eq:original_problem_ini}, i.e., the Fenchel dual \cite{ekeland_temam}} of \eqref{eq:predual_problem_ini} under certain conditions is  \eqref{eq:original_problem_ini}. Several choices for $V_{\Gamma_{N}}(\Omega) $ and $J$ are explored {which are directly} related to the nature of $\alpha$. In all cases considered, $V_{\Gamma_{N}}(\Omega)$ contains $\mathrm{d}$-dimensional vector fields with divergences in $L^2(\Omega)$. In particular, we consider
\begin{enumerate}[\upshape(i)] 
  \item\label{itm:i} If $\alpha$ is a {\bf nonnegative measurable function} (additional assumptions are later explained but continuity is enough to guarantee what follows), then  {we explore two options for  $J$:} 
\begin{equation*}
	J(\bm p)=\int_\Omega \alpha(x) |\bm p(x)|_q \dif x, \qquad \text{ and } \qquad J(\bm p)=\int_\Omega \alpha \dif|\bm p|_q.
\end{equation*}  
In the first case $V_{\Gamma_{N}}(\Omega)$ is a subspace of $L^1(\Omega)^\mathrm{d}$. In the second case $V_{\Gamma_{N}}(\Omega)$ is contained in the space of $\mathbb{R}^\mathrm{d}$-valued Borel measures, so that the second functional denotes the integral of $\alpha$ with respect to the total variation measure of $\bm p$ induced by the $\ell^q$-norm. The two functionals are closely related, and the first can be seen as a restriction of the second one to measures that are absolutely continuous with respect to the Lebesgue measure.
  \item\label{itm:ii} If  $\alpha$ is a {\bf nonnegative Borel measure}, then $V_{\Gamma_D}(\Omega)$ is contained  {in the space of} maps that are $\alpha$ measurable, with $J$ given by
  \begin{equation*}
  	 J(\bm p)=\int_\Omega |\bm p|_q \dif \alpha.
  \end{equation*}
  \end{enumerate}
  
A few words are in order concerning \eqref{eq:original_problem_ini} and \eqref{eq:predual_problem_ini}. Although the objective functional in \eqref{eq:original_problem_ini} is smooth and amenable, the constraint set $K$ makes the  {entire} problem highly nonlinear and nonsmooth. The latter also holds for  \eqref{eq:predual_problem_ini} given the nature of the functional $J$. The development of solution algorithms for both problems is a rather delicate issue that requires appropriate regularization methods that can handle the nonsmothness in an asymptotic fashion.

  The paper focuses on  functional analytic properties of \eqref{eq:original_problem_ini} and \eqref{eq:predual_problem_ini} together with duality relationship properties.  {Additionally, we develop a mixed finite type method  to solve the optimality conditions corresponding to \eqref{eq:original_problem_ini} and \eqref{eq:predual_problem_ini}. } 
  
  \subsection*{Some Bibliography}
  
  The structure of Problems  \eqref{eq:predual_problem_ini} and  \eqref{eq:original_problem_ini} and their inherent difficulties are analogous to the ones that appear in the context of plasticity; see \cite{kt83,t85} and references therein. In particular, the first class of applications for diffusive variational problems with gradient constraints is the elasto-plastic torsion problem. Such a problem has been thoroughly analyzed by Br\'{e}zis, Caffarelli, Evans, Friedman,  Gerhardt, and others; see \cite{MR529814,MR544887,brezis1971equivalence,MR417887,MR385296,MR534111,MR630989,BrezisStamp}. Further, a complete account of the literature can be found in \cite{MR3347004}. A significant amount of the aforementioned works focuses on regularity of solutions, the free boundary, and the equivalence of the gradient constrained problem to a double obstacle one.
  
  The modeling of the evolution of the magnetic field in critical-state models of type-II superconductors also leads to a problem like \eqref{eq:Evo_u_1st} with the addition of a diffusive operator and a state-dependent constraint in some cases; see \cite{MR1765540,MR2652615,Prigozhin,Prigozhin1996,MR3023771,hintermuller2019dissipative,MR3119319}. See \cite{santos2002variational} for a study of evolutionary variational problems with non-constant gradient constraints, and \cite{Miranda_2018} for a complete account of evolutionary problems with derivative bounds.

Analogous problems are found  in mathematical imaging involving total variation regularization \cite{hs06,Hintermueller2004,bartels2017iterative} and more specifically in the weighted total variation version \cite{hintermuller_rautenberg_2017_optimal_I}. There, in contrast to the work here, the $L^\infty$-norm on the gradient is replaced by the $L^1$-norm, leading to a pre-dual problem with a pointwise bound in its state variable.  
  
 	\subsection{Organization of the paper.}	 Preliminaries are  provided, and some notations are made explicit in \Cref{Sec:preliminaries}; elementary results about the generalized gradient constraint are given in \Cref{sec:Grad}. In \Cref{Sec:existence}, we prove existence and uniqueness of the solution to problem \eqref{eq:original_problem_ini} for the cases when $\alpha$ is either a nonnegative Lebesgue measurable function or a nonnegative Borel measure. Existence of solutions to problem \eqref{eq:predual_problem_ini} is addressed in \Cref{Sec:existence_dual}, while for the case when $\bm p$ is a function we require $\mathrm{d}=1$, when $\bm p$ is a measure the dimension restriction is dropped.  
The relation between problems \eqref{eq:original_problem_ini} and \eqref{eq:predual_problem_ini} are considered  in \Cref{Sec:2final}, where a rigorous Fenchel duality result establishes a link between  these two problems.  In particular, in  \Cref{Ssec:pre_dual_function}, we address the  case  where $\alpha$ is a function and the  variable $\bm p$ is either a function or a measure. {The case when $\alpha$ is a measure and an extension of the duality result of the previous section is given in \Cref{Sec:3final}}. Finally in \Cref{Sec:5final}, we introduce a mixed finite element method to solve the underlying problems and present a range of numerical tests. 

	%%%%%%%%%%%%%%%%%%%%%%%%%%%%%%%%%%%%%%%%%%%%%%%%%%%%%%%%
	\section{Notation and Preliminaries} \label{Sec:preliminaries}
	%%%%%%%%%%%%%%%%%%%%%%%%%%%%%%%%%%%%%%%%%%%%%%%%%%%%%%%%

	% - boundary trace
	% - intermediate convergence

		The purpose of this section is  to introduce  notation involving spaces, and convergence notions that are used throughout  the paper; in particular,  we address the  well-known notions of Sobolev spaces and the space of functions of bounded variation. We refer the  reader  to Attouch et al. \cite{attouch} that we follow closely for this introduction together with  the book  of Adams and Fournier \cite{adams_fournier_2003}. 
	
	For a Banach space $X$, we denote its corresponding norm as $\|\cdot\|_X$.
	 %For an arbitrary Banach space $X$, \HA{the induced norm is denoted by} $\|\cdot\|_X$. 
	 For an element $F$ in the topological dual $X'$ of $X$, the duality pairing of $F$ and an arbitrary element $x\in X$ is written as $\langle F,x \rangle_{X',X}$. 
	 Throughout the paper, all Banach spaces are assumed to be real vector spaces.  
	
	The inner product on the Lebesgue space $L^2(\Omega)$ of (equivalence classes of) functions that are square integrable on $\Omega$  is denoted as $(\cdot, \cdot)$, so that $(f,g):=\int_{\Omega} f(x)g(x)\dif x$ for $f,g\in L^2(\Omega)$ where $\dif x$ refers to integration with respect to the Lebesgue measure.

	 The Sobolev space of functions in $L^r(\Omega)$ for $1\leq r<+\infty$ with weak gradients in $L^r(\Omega)^{\mathrm{d}}$ is denoted by $W^{1,r}(\Omega)$, and it is endowed with the norm
	\begin{equation*}
		\|v\|_{W^{1,r}(\Omega)}:=\|v\|_{L^{r}(\Omega)}+\|\nabla v\|_{L^{r}(\Omega)^{\mathrm{d}}},
	\end{equation*}
	where $\nabla v$ denotes the weak gradient of $v$. In the case $r=2$, we use the notation $H^1(\Omega):=W^{1,2}(\Omega)$.
	Given that $\Omega$ is assumed Lipschitz, restriction of a function $v\in W^{1,r}(\Omega)$ to the boundary $\partial\Omega$ is well-defined via the continuous trace map $\gamma_0: W^{1,r}(\Omega)\to L^{r}(\partial\Omega)$. Furthermore, the closed subspace of functions in $W^{1,r}(\Omega)$ that are zero on $\Gamma_D$ is denoted by  $W_{\Gamma_D}^{1,r}(\Omega)$, i.e., 
	\begin{equation*}
		W^{1,r}_{\Gamma_D}(\Omega):=\{v\in W^{1,r}(\Omega): \gamma_0(v)=0 \text{ on } \Gamma_D\}.
	\end{equation*}
	Similarly, we define $H^1_{\Gamma_D}(\Omega):=W^{1,2}_{\Gamma_D}(\Omega)$.
	
	The space of real-valued Borel measures $\mathrm{M}(\Omega)$ is endowed with the norm $\|\mu\|_{\mathrm{M}(\Omega)}:=|\mu|(\Omega)$, where $|\mu|$ is defined for an arbitrary open set $O$ as 
	\begin{equation*}
		|\mu|(O)=\text{sup}\left\{ \langle\mu,z \rangle_{\mathrm{M}(\Omega),C_0(\Omega)}\::\: z\in {C_0(\Omega) },\: \mathrm{supp}(z)\subset O,\; |z(x)|\le 1, \text{ for every }x\in O \right\}.
	\end{equation*}
	Note that $\langle\mu,z \rangle_{\mathrm{M}(\Omega),C_0(\Omega)}=\int_\Omega z\; \dif \mu$, and that $|\mu|$ defines a Borel measure in $\mathrm{M}^+(\Omega)$, the subset of non-negative elements of $\mathrm{M}(\Omega)$, i.e., $\sigma\in \mathrm{M}^+(\Omega)$ if $\sigma(B)\geq 0$ for every Borel set $B\subset \Omega$.
	
	We denote by $\BV(\Om)$, the space of functions $v$ in $L^1(\Omega)$, for which the total variation semi-norm 
	\[ \int_\Om |\mathrm{D}v|_{p}=\text{sup}\left\{ \int_\Om v\; \div\: \bm{p}\; \dif x\::\: \bm{p}\in {C^1_0(\Om)^\mathrm{d}},\; |\bm{p}(x)|_{q}\le 1, \text{ for every }x\in\Om \right\} \]
	is finite and where $q$ is the H\"{o}lder conjugate of $p$, i.e., $1/p+1/q=1$; see~\cite[Section~10.1]{attouch}. The space $\BV(\Om)$ is a Banach space endowed with the norm
	\begin{equation*}
		\|v\|_{\BV(\Om)}:=\|v\|_{L^1(\Omega)}+\int_\Om |\mathrm{D}v|_{p}.
	\end{equation*}
	The operator $\mathrm{D}$ represents the \emph{distributional gradient}, and for a $v\in \BV(\Omega)$, $\mathrm{D}v$ is a  $\RR^d$-valued Borel measure. We use $|\mathrm{D}v|_p$ to denote the total variation measure (associated to the $\ell^p$-norm) of $\mathrm{D}v$, and the total mass $|\mathrm{D}v|_p(\Omega)$ is by definition
	\begin{equation*}
		|\mathrm{D}v|_p(\Omega)=\int_\Om |\mathrm{D}v|_{p}.
	\end{equation*}
	Furthermore, {the Lebesgue decomposition result applied to $\mathrm{D}v$ implies that} there exist measures $\mathrm{D}_a v$ and $\mathrm{D}_s v$ such that 
	\[ \mathrm{D}v=\mathrm{D}_a v +\mathrm{D}_s v,\]
	with $\mathrm{D}_a v$ and $\mathrm{D}_s v$ respectively being absolutely continuous and singular with respect to the  $\mathrm{d}$-dimensional  Lebesgue measure.

	We define now the notions of  \textit{weak} and  \textit{intermediate convergence} of sequences in $\BV(\Om)$ which provide different topologies on the space $\BV(\Om)$. The former is obtained by a subsequence of a bounded sequence in $\BV(\Om)$.  Moreover, the latter is sufficient  to preserve boundary conditions in the sense of the trace {as stated in \Cref{thm:traceBV}} below.
	\begin{definition}[\textsc{Weak convergence for } $\BV(\Om)$]\label{def:weakBV}
		Let $\{u_n\}$ be a sequence in $\BV(\Omega)$ and $u^*\in \BV(\Omega)$. We say that $u_n$ converges to $u^*$ weakly, denoted as $u_n \rightharpoonup  u^*$ in $\BV(\Omega)$, if
		\begin{equation*}
		u_n\to u^* \text{ in } L^1(\Omega), \qquad \text{ and } \qquad  |\mathrm{D}u_n|_p\rightharpoonup |\mathrm{D}u^*|_p \text{ in } \mathrm{M}(\Omega).
		\end{equation*}
	\end{definition}
	Recall that if $\{\mu_n\}$ is a sequence of measures in $\mathrm{M}(\Omega)$ then $\mu_n\rightharpoonup \mu$ in $\mathrm{M}(\Omega)$ for some  $\mu\in \mathrm{M}(\Omega)$, that is, $\mu_n$ weakly converges to $\mu$, if
	\begin{equation*}
	\int_{\Omega}g\dif \mu_n\to 	\int_{\Omega}g\dif \mu,
	\end{equation*}
	for all $g\in C_0(\Omega)$.  
	
	The definition \ref{def:weakBV} is  understood in light of the following fact: If $\{u_n\}$ is a bounded sequence in $\BV(\Omega)$, there exists $u^*\in \BV(\Omega)$ such that along a subsequence $u_n \rightharpoonup  u^*$ in $\BV(\Omega)$. The latter follows since the embedding $\BV(\Omega)\hookrightarrow L^1(\Omega)$ is compact (see Attouch et al. \cite[Theorem 10.1.4.]{attouch}) for Lipschitz domains, and since a bounded sequence of measures admits a weakly  convergent  subsequence.
	
We shall use the direct method of calculus of variations to establish existence of solutions to problems in $\BV(\Omega)$ and with Dirichlet homogeneous boundary conditions on $\Gamma_D$. The space of interest is $\BV_{\Gamma_D}(\Omega)$ defined as
\begin{equation*}
		\BV_{\Gamma_D}(\Omega):=\{v\in \BV(\Omega): \gamma_0(v)=0 \text{ on } \Gamma_D\} ,
	\end{equation*} 
where $\gamma_0$ is a trace operator; see \cite[section 10.2]{attouch}. Notice that we use the same notation for the trace operator in  Sobolev spaces $W^{1,p}(\Omega)$. There is a fundamental issue with the trace in $\BV(\Omega)$ and the application of the direct method as we show next with an standard example  adapted from \cite{attouch}.

	Consider a bounded sequence $\{u_n\}$ in $\BV_{\Gamma_D}(\Omega)$. 
	Then, we can extract a subsequence (not relabeled) of $\{u_n\}$ such that $u_n \rightharpoonup u^*$ in $\BV(\Omega)$ . The problem is that in general it is not possible to say that $u^*\in \BV_{\Gamma_D}(\Omega)$: Let $\Omega=(0,1)$ with $\Gamma_D=\{0\}$, and consider $\{v_n\}$ defined as 
	\begin{equation*}
	v_n(x) = \begin{cases} nx, &\mbox{if } 0<x<1/n, \\
	1, & \mbox{if } 1/n\leq x<1. \end{cases} 
	\end{equation*}
	Then, $v_n\in \BV_{\Gamma_D}(\Omega)$, and $v_n\rightharpoonup v^* \in \BV(\Omega)\setminus \BV_{\Gamma_D}(\Omega)$, with $v^*=1$. The underlying  reason is that the trace operator in $\BV(\Omega)$ is not continuous with respect to \emph{weak convergence}, but it is with respect to the \emph{intermediate convergence} subsequently defined.  We further notice  that $|\mathrm{D}v_n|(0,1)=1$ and $|\mathrm{D}v^*|(0,1)=0$, this discrepancy is central to the issue we are considering. 
	
	\begin{definition}[\textsc{Intermediate convergence}]
		Let $\{u_n\}$ be a sequence in $\BV(\Omega)$ and $u^*\in \BV(\Omega)$. We say that $u_n$ converges to $u^*$ in the sense of intermediate convergence if
		\begin{equation*}
		u_n\to u^* \text{ in } L^1(\Omega), \qquad \text{ and } \qquad  \int_{\Omega}|\mathrm{D}u_n|_p\to \int_{\Omega} |\mathrm{D}u^*|_p.
		\end{equation*}
	\end{definition}
	
	The name \emph{intermediate convergence} arises since it describes a stronger topology than the one of weak convergence, but not as strong as the norm one. The   importance of  the intermediate convergence can be seen in the following result which holds in our case since $\Omega\subset \mathbb{R}^\mathrm{d}$ is a Lipschitz bounded domain.  We refer to Attouch et al. \cite[Theorem 10.2.2]{attouch} for its proof. 	
	\begin{theorem}\label{thm:traceBV}
		The trace operator $\gamma_0:\BV(\Omega)\to L^1(\partial \Omega)$ is continuous when $\BV(\Omega)$ is equipped with the intermediate convergence and when $L^1(\partial \Omega)$ is equipped with the strong convergence.
	\end{theorem}
	We also note that $C^\infty(\overline{\Om})$ is dense in $\BV(\Om)$ in the intermediate convergence topology, for a proof see \cite[Theorem 10.1.2]{attouch}.

	%%%%%%%%%%%%%%%%%%%%%%%%%%%%%%%%%%%
	\subsection{The gradient constraint} \label{sec:Grad}
	%%%%%%%%%%%%%%%%%%%%%%%%%%%%%%%%%%%
	
	A few words are in order concerning the gradient constraint given in the set $K$ defined in \eqref{eq:SetK}. Although in the case when $G=\nabla$ the situation is somewhat standard, if $G=\mathrm{D}$, the distributional gradient for $\BV$ functions, require several  {non-trivial} explanations. In the cases where $\alpha$ is a Borel measure and $v\in \BV(\Om)$, the inequality 
	\begin{equation}\label{eq:IneqIni}
		|\mathrm{D} v|_{p}\leq\alpha
	\end{equation}
	 in \eqref{eq:SetK} 
	is understood in the sense of measures, i.e., \eqref{eq:IneqIni} holds true if 
	\begin{equation} \label{eq:measure_ineq}
		\int_\Om w |\mathrm{D} v|_{p} \leq \int_\Om w \da \,\text{ for all }w\in C^\infty_0(\Omega)\text{ with }w\geq0\text{ in }\Omega,
	\end{equation}
	and equivalently, for every Borel measurable set $S\subset\Omega$, it holds that
	\begin{equation}\label{eq:grad_constraint_measure_borel}
		\int_S |\mathrm{D}v|_{p} \leq \int_S \da.
	\end{equation}
Given that nonnegative Borel measures are inner and outer regular (\cite[Proposition 4.2.1]{attouch}) the condition \eqref{eq:measure_ineq} is equivalent to
\begin{equation}\label{eq:grad_constraint_measure_borel2}
			 \int_{O}|\mathrm{D}v|_{p} \leq\int_O \dif\alpha
\end{equation}
for all open sets $O\subset \Omega$.

It is possible to replace $C_0^\infty(\Omega)$ in \eqref{eq:measure_ineq} by 
	$C^\infty(\overline\Omega)$, which we discuss next.
	
	\begin{proposition}
		The condition in \eqref{eq:measure_ineq} is equivalent to
			\begin{equation}\label{eq:grad_constraint_measure_Cbd}
		\int_\Omega w |\mathrm{D} v|_{p}\leq \int_\Omega w \dif\alpha \quad
		\text{for every } w\in C^\infty(\overline\Omega)\text{ with } w\geq 0 \mbox{ in } \Omega.
	\end{equation}
	\end{proposition}
	
	\begin{proof} Suppose that \eqref{eq:measure_ineq} holds true and let $K_n$ be a sequence of closed sets such that
	\begin{equation}\label{eq:kn_condition}
	\int_{\Omega\setminus K_n} |\mathrm{D}v|_p\to 0\quad
		\text{and}\quad
	\int_{\Omega\setminus K_n} \da\to 0 . 
	\end{equation}
	{The sequence $\{K_n\}$ exists given that measures in $\mathrm{M}^+(\Omega)$ are inner regular; see \cite[Proposition 4.2.1]{attouch}.}
	Let $\tilde w\in C^\infty(\overline\Omega)$ be nonnegative and arbitrary.
	
	Accordingly, let  $\{w_n\}$  in $C_0^\infty(\Omega)$ be {nonnegative}, uniformly bounded  in $\Omega$, and such that  $w_n=\tilde w$ in $K_n$.
	Hence $|\tilde w|+|w_n|$ can be uniformly estimated by a constant, and by \eqref{eq:kn_condition} it holds that %by (uniform) boundedness of $\tilde w$ and $w_n$ and by \eqref{eq:kn_condition}
	\begin{equation*}
		\int_{\Omega} (\tilde w - w_n) |\mathrm{D}v|=\int_{\Omega\setminus K_n} (\tilde w - w_n) |\mathrm{D}v|\to 0\quad
		\text{and}\quad
		\int_{\Omega} (\tilde w - w_n) \da=\int_{\Omega\setminus K_n} (\tilde w - w_n) \da\to 0.
	\end{equation*}
	Since the inequality in\ \eqref{eq:grad_constraint_measure_Cbd}  holds for every $w_n$ {by initial assumption}, it also holds in the limit for $\tilde w$. Furthermore, \eqref{eq:grad_constraint_measure_Cbd} immediately implies \eqref{eq:measure_ineq}, so the result is proven.
	\end{proof}

	%%%%%%%%%%%%%%%%%%%%%%%%%%%%%%%%%%%%%%%%%%%%%%%%%%%%%%%%
	%\section{Existence Theory for \eqref{eq:original_problem_ini}}\label{Sec:existence}
		\section{Existence Theory for \texorpdfstring{$(\mathbb{P})$}{P}}\label{Sec:existence}
	%%%%%%%%%%%%%%%%%%%%%%%%%%%%%%%%%%%%%%%%%%%%%%%%%%%%%%%%

	%%the content of Existence_predual is now included below
	%%\input{Existence_predual}

	In this section, we discuss the existence and uniqueness of solution to the problem \cref{eq:original_problem_ini}. We start with the case when $\alpha$ is a measure, and the case when $\alpha$ is a function follows as a special one. In particular, existence of solutions is studied in  the function spaces  $U_{\Gamma_D}(\Omega)=\BV_{\Gamma_D}(\Omega)$ and $U_{\Gamma_D}(\Omega)=W^{1,1}_{\Gamma_D}(\Omega)$. Both of these spaces share the same difficulty: Bounded sequences do not necessarily admit convergent (in some sense) subsequences that preserve the zero boundary condition  on  $\Gamma_D$ in the limit. The main purpose of this section is to overcome this obstacle.
 	%%%%%%%%%%%%%%%%%%%%%%%%%%%%%%%%%%%%%%%%%%%%%%%%%%%%%%%%
	\subsection{The case when \texorpdfstring{$\alpha$}{alpha} is a nonnegative Borel measure} 
	%%%%%%%%%%%%%%%%%%%%%%%%%%%%%%%%%%%%%%%%%%%%%%%%%%%%%%%%
%		
	We consider in this section that $\alpha\in \mathrm{M}^+(\Omega)$ and hence the state space is given by 
	\begin{equation*}
		U_{\Gamma_D}(\Omega)=\BV_{\Gamma_D}(\Om).
	\end{equation*}
	We start  {by proving} the following lemma which gives sequential precompactness of some classes of bounded sets in $\BV_{\Gamma_D}(\Om)$. {These bounded sets} are subsets  of $K$ which in this case is defined as
	\begin{equation*}
	K= \{v\in \BV_{\Gamma_D}(\Omega) : |\mathrm{D} v|_{p}\leq\alpha\}.
\end{equation*}

	\begin{lemma}\label{lem:Pro}
		Let $\alpha\in \mathrm{M}^+(\Omega)$, then the set
		\begin{equation*}
			K^*=K\cap \{v\in L^1(\Omega): \|v\|_{L^1(\Omega)}\leq M\}
		\end{equation*} 
		is sequentially precompact in the sense of the intermediate convergence of $\BV(\Omega)$ for any $M>0$. 
	\end{lemma}

	\begin{proof}
		Let $\{v_n\}$ be a sequence in  $K^*$, then it is bounded in $\BV(\Omega)$, and thus $v_n \rightharpoonup v^*$ in $\BV(\Omega)$ for some $v^*\in \BV(\Omega)$ along a subsequence (not relabelled). Since $|\mathrm{D}v_n|_p\rightharpoonup |\mathrm{D}v^*|_p$  in  $\mathrm{M}(\Omega) $,  and $|\mathrm{D} v_n|_p\leq\alpha$ it follows that for every open set $O\subset \Omega$ that 
		\begin{equation}\label{eq:IneqO}
			|\mathrm{D} v^*|_p(O)\leq \liminf_{n\to\infty} |\mathrm{D} v_n|_p(O) \leq \alpha (O),
		\end{equation}
		where we have used the lower-semicontinuity property for open sets of weak convergence of measures; see \cite[Proposition 4.2.3]{attouch}. Additionally, since elements in $\mathrm{M}(\Omega)$ are outer (and inner) regular, we have that for a Borel set $B$ it holds that $\mu(B)=\inf \mu (O)$ where the infimum  is taken over all open sets such that $B\subset O$; see \cite[Proposition 4.2.1]{attouch}. Thus,
		\begin{equation}
		|\mathrm{D} v^*|_p(B)\leq \alpha (B)
		\end{equation}
	    follows from \eqref{eq:IneqO} by taking the infimum over $\{O \text{ open }: B\subset O\}$.  
			
		In order to prove that $v_n$ converges to $v^*$ in the sense of intermediate convergence, we are only left to prove that $|\mathrm{D}v_n|\rightharpoonup |\mathrm{D}v^*|$ narrowly in $\mathrm{M}^+(\Omega)$ (see  \cite[Proposition 10.1.2]{attouch}). The latter meaning that $\int_{\Omega}\varphi|\mathrm{D}v_n|\to \int_{\Omega} \varphi|\mathrm{D}v^*|$ for each continuous and bounded $\varphi$ on $\Omega$. Given that $\alpha\in \mathrm{M}^+(\Omega)$ we have that 
		 for each $\epsilon>0$ there exists a compact set $\Lambda_\epsilon\subset \Omega$ such that 
		\begin{equation*}
			\alpha(\Omega\setminus \Lambda_\epsilon)\leq \epsilon.
		\end{equation*}
		Since $v_n\in K$, then $|\mathrm{D} v_n|\leq \alpha$, and  hence for each $\epsilon>0$ the compact set $\Lambda_\epsilon\subset \Omega$, is such that
			\begin{equation*}
			|\mathrm{D} v_n|(\Omega\setminus \Lambda_\epsilon)\leq \epsilon, \qquad \text{ for all } n\in \mathbb{N}.
		\end{equation*}
		Then, by Prokhorov Theorem (see \cite[Theorem 4.2.3]{attouch}), there is a subsequence of $\{|\mathrm{D} v_n|\}$ (not relabelled) that $|\mathrm{D}v_n|\rightharpoonup |\mathrm{D}v^*|$ narrowly in $\mathrm{M}^+(\Omega)$. That is, along a subsequence, $v_n$ converges to $v^*$ in the sense of intermediate convergence. This implies that 
		\begin{equation*}
			v^*\in \BV_{\Gamma_D}(\Omega),
		\end{equation*}
		by virtue of Theorem \ref{thm:traceBV} and the fact that $v_n\in \BV_{\Gamma_D}(\Omega)$ for all $n\in \mathbb{N}$.	
	\end{proof}
	The above results particularly means that for a sequence $\{v_n\}$ in $K$ that is bounded in $\BV(\Omega)$, there exists a subsequence that converges to some $u^*\in \BV(\Omega)$ in the sense of intermediate convergence. Further, $u^*\in \BV_{\Gamma_D}(\Omega)$ and also $u^*\in K$.   A direct consequence of the above lemma is the following result.
	\begin{theorem}\label{thm:exis_original_measure}
		If $\alpha \in \mathrm{M}^+(\Omega)$, then there exists a unique solution  to \eqref{eq:original_problem_ini} in $\BV_{\Ga_D}(\Om)$.
			\end{theorem}
	\begin{proof}
		Consider an infimizing sequence $\{u_n\}$ for \eqref{eq:original_problem_ini}. It follows that $\{u_n\}$ is bounded in $L^2(\Omega)$ and {hence} Lemma \ref{lem:Pro} is applicable. That is, there is a subsequence of $\{u_n\}$ (not relabelled) such that $u_n\rightharpoonup u^*$ in $L^2(\Omega)$, and $u_n\to u^*$ in the sense of the intermediate convergence for $\BV(\Om)$, and further $u^*\in K$. Finally, by exploiting the weakly lower semicontinuity property of the objective functional in \eqref{eq:original_problem_ini}, we have that $u^*\in K$ is a minimizer.
	\end{proof}
	Next we discuss the case when $\alpha$ is a function.

	%%%%%%%%%%%%%%%%%%%%%%%%%%%%%%%%%%%%%%%%%%%%%%%%%%%%%%%%
	\subsection{The case when \texorpdfstring{$\alpha$}{alpha} is an integrable function} 
	%%%%%%%%%%%%%%%%%%%%%%%%%%%%%%%%%%%%%%%%%%%%%%%%%%%%%%%%
		 {In this section, we let $\alpha:\Omega\to\mathbb{R}$ be a nonnegative and integrable function, leading to
		 \begin{equation*}
		 	U_{\Gamma_D}(\Omega)=W^{1,1}_{\Gamma_D}(\Om).
		 \end{equation*} 
		%This particular case can be interpreted as partially contained in the previous section by assuming that $\alpha$ is a measure absolutely continuous with respect to the Lebesgue one. 
		This case can be interpreted (to some extent) as a special case of the one in the previous subsection under the assumption that $\alpha$ is a measure absolutely continuous with respect to the Lebesgue measure.} 
		However, we proceed in a slightly different fashion by considering $\alpha$ as a function and the state space contained in $W^{1,1}(\Omega)$; this provides further insight  on bounded sequences in $K$ and in Sobolev spaces.
		In this case, we have $K$  given by
	\begin{equation*}
	K= \{v\in W^{1,1}_{\Gamma_D}(\Omega) : |\nabla v|_{p}\leq\alpha \: \text{ a.e.}\}.
\end{equation*}
	{Next we state a version of \Cref{lem:Pro} adapted to the current setting which} can be used {to prove} existence of solutions  {to} \eqref{eq:original_problem_ini}.   
	
	\begin{lemma}\label{lem:Pro2}
		Let $\alpha\in L^1(\Omega)^+$ and $M>0$, then every sequence $\{v_n\}$  in the  set
		\begin{equation*}
			K^*=K\cap \{v\in L^1(\Omega): \|v\|_{L^1(\Omega)}\leq M\}
		\end{equation*} 
		admits a subsequence satisfying  
			\begin{equation*}
			v_n\to v^* \text{ in } L^1(\Omega), \qquad \text{ and } \qquad  \int_{\Omega}|\nabla v_n(x)|_p\dif x\to \int_{\Omega} |\nabla v^*(x)|_p\dif x,
		\end{equation*}
		for some $v^*\in K^*$, which is also the weak limit in $ W^{1,1}_{\Gamma_D}(\Omega)$ of the same subsequence.
	\end{lemma}
	The above can be seen as a consequence of equi-integrability of the set $K$. {Recall that a family of functions $\mathcal{F}\subset L^1(\Omega)$ is \emph{equi-integrable} provided that for every $\epsilon>0$, there exists a $\delta>0$ such that for every set $A\subset \Omega$ with $|A|<\delta$ we have that $\int_{A}|u|\dif x<\epsilon$ for all $u\in\mathcal{F}.$ Further, the Dunford-Pettis theorem states that if $\{u_n\}$ is a bounded sequence in $L^1(\Omega)$ and is equi-integrable, then $u_n\rightharpoonup u$ along a subsequence for some $u\in L^1(\Omega)$. Hence, since} $K$ is bounded in $W^{1,1}(\Omega)$, and the gradients are equi-integrable, it is simple to infer strong convergence in $L^1(\Omega)$ together with weak convergence of the gradients in $L^1(\Omega)$. {The improvement of the latter convergence} is done again via Prokhorov's result as in the proof of Lemma \ref{lem:Pro} leading to an equivalent of the intermediate convergence in $\BV(\Omega)$. The trace preservation follows directly from the same proof. Further note that the convergence determined does not imply strong convergence in $W^{1,1}(\Omega)$ since this  space is not uniformly convex. Another formulation of the above lemma is that bounded sets with equi-integrable gradients are compact in $W_{\Gamma_D}^{1,1}(\Omega)$ when endowed with the metric
	\begin{equation*}
		d(v,u):=\|u-v\|_{L^1(\Omega)}+\left|\int_{\Omega} |\nabla u(x)|_p\dif x - \int_{\Omega}|\nabla v(x)|_p\dif x\right|.
	\end{equation*} 
	With the use of \Cref{lem:Pro2} and following the same argument as before for \Cref{thm:exis_original_measure}, we have
	\begin{theorem}\label{thm:exis_original_fun}
		If $\alpha \in L^1(\Omega)^+$, then there exists a unique solution  to \eqref{eq:original_problem_ini} in $W^{1,1}_{\Ga_D}(\Om)$.
	\end{theorem}

	%%%%%%%%%%%%%%%%%%%%%%%%%%%%%%%%%%%%%%%%%%%%%%%%%%%%%%%%
	%\section{Existence Theory for \eqref{eq:predual_problem_ini}}\label{Sec:existence_dual}
		\section{Existence Theory for the pre-dual problem \texorpdfstring{$(\mathbb{P}^*)$}{P*}}\label{Sec:existence_dual}
	%%%%%%%%%%%%%%%%%%%%%%%%%%%%%%%%%%%%%%%%%%%%%%%%%%%%%%%%

The focus of this section is on existence and uniqueness of solutions of problem \eqref{eq:predual_problem_ini} {under different functional analytic} settings. In particular, we focus on {two cases where $\bm p$ is either  {(i)} a function or}  {(ii)} a Borel measure. In the first case, we {let} $\alpha$ be {either a} function or a measure; {here, existence results are limited to $\mathrm{d}=1$. On the other hand, in the second case we establish an existence and uniqueness result for $\bm p$ with arbitrary $\mathrm{d} \in \mathbb{N}$, for} a {specific class of $\alpha$'s (to be specified later)}. {Furthermore}, this second case requires a nonstandard space of vector measures with divergences in {$L^2(\Omega)$. Remarkably, a version of the integration-by-parts formula still holds in this general setting}; such a construct is rather recent \cite{vsilhavy2008divergence}. We start with the case when $\bm p$ is a function.

\subsection{The case when \texorpdfstring{${\bm p}$}{p} is a function and \texorpdfstring{$\alpha$}{alpha} is either a function or a measure}\label{sec:apfunc}
We begin this section by considering that $\alpha\in L^1(\Omega)^+$  and $J$ {is} defined as
\begin{equation}\label{eq:apfunctions}
J(\bm p)=\int_\Omega \alpha(x) |\bm p(x)|_q \dif x.	
\end{equation}
Moreover, we define  
	\begin{equation*}
	\|\bm{p}\|_{\alpha,2}:=\int_\Omega \alpha(x)|\bm{p}(x)|_{q} \dif x+ \|\div\: \bm{p}\|_{L^2(\Omega)},
	\end{equation*}
for $\bm p \in C^\infty(\overline{\Omega})^\mathrm{d}$. 

We assume that if $\mathrm{d}=1$ and $\Gamma_N=\emptyset$ then $\alpha$ is not identically zero, and if $\mathrm{d}>1$ then $\alpha>0$ a.e. in $\Omega$. Thus, the space $V_{\Gamma_{N}}(\Omega)$ is defined  by 
	\begin{equation}\label{eq:V_gamma_N_defn}
	V_{\Gamma_{N}}(\Omega) := \overline{E(\Omega)}^{\|\cdot\|_{\alpha,2}},
	\end{equation}
	where 
		\begin{equation*}
	E(\Omega) := \{ \bm p\in C^{\infty}(\overline{\Omega})^\mathrm{d} \::\:  \overline{\supp\, (\bm p)}\cap \Gamma_N =\emptyset \}.
	\end{equation*}
	It follows that	$V_{\Gamma_{N}}(\Omega)$ is a Banach space: If $\mathrm{d}>1$, the result is clear given that $\alpha>0$ a.e. in $\Omega$. If $\mathrm{d}=1$, then $V_{\Gamma_{N}}(\Omega)=H^1_{\Gamma_N}(\Omega)$ which follows  from the fact that $J(\bm p)+\frac{1}{2}\int_{\Omega}{|\bm p'(x)|^2}\dif x $ is an equivalent norm (to the usual one) on $H^1_{\Gamma_N}(\Omega)$. The latter is due to $J(\bm p)=\int_{\Omega} \alpha(x)|\bm p(x)|\dif x$ being a seminorm in $H^1_{\Gamma_N}(\Omega)$ and norm on the constants, i.e. for $a\in \mathbb{R}$, $J(a)={| a |}\alpha (\Omega)=0$ iff $a=0$; see \cite[Chapter 1.4]{temam1997infinite}.  We can now establish existence of a solution to problem \cref{eq:predual_problem_ini}.

	\begin{theorem}\label{thm:docexist}
		Let $\mathrm{d}=1$, $\alpha\in L^1(\Omega)^+$, and if $\Gamma_N=\emptyset$ then suppose that $\alpha$ is not identically zero. Consider $J$  {as defined  in} \eqref{eq:apfunctions} on $V_{\Gamma_{N}}(\Omega)$ as  {in} \eqref{eq:V_gamma_N_defn}. Then, there exists a unique solution to \eqref{eq:predual_problem_ini}. 
		%The solution is unique provided that $\alpha>0$ almost everywhere or  {$\Gamma_N\neq \emptyset$}. 
	\end{theorem} 
	\begin{proof}
		The proof is based on the direct method.
		{Let $\mathcal{J}\colon V_{\Gamma_{N}}(\Omega)\to \RR$ be} the objective function in \cref{eq:predual_problem_ini}, that is,
		\[ \mathcal{J}(\bm p) := \frac12\int_\Omega|\bm p'(x) - f(x)|^2 \dif x + J(\bm p),  \]
		 {and let} $\{\bm p_n\}_{n=1}^\infty$ in $V_{\Gamma_{N}}(\Omega)$ be an infimizing sequence of $\mathcal{J}$.
		Note that $\frac{1}{2}\int_{\Omega}{|\bm p'(x)|^2}\dif x+\int_{\Omega}\alpha |\bm p(x)|\dif x$ is a norm in $H^1_{\Gamma_N}(\Omega)$; see \cite[Chapter 1.4]{temam1997infinite}. Hence, $\{\bm p_n\}_{n=1}^\infty$ is bounded in $V_{\Gamma_{N}}(\Omega)$, and  there exists a weakly convergent (not relabeled) subsequence  $\{\bm p_{n}\}_{n=1}^\infty$ such that  
		$\bm p_{n} \rightharpoonup \bar {\bm p}$ in $H^1_{\Gamma_N}(\Omega)$. By the compact embedding of $H^1_{\Gamma_N}(\Omega)\hookrightarrow C(\overline{\Omega})$ (see \cite[Chapter 6]{adams_fournier_2003}) we have existence of a subsequence (not relabeled) $ \bm p_{n} \to \bar{\bm p}$  in $C(\overline{\Omega})$.
				Finally, weak lower semicontinuity of $\mathcal{J}(\bm p)$ yields that $\bar {\bm p}\in V_{\Gamma_{N}}(\Omega)$ is a solution to \cref{eq:predual_problem_ini}.	The strict convexity of the objective functional provides uniqueness to the solution.	
			\end{proof}	

An analogous approach can be considered when  $\alpha$ is a non negative Borel measure (and not identically zero), that is, when $\alpha\in \mathrm{M}^+(\Omega)$. In particular, we set
	\begin{equation}\label{eq:Jda}
		J({\bm p}) = \int_\Omega |\bm p|_q \dif \alpha,
	\end{equation}
and we construct the space $V_{\Gamma_{N}}(\Omega)$  in the same  {way as} in \cref{eq:V_gamma_N_defn}, but  {with the norm} {$\|\cdot\|_{\alpha,2}$ defined as}
	\begin{equation*}
	\|\bm{p}\|_{\alpha,2}:=\int_\Omega |\bm{p}|_{q} \dif \alpha+ \|\div\: \bm{p}\|_{L^2(\Omega)},
	\end{equation*}
	and assuming that if $\mathrm{d}=1$ and $\Gamma_N=\emptyset$ then $\alpha$ is not identically zero, and if $\mathrm{d}>1$ then $\alpha(B)>0$ if $|B|>0$ and $B\subset\Omega$ is a Borel set.

The existence result of Theorem \ref{thm:docexist} follows mutatis mutandis: Since $\frac{1}{2}\int_{\Omega}{|\bm p'(x)|^2}\dif x+\int_{\Omega}|{\bm p}| \dif \alpha$ is again a norm in $H^1_{\Gamma_N}(\Omega)$, see \cite[Chapter 1.4]{temam1997infinite}, the exact argument is applicable in this case.

We can now focus on the case when $\bm p$ is a measure which provides a general setting for the problem of interest in terms of existence, uniqueness, and duality results.

		\subsection{The case when \texorpdfstring{$\alpha$}{alpha} is a function and \texorpdfstring{${\bm p}$}{p} is a measure} We focus now on problem \eqref{eq:predual_problem_ini} when $J$ is defined as
		\begin{equation}\label{eq:Jdq}
			J(\bm p)=\int_\Omega \alpha \dif|\bm p|_q,
		\end{equation}
		and $\bm p$ is a Borel measure.  Notice that the above functional can be seen as a generalization of the functional in \eqref{eq:apfunctions}.  {The latter  can be obtained by letting $\bm p$ be} absolutely continuous with respect to the Lebesgue measure.

		 The functional analytic setting in this section, requires $\bm p$ to be a measure with  {divergence of $\bm p$ in $L^2(\Omega)$,} and $\alpha$ to be measurable with respect to $|\bm p|_q$. We start with a proper definition of such spaces and their properties. We disregard the possible ``boundary conditions'' for the variable ${\bm p}$, so that $\Gamma_N=\emptyset$ and we define $V_{\Gamma_{N}}(\Omega)$ as follows:
\begin{equation}\label{eq:Vacont}
	V_{\Gamma_{N}}(\Omega):=W(\Omega)=\{\bm p\in \mathrm{M}(\Omega)^\mathrm{d}: \div\:\bm p \in L^2(\Omega)\},	
\end{equation}  
where $\mathrm{M}(\Omega)^\mathrm{d}$ corresponds to the $\mathbb{R}^\mathrm{d}$-valued Borel measures in $\Omega\subset \mathbb{R}^\mathrm{d}$. Specifically, $\bm p\in W(\Omega)$ if there exists $h\in L^2(\Omega)$ such that
 \begin{equation}\label{eq:IntPart}
 	\int_{\Omega} \nabla \varphi \cdot \mathrm{d}\bm p=-\int_\Omega\varphi h \dx, \qquad  \forall \varphi\in C_c^\infty(\Omega),
 \end{equation}
and we define $\div\:\bm p:= h$. The space $W(\Omega)$ is a Banach space when endowed with the norm 
\begin{equation}\label{eq:normW}
	\|\bm w\|_{W(\Omega)}:= |\bm w|_q(\Omega)    +\|\div\: \bm w\|_{L^2(\Omega)},     
\end{equation} 
where $q\in [1,+\infty]$ and
\begin{equation*}
	|\bm w|_q(\Omega):=\sup \left\{ \langle \bm w, \bm v\rangle : \bm v \in C_c(\Omega)^\mathrm{d} \: \text{ with } \: |\bm v(x)|_{p}\leq 1  \quad \forall x\in \Omega\right\}.
\end{equation*}
Note that above $\langle \cdot, \cdot \rangle$ is the duality pairing between $\mathrm{M}(\Omega)^\mathrm{d}$ and $C_c(\Omega)^\mathrm{d}$, and hence
\begin{equation*}
	\langle \bm w, \bm v\rangle=\int_\Omega {\bm v} \cdot \dif\bm w = {\sum _{i=1}^d} \int_\Omega {v_i}  \dif w_i.
\end{equation*}
{Similarly to the definition of $|\bm w|_q(\Omega)$, we can define $|\bm w|_q(A)$ for any open set $A$, and subsequently for an arbitrary} Borel set $A$. {Hence}, $|\bm w|_q$ induces a nonnegative measure (the \emph{total variation measure of }$\bm w$); in addition $|\bm w|_q(\Omega)=\int_\Omega \dif|\bm w|_q$. Note that the space $W(\Omega)$ contains regular maps, clearly if $\bm p \in C_c^1(\Omega)^\mathrm{d}$ then $\bm p \in W(\Omega)$, in this case ``$\dif |\bm p|_q= |{\bm p}|_q \dx$'' where $\dx$ is the Lebesgue measure.

A note on the space $W(\Omega)$  {is in order}. Although one may be inclined to think that vector fields whose divergences are in $L^2(\Omega)$ would always have better regularity than just the measure type, this is not true. We consider an example developed by \v{S}ilhav\'{y} \cite{vsilhavy2008divergence} to show otherwise. Let $u\in \BV(\Omega)$ with $\Omega\subset \mathbb{R}^2$, and define $\bm p = (D u)^{\perp}$ with $(a_1,a_2)^\perp=(a_2,-a_1)$ with $Du$ the distributional (measure valued) gradient of $u$; it follows that $\div\:\bm p=0$. This can be seen as  {follows:}   $C^\infty(\overline{\Omega})$ is dense (in the sense of the intermediate convergence) in $\BV(\Omega)$, this means in particular that $\lim \int_{\Omega} \nabla \varphi \cdot \bm p_n \dx=\int_{\Omega} \nabla \varphi \cdot \mathrm{d}\bm p$ for such a smooth sequence defined as $\bm p_n=(D u_n)^{\perp}$ with $u_n\in C^\infty(\overline{\Omega})$. Since also $\int_{\Omega} \nabla \varphi \cdot \bm p_n \dx=0$, the result follows by taking the limit and from \eqref{eq:IntPart}.

Following \v{S}ilhav\'{y}~\cite{vsilhavy2008divergence}, we have a form of  {integration-by-parts formula} together with a trace result. We denote by $\mathrm{Lip}^{B}(\Lambda)$ the space of Lipschitz maps $h:\Lambda\to\mathbb{R}$ for $\Lambda\subset \mathbb{R}^k$ and endow it with the norm
	\begin{equation*}
		\|h\|_{\mathrm{Lip}^{B}(\Lambda)}:= \mathrm{Lip}(h)+\sup_{x\in \Lambda}|h(x)|,
	\end{equation*}
where $\mathrm{Lip}(h)$ is the Lipschitz constant of $h$ on $\Lambda$. It follows that  for each $\bm p\in W$ there exists a linear functional $\mathrm{N}_{\bm p}:\mathrm{Lip}^{B}(\partial\Omega)\to \mathbb{R}$ such that for all $v\in \mathrm{Lip}^{B}(\overline{\Omega})$ we have
\begin{equation}\label{eq:IntParts}
	\mathrm{N}_{\bm p}(v|_{\partial \Omega})=\int_{\Omega} \nabla v \cdot \dif\bm p+\int_\Omega v\: \div\:\bm p \dx.
\end{equation}
Further, $\mathrm{N}_{\bm p}$ is bounded in the following sense
\begin{equation*}
|	\mathrm{N}_{\bm p}(g)|\leq \left(|\bm p|_q(\Omega) +|\div\: \bm p|(\Omega)\right) \|g\|_{\mathrm{Lip}^{B}(\partial\Omega)}\leq C \|\bm p\|_V \|g\|_{\mathrm{Lip}^{B}(\partial\Omega)},
\end{equation*}
for some $C>0$, and all $\bm p \in W$ and all $g\in \mathrm{Lip}^{B}(\partial\Omega)$. Provided that $\bm p$ and $v$ have enough differential regularity, we observe 
\begin{equation*}
	\mathrm{N}_{\bm p}(v|_{\partial \Omega})=\int_{{ \partial\Omega}} v\; \bm p \cdot \nu  \:\text{d}\,\mathcal{H}^{ {\mathrm{d}-1}}
\end{equation*}
as expected.  {Thus}, \eqref{eq:IntParts} is an extension of the usual  {integration-by-parts} formula.

We are now ready to state and prove the existence and uniqueness result for problem \eqref{eq:predual_problem_ini} under the setting above.

\begin{theorem}\label{thm:unique_existence_predual_p_measure}
Let $\alpha\in C(\overline{\Omega})$ be such that $\alpha(x)>0$ {for all $x \in \overline\Omega$}, and consider $J$ defined  by \eqref{eq:Jdq} on $V_{\Gamma_{N}}(\Omega)=W(\Omega)$ as  {given in} \eqref{eq:Vacont}.
	Then, problem \eqref{eq:predual_problem_ini} admits a unique solution.
\end{theorem}
\begin{proof} Note first  that $J$ is well-defined given that $\alpha$ is measurable with respect to all Borel measures.
	Consider an infimizing sequence $\{\bm p_n\}$. Since $\min_{x\in\overline{\Omega}}\alpha(x)>0$, then  $\{\bm p_n\}$ is bounded in {$V_{\Gamma_N}(\Omega)$}. Hence, we can extract a subsequence (not relabelled) such that $\bm p _n \rightharpoonup \bm p^*$ in $\mathrm{M}(\Omega)^\mathrm{d}$ for some $\bm p^*\in \mathrm{M}(\Omega)^\mathrm{d}$ and $\div\:\bm p_n \rightharpoonup h$ in $L^2(\Omega)$ for some $h\in L^2(\Omega)$.  {Furthermore}, for $\varphi\in C_c^\infty(\Omega)$ arbitrary
	\begin{equation*}
		( \varphi , \div\: \bm p^*)_{{ L^2(\Omega)}}=- \int_{\Omega} \nabla \varphi \cdot \mathrm{d}\bm p^*=-\lim_{n\to\infty} \int_{\Omega} \nabla \varphi \cdot \mathrm{d}\bm p_n=\lim_{n\to\infty}( \varphi , \div\: \bm p_n)_{{ L^2(\Omega)}}=( \varphi , h)_{{ L^2(\Omega)}},
	\end{equation*}
	so that $h=\div\: \bm p^*$, i.e., $\bm p^*\in W(\Omega)$.

	Since the map $\bm p \mapsto |\bm p|_q$ is weakly lower semicontinuous, $\alpha  \bm p_n\rightharpoonup \alpha \bm p^*$ in $\mathrm{M}(\Omega)^\mathrm{d}$, and $ |\bm q|_q= \alpha|\bm p|_q$ for $\bm q= \alpha \bm p$, we have that $\bm p^*$  {is a minimizer} by a weakly lower semicontinuity argument.  
	Uniqueness follows from the strict convexity of the  {objective} functional.
\end{proof}

 At this point, one would be tempted to extend the result to the case where $\Gamma_N\neq \emptyset$, for example, by defining 
\begin{equation}\label{eq:Vadisc}
	V_{\Gamma_N}(\Omega)=W(\Omega)\cap \{\bm p \in W: \mathrm{N}_{\bm p}(v|_{\partial \Omega})=0 \quad \forall v\in \mathrm{Lip}^{B}_{\Gamma_D}(\overline{\Omega})\}.
\end{equation}
While the space above is well-defined, it is not clear if the weak limits of sequences in the space also belong to it. In fact, if $\bm p_n\in V_{\Gamma_N}(\Omega)$ is bounded, then 
\begin{equation*}
		\int_{\Omega} \nabla v \cdot \dif\bm p_n=-\int_\Omega v\; \div\:\bm p_n \dx,
\end{equation*}
for each $v\in \mathrm{Lip}^{B}_{\Gamma_D}(\overline{\Omega})$. However, the weak limit along a subsequence argument is not enough to pass to the limit in the left hand side given that $\nabla v$ is not necessarily of compact support. This remains an open problem.

	%\newpage
	%%%%%%%%%%%%%%%%%%%%%%%%%%%%%%%%%%%%%%%%%%%%%%%%%%%%%%%%
%	\section{Duality relation between \cref{eq:original_problem_ini} and \cref{eq:predual_problem_ini}% Stationary problem with $\alpha$ a function
%	} \label{Sec:2final}
		\section{Duality relation between \texorpdfstring{$(\mathbb{P})$}{P} and \texorpdfstring{$(\mathbb{P}^*)$}{P*}% Stationary problem with $\alpha$ a function
	} \label{Sec:2final}
	%%%%%%%%%%%%%%%%%%%%%%%%%%%%%%%%%%%%%%%%%%%%%%%%%%%%%%%%

%the content of sec2final is now included here
%	\input{sec2final}

In this section, we discuss the dual problem corresponding to \cref{eq:predual_problem_ini}. 
We start with the case when $\alpha$ is a Lebesgue measurable function and further subdivide it into two subsections. In \Cref{Ssec:pre_dual_function} we discuss the case when the pre-dual variable ${\bm p}$ is a function and in the following \Cref{Ssec:p_measure}  {we assume that the variable ${\bm p}$ is a} measure.  {Next in \Cref{Sec:3final}, we consider the case where $\alpha$ is a measure and the pre-dual variable ${\bm p}$ is a function.} In general, we prove that 
\begin{equation*}
	\text{Problem \eqref{eq:original_problem_ini} is the Fenchel dual of Problem \eqref{eq:predual_problem_ini}.}
\end{equation*}

In order to keep the discussion self-contained, we introduce the following notation and terminology. For an extended real valued function $ {\psi}:X\to\RRR$ over a Banach space $X$, by  {$\psi^\ast$} we denote its convex conjugate, which is defined by (e.g. see \cite[p.~16]{ekeland_temam})
	\begin{equation}\label{eq:Fcj}
		 {\psi}^\ast: X^\ast\to\RRR,\qquad  {\psi}^\ast(x^\ast) = \sup_{x\in X}\left\{ \langle x^\ast, x \rangle _{X^\ast,X} -  {\psi}(x) \right\}.
	\end{equation}
	Provided that the operator $\div: V\to L^2(\Omega)$ is defined for a Banach space $V$, and it is bounded, its adjoint $(\div)^\ast: L^2(\Omega)\to V^\ast$
	is  {well-defined} and is given by $\langle (\div)^\ast v ,\bm p \rangle_{V^\ast,V}=(v, \div\: \bm p) $ for all $v\in L^2(\Omega)$ and all ${\bm p} \in V$.

%%%%%%%%%%%%%%%%%%%%%%%%%%%%%%%%%%%%%%%%%%%%%%%%%%%%%%%%
\subsection{The case when \texorpdfstring{$\alpha$}{alpha} is a function %Fenchel dual of \eqref{eq:predual_problem_ini}
} \label{Ssec:pre_dual_function}
%%%%%%%%%%%%%%%%%%%%%%%%%%%%%%%%%%%%%%%%%%%%%%%%%%%%%%%%
We first consider the case where $\alpha$ is a non negative Lebesgue measurable function and we accordingly set
\begin{equation*}%\label{eq:J_alpha_function}
	J(\bm p)=\int_\Omega \alpha(x) |\bm p(x)|_q \dif x \qquad \text{or} \qquad J(\bm p)=\int_\Omega \alpha\dif|{\bm p}|_{q},
\end{equation*} 
in \cref{eq:predual_problem_ini} for the cases when $\bm p$ is a function or a measure, respectively.  For each of the choices of $J$ above, we will also establish the strong duality to \cref{eq:original_problem_ini}. We assume throughout this section (and for the sake of simplicity) that 
\begin{equation*}
\alpha\in C(\overline{\Omega}),	 \quad \text{and} \quad \alpha(x)>0,
\end{equation*}
for all $x\in \overline{\Omega}$
as discussed in \Cref{Sec:1}, together with 
\begin{equation*}
	U_{\Ga_D}(\Om)=W^{1,1}_{\Ga_D}(\Om), \qquad \text{and} \qquad G=\nabla \,,
\end{equation*}
and hence,
	\begin{equation*}
	K= \{v\in W^{1,1}_{\Gamma_D}(\Omega) : |\nabla v|_{p}\leq\alpha \: \text{ a.e. in }\Om\}.
\end{equation*}
Note that in \Cref{Sec:existence} we proved the existence and uniqueness of solution to \Cref{eq:original_problem_ini}.

	We compute the dual problem to \cref{eq:predual_problem_ini} and show that it is given by problem \eqref{eq:original_problem_ini}.
	Defining $F: L^2(\Omega)\to\RR$ by
	\begin{equation}\label{eq:FF}
	F(v) := \frac12 \int_\Omega |v(x) - f(x)|^2\dx,
	\end{equation}
	 {the} problem \cref{eq:predual_problem_ini} can be written %in this decoupled way 
	as
	\begin{equation}\label{eq:primal_problem_decomposed}%\label{eq:mod_primal}
		\inf_{\bm p \in V_{\Gamma_{N}}(\Omega)}\ J(\bm p) + F(\div\:\bm p),
	\end{equation}
	for $\div: V_{\Gamma_{N}}(\Omega) \to L^2(\Om)$, where the space $V_{\Gamma_{N}}(\Omega)$ is chosen based on whether $\bm p$ is a function or a measure.

	By \cite[p.~61]{ekeland_temam},
	the Fenchel dual of \eqref{eq:predual_problem_ini}
	{
		with respect to the perturbation function
		\[
			\phi: V_{\Gamma_N}(\Omega) \times L^2(\Omega) \to \mathbb{R}\cup\{\infty\},\qquad
			\phi(\bm p, u) = J(\bm p) + F(\div\:\bm p - u)
		\]
	}
	is given by
	\begin{equation}\label{eq:dual_problem_decomposed}
		\inf_{u \in L^2(\Omega)} %\mathcal{J}^\ast(u) = 
		J^\ast(\div^\ast\, u) + F^\ast(-u),
	\end{equation}
	where the convex conjugates $J^\ast : (V_{\Gamma_{N}}(\Omega))^\ast \rightarrow \RRR$, $F^\ast : L^2(\Om) \rightarrow \RRR$ of $J$ and $F$ are  defined according to \eqref{eq:Fcj}, 
	see also \cite[p.~17]{ekeland_temam} for more details.

	\subsubsection{Duality when ${\bm p}$ is a function}\label{Ssec:p_function}
	
	Now we show that the problem~\eqref{eq:original_problem_ini} is the dual to problem~\eqref{eq:predual_problem_ini}. In this section, we assume that $ V_{\Gamma_{N}}(\Omega)$ is given by \eqref{eq:V_gamma_N_defn}, and that
	\begin{equation*}
		J(\bm p)=\int_\Omega \alpha(x) |\bm p(x)|_q \dif x.
	\end{equation*}
	 We start by proving {the following result:}

\begin{theorem}\label{lemma:polar_F}
	For every $u\in L^2(\Omega)$, it holds that
	$J^\ast(\div^\ast u) = I_K(u)$.
\end{theorem}
We break the proof of the above theorem  {into} \Cref{lemma:polar_F_a,lemma:polar_F_step_4}, which we state after the following observation.
\begin{remark} \label{rem:Jstar_prereq}
	Observe that $J^\ast(\div^\ast u)$ only takes the value $0$  {or} $+\infty$:  By the definition of the convex conjugate $J^*$, for any $u\in L^2(\Omega)$ it holds that
	\begin{equation}
		J^\ast(\div^\ast u)\geq(u, \div\:\bm 0)_{ } -\int_\Omega\alpha(x)|\bm 0|_{q}\,\dx = 0.
	\end{equation}
	If $J^\ast(\div^\ast u) > 0$, i.e. there exists a $\bm p \in V_{\Gamma_{N}}(\Omega)$ such that
	\begin{equation*}
	\langle \div^\ast u, \bm p \rangle_{V_{\Gamma_{N}}(\Omega)^\ast, V_{\Gamma_{N}}(\Omega)} - \int_\Omega \alpha(x)|\bm p(x)|_q \dif x> 0,
	\end{equation*}
	{we can scale $\bm{p}$ by an arbitrarily large $\lambda \in \RR^+$ leading to $J^\ast(\div^\ast u) = +\infty$.}

\end{remark}

\begin{lemma}\label{lemma:polar_F_a}
	Let $u\in L^2(\Omega)$ with $J^\ast(\div^\ast u) =0$. Then the following hold true:
	\begin{enumerate}[\upshape(i)] 
		\item $u\in \BV(\Omega)$; \label{lemma:polar_F_a1}
		\item $|\mathrm{D} u|_p\leq \alpha$; \label{lemma:polar_F_a2}
		\item $\mathrm{D}u = \nabla u$ and $u \in W^{1,1}(\Omega)$; \label{lemma:polar_F_a2regular}
		\item $\gamma_0(u) = 0$ on $\Gamma_D$ \label{lemma:polar_F_a3}
	\end{enumerate}
	and therefore $u\in K$.
\end{lemma}
\begin{proof}
	\begin{enumerate}[(i)]
	\item
	First we show that $J^\ast(\div^\ast u)=0$ implies $u\in \BV(\Om)$.

	Suppose $u\notin \BV(\Om)$. Then, since $C_0^1(\Om)^\mathrm{d} \subset \Va$, we have that
	\begin{equation}\label{eq:Ineq1}
		\begin{aligned}
			J^\ast(\div^\ast\, u) &= \sup_{\bm p\in \Va}\left\{\langle \div^\ast\, u, \bm p \rangle_{\Va^\ast,\Va} - \int_\Omega \alpha(x)|\bm p(x)|_{q}\dif x\right\}\\
														&\ge \sup_{\substack{\bm p\in C_0^1(\Om)^{ {\textrm{d}}}\\ |\bm{p}|_{q} \le 1}}\left\{(u, \div\: \bm p)_{ } - \int_\Omega \alpha(x)|\bm p(x)|_{q}\dif x\right\}\\
														&\ge \sup_{\substack{\bm p\in C_0^1(\Om)^{ {\textrm{d}}}\\ |\bm{p}|_{q} \le 1}}\left\{(u, \div\: \bm p)_{ }\right\} - \int_\Omega \alpha(x)  \dif x
		\end{aligned}
	\end{equation}
	Then, by using definition of a function of bounded variation, see \cite[Definition 10.1.1]{attouch}, we have that the supremum on the right hand side of the above inequality is $+\infty$ if $u\notin \BV(\Om)$ and hence, $u\in \BV(\Om)$ if $J^\ast(\div^\ast\,u)<+\infty$.

	\item 
			As $u \in \BV(\Om)$, we have that $\mathrm{D}u \in \mathrm{M}(\Omega)^\mathrm{d}$
			and the inequality $|\mathrm{D} u|_{p}\leq\alpha$ is understood in the sense of
			\cref{eq:grad_constraint_measure_borel2}.
			Hence, if  
			\begin{equation}\label{eq:violation_meas}
				\int_O |\mathrm{D}u|_{p}- \int_O \alpha(x) \dif x \leq 0,
			\end{equation}
			for an arbitrary open set $O\subset \Omega$, then  {the required condition $|\mathrm{D} u|_{p}\leq\alpha$ immediately follows}.

			By the assumption $J^\ast(\div^\ast {u}) = 0$, and using integration by parts, we observe that
			\begin{equation*}
			\begin{aligned}
			0=J^\ast\left(\div^\ast u\right) &= \sup_{\bm p \in \Va} \langle \div^\ast u, \bm p\rangle - \int_\Omega  \alpha(x)|\bm p(x)|_q \dif x\\
			&\geq \sup_{\bm p \in  C^1_0({\Om})^{ {\textrm{d}}}} \left\{ \int_\Omega \bm{p} \mathrm{D} u - \int_\Omega\alpha(x)|\bm p(x)|_q \dif x \right\}\\
			&\geq \sup_{\substack{\bm p \in  C^1_0({O})^{{\textrm{d}}}\\ {|\bm p|_q}\le 1}}  \left\{\int_O \bm{p} \mathrm{D} u - \int_O\alpha(x)|\bm p(x)|_q \dif x \right\} \\
			&\geq \sup_{\substack{\bm p \in  C^1_0({O})^{{\textrm{d}}}\\ {|\bm p|_q}\le 1}}  \left\{\int_O \bm{p} \mathrm{D} u \right\} - \int_O\alpha(x) \dif x  \\
			&= \int_O  |\mathrm{D} u|_p - \int_O \alpha(x) \dif x,\\
			\end{aligned}
			\end{equation*}
			where the last inequality follows using the definition of $\int_O |\mathrm{D}u|_p$ and \cref{eq:violation_meas}.

		\item 
			By (\ref{lemma:polar_F_a1}) {and} (\ref{lemma:polar_F_a2}), it holds that
	\begin{equation}
		\int_S |\mathrm{D} u|_p\leq \int_S \alpha(x)\dif x,
	\end{equation}
	for every Borel set $S$  {(see \eqref{eq:grad_constraint_measure_borel})}, and especially for every Borel set of Lebesgue measure zero,
	%where thereby $\int_S Du = 0$
	it follows that $|\mathrm{D}u|_p$ vanishes on every set of measure zero, and
	hence $\mathrm{D}u$ is absolutely continuous w.r.t. the  {$\mathrm{d}$-dimensional} Lebesgue measure,
	and therefore $\mathrm{D}u=\nabla u $, i.e., the distributional gradient is a weak gradient.
	Thus, $u\in W^{1,1}(\Omega)$.

	\item To obtain the boundary conditions on $u$, we will show that if $J^\ast(\div^\ast u)=0$, then $\gamma_0(u) = 0$.

		Since $u\in \BV(\Om)$, then using \cite[Theorem~10.2.2]{attouch} we have that $\ga_0(u)\in L^1(\partial \Om)$ and
		\begin{align*}
			0=J^\ast(\div^\ast u) &= \sup_{\bm p\in \Va}\left\{\langle \div^\ast u, \bm p \rangle_{V^\ast, V} - \int_\Omega \alpha(x)|\bm p(x)|_{q} \dif x\right\} \nonumber\\
													&= \sup_{\bm p\in \Va\cap C^1(\overline{\Om})}\left\{(  u,\div\: \bm p )_{ { }} - \int_\Omega \alpha(x)|\bm p(x)|_{q}\dif x\right\}\nonumber\\
		&= \sup_{\bm p\in \Va\cap C^1(\overline{\Om})}\left\{-\int_\Omega \bm p(x) \cdot  \nabla u(x)\dif x + \int_{\Ga_D} \ga_0(u) \bm p\cdot \nu \, \textup{d}\,\mathcal{H}^{ {\mathrm{d}-1}}- \int_\Omega \alpha(x)|\bm p(x)|_{q}\dif x\right\}.
		%TODO: overfull box
		\end{align*}
{Whence for all  $\bm p \in \Va\cap C^1(\overline{\Om})$, we have} 
		\begin{align*}
			-\int_\Omega \bm p(x)  \cdot \nabla u(x)\dif x + \int_{\Ga_D} \ga_0(u) \bm p\cdot \nu \, \text{d}\,\mathcal{H}^{ {\mathrm{d}-1}}- \int_\Omega \alpha(x)|\bm p(x)|_{q}\dif x \le 0 .
		\end{align*}
		
 {Subsequently for all $\bm p \in \Va\cap C^1(\overline{\Om})$, we arrive at}
		\begin{align}\label{eq:eq1}
			\left| \int_{\Ga_D} \ga_0(u) \bm p\cdot \nu \, \text{d}\,\mathcal{H}^{{\mathrm{d}-1}} \right| \le \int_\Omega |\bm p(x)\;\nabla u(x)|\dif x + \int_\Omega \alpha(x)|\bm p(x)|_{q}\dif x .
		\end{align}
		To get \eqref{eq:eq1}, for a $\bm{p} \in \Va\cap C^1(\overline{\Om})$, choose $s\in \{-1,1\}$ such that $$\int_{\Ga_D} \ga_0(u) s\bm p\cdot \nu \, \text{d}\,\mathcal{H}^{ {\mathrm{d}-1}}\ge 0,$$ then for $\bm w = s\bm p \in \Va\cap C^1(\overline{\Om})$, we obtain that
		\begin{align*}
			\left| \int_{\Ga_D} \ga_0(u) \bm p\cdot \nu \, \text{d}\,\mathcal{H}^{ {\mathrm{d}-1}} \right| &= \int_{\Ga_D} \ga_0(u) s\bm p\cdot \nu \, \text{d}\,\mathcal{H}^{{\mathrm{d}-1}} = \int_{\Ga_D} \ga_0(u) \bm w\cdot \nu \, \text{d}\,\mathcal{H}^{{\mathrm{d}-1}}\\
			& \le \int_\Omega \bm w(x)   \,\nabla u(x)\dif x + \int_\Omega \alpha(x)|\bm w(x)|_{q}\dif x \\
			&\le  \left| \int_\Omega s\bm p(x)\;\nabla u(x)\dif x + \int_\Omega \alpha(x)|s \bm p(x)|_{q}\dif x \right|\\
		& \le  \int_\Omega |\bm p(x)\,\nabla u(x)|\dif x + \int_\Omega \alpha(x)|\bm p(x)|_{q}\dif x.
		\end{align*}

		Now for $\varepsilon>0$, by inner regularity \cite[pp. 95, proposition 15.1]{dibenedetto}, there exist closed subsets $\Ga_D^\varepsilon \subset \Ga_{D}$ and $ \Om^\varepsilon \subset \Om$  such that
		\begin{align*}
		&\left| \int_{\Om} |\nabla u(x)|_{p}\dif x- \int_{\Om^\varepsilon} |\nabla u(x)|_{p} \dif x\right|<\varepsilon, \quad 
		\left| \int_{\Om} \alpha(x)\dif x- \int_{\Om^\varepsilon} \alpha(x)\dif x \right|<\varepsilon
		\quad \text{and} \quad\\
		&\left\vert  \int_{\Ga_D} \left|\ga_0(u)\right| \, \text{d}\,\mathcal{H}^{ {\mathrm{d}-1}} - \int_{\Ga_D^\varepsilon} \left|\ga_0(u)\right| \, \text{d}\,\mathcal{H}^{ {\mathrm{d}-1}}   \right\vert < \varepsilon.
		\end{align*} 
		{Then, by Urysohn's} lemma there exists $\phi_\varepsilon\in C^\infty(\overline{\Om})$ satisfying, $0\le \phi_\varepsilon \le 1,$  {such that}
		\[  \phi_\varepsilon =1 \; \text{ on }\; \Ga_D^\varepsilon \quad \text{and} \quad 
		     \phi_\varepsilon =0 \; \text{ on }\; \Om^\varepsilon \cup  {\overline{\Ga}_N} . \]
		Then for any $\bm{q}\in C^1(\overline{\Om})$, applying  \eqref{eq:eq1} to $\bm{p}=\bm{p}_\varepsilon:= \phi_\varepsilon\, \bm{q} \in \Va\cap C^1(\overline{\Om})$, we obtain that
		 \begin{align*}
		 	\left| \int_{\Ga_D} \ga_0(u) \bm{p}_\varepsilon \cdot \nu \, \text{d}\,\mathcal{H}^{ {\mathrm{d}-1}} \right|
			&\le \int_\Omega |\bm{p}_\varepsilon(x)\cdot\nabla u(x)|\dif x + \int_\Omega \alpha(x)|\bm{p}_\varepsilon(x)|_{q} \dif x\\
		 	&\le \int_{\Om \setminus \Om^\varepsilon} |\bm{p}_\varepsilon(x)\cdot\nabla u(x)|\dif x + \int_{\Om \setminus \Om^\varepsilon} \alpha(x)|\bm{p}_\varepsilon(x)|_{q}\dif x. 
		 \end{align*}
		 {Further, from
		 \begin{equation*}
		 	\int_{\Gamma_D} \gamma_0(u) {\bm q} \cdot \nu \, \text{d}\,\mathcal{H}^{\mathrm{d}-1}
		 = \int_{\Gamma_D \setminus \Gamma_D^\varepsilon} \gamma_0(u) ({\bm q}-{\bm p}_\varepsilon) \cdot \nu \, \text{d}\,\mathcal{H}^{\mathrm{d}-1} + \int_{\Gamma_D} \gamma_0(u) {\bm p}_\varepsilon \cdot \nu \, \text{d}\,\mathcal{H}^{\mathrm{d}-1}
		 \end{equation*}
we infer that}
		 \begin{align*}
		 	\left|\int_{\Ga_D} \ga_0(u) \bm q\cdot \nu \, \text{d}\,\mathcal{H}^{ {\mathrm{d}-1}}\right|- \left|\int_{\Ga \setminus \Ga_D^\varepsilon} \ga_0(u)(\bm{p}_\varepsilon-\bm q)\cdot \nu \text{d}\,\mathcal{H}^{ {\mathrm{d}-1}}\right| \le\left| \int_{\Ga_D} \ga_0(u) \bm{p}_\varepsilon \cdot \nu \, \text{d}\,\mathcal{H}^{ {\mathrm{d}-1}} \right|.
		 \end{align*}
		{ Next, using the two inequalities above {in conjunction with}
			\[ \left| \int_{\Om \setminus \Om^\varepsilon} |\bm{p}_\varepsilon(x) \cdot\nabla u(x)| + \alpha(x) |\bm{p}_\varepsilon(x)|_{q} \dx \right| \le 2\varepsilon \|\bm q\|_{{L^\infty(\Omega)}}, \]
		 and
		 \[ \left| \int_{\Ga \setminus \Ga_D^\varepsilon} \ga_0(u)(\bm{p}_\varepsilon-q)\cdot \nu \text{d}\,\mathcal{H}^{{\mathrm{d}-1}} \right| \leq 2\varepsilon \|\bm q\|_{{L^\infty(\Omega)}}, \]}
		 we obtain that
		 \[\left|\int_{\Ga_D} \ga_0(u) \bm q\cdot \nu \, \text{d}\,\mathcal{H}^{ {\mathrm{d}-1}}\right| \leq 4\varepsilon \|\bm q\|_{ {L^\infty(\Omega)}}.\]
		 Now since $\bm{q}\in C^1(\overline{\Om})$ and $\varepsilon>0$ {have been chosen} arbitrarily, {it follows that}
		 \[\left|\int_{\Ga_D} \ga_0(u) \bm q\cdot \nu \, \text{d}\,\mathcal{H}^{ {\mathrm{d}-1}}\right| = 0, \quad \text{for all} \; \bm{q}\in C^1(\overline{\Om}). \]
		 {This immediately leads to the required result,  $\ga_0(u)=0$ a.e. on $\Ga_D$, and the proof is 
		 complete.}
	\end{enumerate}
\end{proof}

Finally, the converse result remains to be shown, i.e., if $u\in K$, then $J^\ast(\div^\ast u)=0$;  we prove this next.

\begin{lemma}\label{lemma:polar_F_step_4}
	If $u\in K$, then $J^\ast(\div^\ast u)=0$.
\end{lemma}

		%%%%%%%%%%%%% Lemma 1 %%%%%%%%%%%%%%%%%%%
	\begin{proof} 	 
			 {Since $u \in K$, therefore by the definition of $K$,} it holds that $u\in W^{1,1}_{\Gamma_D}(\Omega)$ and $|\nabla u|_{p}~\le~\al$
		 a.e. in $\Om$. {Next, using the definition of the convex conjugate $J^\ast$ of $J$, we obtain that}
		\begin{align}
			J^\ast(\div^\ast\, u) &= \sup_{\bm p\in \Va}\left\{\langle \div^\ast\, u, \bm p \rangle_{ {\Va^*,\Va}} - \int_\Omega \alpha(x)|\bm p(x)|_{q}\dif x\right\} \nonumber \\
														&= \sup_{\bm p\in \Va}\left\{(u, \div\: \bm p)_{} - \int_\Omega \alpha(x)|\bm p(x)|_{q}\dif x\right\} .\label{eq:Jstar}  
		\end{align}		
 {Next, by using the density of $C^1(\overline\Omega)^{ {\mathrm{d}}} \cap \Va$ in  $\Va$, from \eqref{eq:Jstar}, we obtain that }
		\begin{equation*}
		\begin{aligned}        
	J^\ast(\div^\ast\, u) &= \sup_{\bm p\in C^1(\overline\Omega)^{ {\mathrm{d}}} \cap \Va}\left\{\int_\Omega u(x)\ \div\: \bm p(x)\,\dx - \int_\Omega \alpha(x)|\bm p(x)|_{q}\dif x\right\}\\
	&= \sup_{\bm p\in C^1(\overline\Omega)^{ {\mathrm{d}}} \cap \Va}\left\{- \int_\Omega \bm p(x)\cdot\nabla u(x) \dx - \int_\Omega\alpha(x)|\bm p(x)|_{q}\dif x + \int_{\Ga_D} \ga_0(u) \bm p\cdot \nu \, \textup{d}\,\mathcal{H}^{ {\mathrm{d}-1}}\right\}\\
		&\le  \sup_{\bm p\in C^1(\overline\Omega)^{ {\mathrm{d}}} \cap \Va}\left\{\int_{\Omega} |\nabla u(x)|_{p}|\bm p(x)|_{q}\dif x - \int_\Omega\alpha(x)|\bm p(x)|_{q}\dif x \right\}\\
		&\leq 0 . 
		\end{aligned}
		\end{equation*}
		%where the third equality holds by density and because the function is continuous.
		Thus, since $J^\ast(\div^\ast\, u)$ is non-negative  {(we can set $\bm p \equiv 0$ in the definition of $J^\ast$)}, it follows that $J^\ast(\div^\ast\, u)=0$  {and the proof is complete.}
	\end{proof}
	
	Next we compute the {conjugate} function of the function $F$.
	\begin{proposition}\label{lemma:polar_G}
		The  {conjugate} function of $F$  {defined in \eqref{eq:FF}} is given by
		\begin{equation}\label{eq:G_star}
			F^\ast(u) = \frac12\|u\|_{L^2(\Omega)}^2 + (f,u)
		\end{equation} 
	\end{proposition}
	\begin{proof}
		The proof is an immediate consequence of the definition of $F^\ast$. Recalling the definition of~$F^\ast$ and rearranging the terms, we obtain that 
		\begin{align*}
		F^\ast(u) &= \sup_{v\in L^2(\Omega)}\{(u, v)_{ } - F(v)\}
		%=& \sup_{v\in L^2(\Omega)}\left\{(u, v)_{L^2(\Omega)} - \frac12\int_\Omega |v+f|_2^2\right\}\\
		    = \sup_{v\in L^2(\Omega)}\left\{(u, v)_{ } - \frac12\|v-f\|^2_{L^2(\Omega)}\right\}\\
		 &= \sup_{v\in L^2(\Omega)}\left\{(u+f, v)_{ } - \frac12\|v\|^2_{L^2(\Omega)}\right\} - \frac12\|f\|^2_{L^2(\Omega)} .
		\end{align*}
		The result then follows {from elementary calculus}.		
\end{proof}

%%%%%%%%%%%%%%%%%%%%%%%%%%%%
% resulting duality result %
%%%%%%%%%%%%%%%%%%%%%%%%%%%%
\begin{proposition}[\textsc{Strong duality}]
	\label{prop:duality}
	{The problem} \eqref{eq:original_problem_ini} is the Fenchel dual to problem  \eqref{eq:predual_problem_ini}, and for these problems the equality strong duality, i.e.,
	\begin{equation}\label{eq:strong_duality}
		\inf_{{\bm p}\in V_{\Gamma_N}(\Omega)} J({\bm p}) + F(\div\: {\bm p}) =
		-\inf_{u\in L^2(\Omega)} J^\ast(\div^\ast\, u) + F^\ast(-u)
	\end{equation}
	holds. Further, $\overline{{\bm p}}$ solves problem \eqref{eq:predual_problem_ini} if and only if the following extremality relation holds:
	\begin{equation}\label{eq:duality_rel}
		\begin{aligned}
		\overline{u} = f - \div\:\overline{\bm{p}} \quad {\mbox{in } \Omega}, \quad \text{ and } \quad\,\nabla \overline{{u}}\in  \partial J(\overline{\bm{p}}), 
		\end{aligned}
	\end{equation}
	where $\overline{u}$ denotes the solution of \eqref{eq:original_problem_ini},
	and $\partial J(\bm q)$ denotes the subdifferential of $J$ at a point $\bm q$.
\end{proposition}

\begin{proof}
	 As a corollary to Theorems~\ref{lemma:polar_F} and \ref{lemma:polar_G}, it immediately follows that
	 the dual of problem~\eqref{eq:predual_problem_ini} which is given in \cref{eq:dual_problem_decomposed} is identical to problem~\eqref{eq:original_problem_ini}.  Using that $J$ and $F$ are convex and continuous proper functions and bounded from below,
	equality \eqref{eq:strong_duality} and the extremality relation \eqref{eq:duality_rel} follow from the application of Theorem~III.4.1 and Proposition~III.4.1 in \cite[p. 59]{ekeland_temam} in its decomposed form, which is described in Remark~III.4.2 therein, where condition (4.20) is satisfied by any ${\bm p}\in V_{\Gamma_N}(\Omega)$.
\end{proof}

{
	\begin{remark}
	The duality between \eqref{eq:original_problem_ini} and \eqref{eq:predual_problem_ini} holds symmetrically, i.e. \eqref{eq:predual_problem_ini} is the dual to problem \eqref{eq:original_problem_ini} as well. Defining the perturbation function
$\phi: V_{\Gamma_N}(\Omega) \times L^2(\Omega) \to \mathbb{R}\cup\{\infty\}$
by
$\phi(\bm p, u) = J(\bm p) + F(\div\:\bm p - u)$
following the framework  in \cite[pp. 58--60]{ekeland_temam},
\eqref{eq:predual_problem_ini} can be written as
\[
	\inf_{\bm p \in V_{\Gamma_N}(\Omega)} \phi(\bm p, 0)
\]
and the application of \cite[(4.20) in  p. 61]{ekeland_temam} yields
that $\phi$ is 
convex,
l.s.c.,
proper,
and bounded from below,
given that the same holds true for $J$ and $F$. Thus, it follows from \cite[p. 49]{ekeland_temam} that $\phi^{\ast\ast} = \phi$ and that
	\eqref{eq:predual_problem_ini}
	is identical to its bidual problem
	\[
		\inf_{p \in V_{\Gamma_N}(\Omega)} \phi^{\ast\ast}(\bm p, 0),
	\]
	i.e., to the dual problem to \eqref{eq:original_problem_ini},
	{with respect to the perturbation function $\phi^{\ast}$.}
\end{remark}

{
	Note that though we assume $\alpha\in C(\overline\Omega)$,
	results within this section  hold for $\alpha\in L^1(\Omega)$. However recall that the existence result for this case (c.f. \Cref{sec:apfunc}) only stands in the case $\mathrm{d}=1$.
}

%%%%%%%%%%%%%%%%%%%%%%%%%%%%%%%
% Generalization to p measure %
%%%%%%%%%%%%%%%%%%%%%%%%%%%%%%%

%%%%%%%%%%%%%%%%%%%%%%%%%%%%%%%%%%%%%%%%%%%%%%%%%%%%%%%%%%%%%%%%%%%%
\subsubsection{Duality when $\bm p$ is  a measure}\label{Ssec:p_measure}
%%%%%%%%%%%%%%%%%%%%%%%%%%%%%%%%%%%%%%%%%%%%%%%%%%%%%%%%%%%%%%%%%%%%

 {We consider now the duality result in the framework of the variable $\bm p$ in the space of Borel measures with $L^2(\Omega)$ divergences. Surprisingly, the dual problem remains the same. We recall in this framework that}
\begin{equation*}
	\Gamma_N=\emptyset \qquad \text{and}\qquad V_{\Gamma_N}(\Omega)=W(\Omega),
\end{equation*}
as in \eqref{eq:Vacont}, and  $\mathrm{d} \ge 1$. Since we already assumed that $\alpha\in C(\overline{\Omega})$ is positive, existence of a unique solution  {follows from} 
%is given; see 
Theorem~\ref{thm:unique_existence_predual_p_measure}.

We again propose to follow the Fenchel dual approach and let $J: V_{\Gamma_N}(\Omega)\to\RR$ and $F: L^2(\Omega)\to\RR$ be
	\begin{equation*}
	J(\bm p) := \int_\Omega \alpha \:\mathrm{d}|\bm p|_q
	\qquad\text{and}\qquad
	F(v) := \frac12 \int_\Omega |v(x) - f(x)|^2\dx.
	\end{equation*}
In this setting, it also holds that for every $u\in L^2(\Omega)$, we have
		$J^\ast(\div^\ast\, u) = I_K(u)$, i.e., \Cref{lemma:polar_F}. In fact, we show that   \Cref{lemma:polar_F_a} and \Cref{lemma:polar_F_step_4} remain true under the functional analytic setting of this section.

\begin{proof}[Proof of \Cref{lemma:polar_F_a}]
	\begin{enumerate}
		\item [(\ref{lemma:polar_F_a1})]

	By choosing $\dif \bm p= \hat{ \bm p}(x) \dx$ with $\hat{ \bm p}\in C_0^1(\Omega)^{\mathrm{d}}$ and $|\hat{\bm p}|_{q}\leq 1$, we obtain the same inequality as \eqref{eq:Ineq1}.  {Moreover, by following similar steps as before, we can} show that $u\in \BV(\Omega)$ 
	%\end{proof}	
	
		\item [(\ref{lemma:polar_F_a2})]
	%\begin{proof}[Proof of \Cref{lemma:polar_F_a} (\ref{lemma:polar_F_a2})]
		The proof follows identically as before by considering $\dif \bm p_\epsilon= \hat{ \bm p}_\epsilon(x) \dx$ with $\hat{ \bm p}_\epsilon\in C_0^1(\Omega)^{\mathrm{d}}$.% satisfying \eqref{eq:p_eps_properties}. 
	%\end{proof}
	
		\item [(\ref{lemma:polar_F_a2regular})]
	%\begin{proof}[Proof of \Cref{lemma:polar_F_a} (\ref{lemma:polar_F_a2regular})]
		The same proof applies.
	%\end{proof}
	
	%\begin{proof}[Proof of \Cref{lemma:polar_F_a} (\ref{lemma:polar_F_a3})]
		\item [(\ref{lemma:polar_F_a3})]
		Note that $|\nabla u|\leq \alpha $ a.e. implies that {$u\in W^{1,\infty}(\Omega)$, given that $\alpha\in 
		C(\overline{\Omega})$.} As shown before, in \Cref{rem:Jstar_prereq}, $J^\ast(\div^\ast u)<+\infty$ implies $J^\ast(\div^\ast u)=0$ which yields  for $\dif \bm p= \hat{ \bm p}(x) \dx$ with $\hat{ \bm p}\in C^1(\overline{\Omega})^N$ the following
		\begin{align*}
		-\int_\Omega \hat{\bm p}(x)  \cdot \,\nabla u(x) \dx + \int_{ {\partial\Omega}} u \hat{\bm p} \cdot \nu\, \text{d}\mathcal{H}^{ {\mathrm{d}-1}}- \int_\Omega \alpha(x){   | \hat{\bm p}(x)|_q} \dx \leq 0 ,\; \text{for all}\ \hat{\bm p} \in C^1(\overline{\Omega})^N,
		\end{align*}
		and the proof follows identically leading to $u|_{\Gamma_D}=0$.
	\end{enumerate}
	\end{proof}
	
		\begin{proof}[Proof of \Cref{lemma:polar_F_step_4}]
			Let $u\in K$, then from the definition of $K$, it follows that $u\in W^{1,1}_{0}(\Omega)$ and $|\nabla u|_p\leq \alpha$.  {Furthermore, } {
		\begin{equation*}
		\begin{aligned}
					J^\ast(\div^\ast\, u) &= \sup_{\bm p\in V_{\Gamma_N}(\Omega)}\left\{-\int_\Omega  \,\nabla u \cdot \dif \bm p   + \mathrm{N}_{\bm p}(u|_{\partial \Omega}) - \int_\Omega \alpha \:\mathrm{d}|\bm p|_q\right\}\\
				&= \sup_{\bm p\in V_{\Gamma_N}(\Omega)}\left\{-\int_\Omega \,\nabla u\cdot \dif \bm p   - \int_\Omega \alpha \:\mathrm{d}|\bm p|_q\right\}\\
				&\leq \sup_{\bm p\in V_{\Gamma_N}(\Omega)}\left\{\int_\Omega |\nabla u|_p\:\mathrm{d}|\bm p|_q   - \int_\Omega \alpha \:\mathrm{d}|\bm p|_q\right\}\\
				&= \sup_{\bm p\in V_{\Gamma_N}(\Omega)}\left\{\int_\Omega (|\nabla u|_p-\alpha) \:\mathrm{d}|\bm p|_q\right\}\\
														&\leq 0
		\end{aligned}
		\end{equation*}}
i.e., it follows that $J^\ast(\div^\ast u)=0$. The proof is complete.
	\end{proof}

From \Cref{lemma:polar_F},  it follows that  the duality result of \Cref{prop:duality} also holds in this setting; the proof is straightforward.

%%%%%%%%%%%%%%%%%%%%%%%%%%%%%%%%%%%%%%%%%%%%%%%%%%%%%%%%
\subsection{The case when \texorpdfstring{$\alpha$}{alpha} is a measure%Stationary problem with $\alpha$ a measure
} \label{Sec:3final}
%%%%%%%%%%%%%%%%%%%%%%%%%%%%%%%%%%%%%%%%%%%%%%%%%%%%%%%%

	In this section, we will extend the duality result of \Cref{prop:duality} by letting $\alpha$ to be a non negative Borel measure, that is, $\alpha\in \mathrm{M}^+(\Omega)$. {However, $\bm p$ is a function in this setting.}
	In its more general form, in  problem~\eqref{eq:predual_problem_ini}, we set
	\begin{equation}\label{eq:predual_problem_gen_J}
		J({\bm p}) = \int_\Omega |\bm p|_q \dif\alpha \, .
	\end{equation}
	 {The results in this subsection are} a generalization of the case of the Lebesgue integrable constraint $\alpha$, that was presented in \Cref{Ssec:pre_dual_function}.
	 {We shall assume that $\bm p \in V_{\Ga_N}(\Omega)$, see \cref{eq:V_gamma_N_defn} for the definition of 
	$V_{\Ga_N}(\Omega)$.}
	%\[ V_{\Ga_N}\coloneqq\text{cross refer it to previous section} \] 
	
	Since $\alpha\in \mathrm{M}^+(\Omega)$, as we discussed in \Cref{Sec:1}, we let
	\begin{equation}
		U_{\Ga_D}(\Om)=\BV_{\Ga_D}(\Om) \qquad \text{and}\qquad G=D,
	\end{equation}
the distributional gradient, and hence
	\begin{equation*}
	K= \{v\in \BV_{\Gamma_D}(\Omega) : |\mathrm{D} v|_{p}\leq\alpha\}.
\end{equation*}
	We prove that the dual problem to \eqref{eq:predual_problem_ini} is given by
	\cref{eq:original_problem_ini}	with inequality {constraint $|\mathrm{D}u|_q\le \alpha$} being understood {in the sense of \cref{eq:grad_constraint_measure_Cbd}}.
	
	Recall that in \cref{Sec:existence} we  have shown  existence and uniqueness of solution to \Cref{eq:original_problem_ini}. Here, we show that dual of problem \eqref{eq:predual_problem_ini} is given by \eqref{eq:original_problem_ini}. We start by writing \cref{eq:predual_problem_ini} as
	\[ \inf_{\bm  p\in V_{\Ga_N}(\Omega) }  {F(\div\:\bm p)} + J(\bm p), \]  
	with $J: V_{\Ga_N}(\Omega)\to\RR$ as in \eqref{eq:predual_problem_gen_J} and $F: L^2(\Omega)\to\RR$, as before, given by
	\begin{equation*}
	F(v) := \frac12 \int_\Omega |v(x) - f(x)|^2\dx.
	\end{equation*}
	We prove now that Theorem~\ref{lemma:polar_F} holds also true in the current setting.  {For brevity, we only discuss the essential modifications needed in} \Cref{lemma:polar_F_a,lemma:polar_F_step_4}.

	\begin{proof}[Proof of Theorem~\ref{lemma:polar_F}] 
		This proof  follows along the same lines as the proof to Theorem~\ref{lemma:polar_F}. 
		We start by observing that the discussion in  {Remark~\ref{rem:Jstar_prereq}} holds  {in the current setting as well, i.e.,} $J^\ast(\div^\ast\:u)$ only takes the values $0$ and $+\infty$. We now prove the result.

		The proof that $J^\ast(\div^\ast u)=0$ implies that $u \in K$
		follows along the lines of \Cref{lemma:polar_F_a}. {Indeed \eqref{lemma:polar_F_a1} and \eqref{lemma:polar_F_a2} in \Cref{lemma:polar_F_a}} apply directly, and  {for} \eqref{lemma:polar_F_a3} everything follows in the same way, when $\mathrm{D}u$ and $\dif\alpha$ are measures instead of the functions $\nabla u$ and  $\alpha(x)$.

		 {On the other hand, the converse (\Cref{lemma:polar_F_step_4}), i.e.,} $u\in K$ implies that $J^\ast(\div^\ast u)=0$  follows from the calculations below. 
		 {Recall that if $u\in K$, then} $u\in \BV_{\Gamma_D}(\Omega)$ and $|\mathrm{D} u|_{p} ~\le~\al$ in the sense of \eqref{eq:grad_constraint_measure_Cbd}.  Therefore,
			\begin{equation*}
			\begin{aligned}
			0\le J^\ast(\div^\ast\, u) &= \sup_{\bm p\in \Va}\left\{\langle \div^\ast\, u, \bm p \rangle_{ {(\Va)^*,\Va}} - \int_\Omega {|\bm p|_{q}}\text{d}\alpha\right\}\\
																 &= \sup_{\bm p\in C^1(\overline\Omega)^{ {\mathrm{d}}} \cap \Va}\left\{ {-\int_{\Omega} \bm p\, \mathrm{D}u}  - \int_\Omega{|\bm p|_{q}} \text{d}\alpha+ \int_{\Ga_D} \ga_0(u) \bm p\cdot \nu \, \textup{d}\,\mathcal{H}^{ {\mathrm{d}-1}}\right\}\\
																 &= \sup_{\bm p\in C^1(\overline\Omega)^{ {\mathrm{d}}} \cap \Va}\left\{ {-\int_{\Omega} \bm p\, \mathrm{D}u}  - \int_\Omega{|\bm p|_{q}} \text{d}\alpha \right\}\\
			&\le \sup_{\bm p\in C^1(\overline\Omega)^{ {\mathrm{d}}} \cap \Va}\left\{\int_{\Omega} |\bm p|_q |\text{D}u|_p  - \int_\Omega{|\bm p|_q} \text{d}\alpha \right\}\leq 0 ,
			\end{aligned}
			\end{equation*}
	 {and the proof is complete.}		
	\end{proof}

Finally, note that it follows identically as before that the polar function of $F$ is given by
		\begin{equation}\label{eq:G_star_gen}
		F^\ast(u) : = \frac12\|u\|_{L^2(\Omega)}^2 - (f,u)_{ }.
		\end{equation} 
		Hence, the duality result of proposition \ref{prop:duality} also holds in the case where $\alpha$ is a measure, with $\nabla$ replaced by $\mathrm{D}$.
 
	%%%%%%%%%%%%%%%%%%%%%%%%%%%%%%%%%%%%%%%%%%%%%%%%%%%%%%%%
	\section{A Finite Element Method with Applications} \label{Sec:5final}
	%%%%%%%%%%%%%%%%%%%%%%%%%%%%%%%%%%%%%%%%%%%%%%%%%%%%%%%%

	%the content of numerics_result is now included below
	%\input{numerical_results.tex}
	%\input{numerics}
	The purpose of this section is to illustrate the applicability of the proposed primal-dual approach to solve Problems \eqref{eq:original_problem_ini} and \cref{eq:predual_problem_ini}. We assume throughout this section that $p=q=2$. 
	
	Recall that Problem \eqref{eq:original_problem_ini} in the case that $\alpha\in L^\infty(\Omega)^+$ is given by
	\begin{equation}	\label{eq:Umod}
			\min \frac12\int_\Om|u(x)|^2\dx - \int_\Omega f(x) u(x) \dx \quad \text{ over } u\in W^{1,\infty}_{\Gamma_D}(\Omega),\quad \text{ s.t. }  |\nabla u|_2\leq \alpha \: \text{ a.e. }
	\end{equation}
and that the pre-dual problem \cref{eq:predual_problem_ini}  is given by
	\begin{equation}
	\label{eq:Dmod}
	\min \frac12 \| \div\: \bm p - f \|_{L^2(\Omega)}^2 + \int_\Omega \alpha(x) | \bm p(x)|_{2}\dif x  \quad \text{ over } \bm p \in V_{\Ga_N}(\Omega)\, .
	\end{equation}
	Now the first order (necessary and sufficient) optimality condition corresponding to \eqref{eq:Dmod} in the strong form is given by: Find $\bm p:\Omega\to \mathbb{R}^\mathrm{d}$ satisfying 
	\begin{equation}\label{eq:pstrong}	
	\begin{aligned}
	-\nabla  \left( \div\: \bm p - f \right) + \partial \left( \| \alpha |\bm p|_{2} \|_{L^1(\Omega)} \right)   &\ni 0  \quad \mbox{in } \Omega, \\
	\bm p \cdot \nu &= 0 \quad \mbox{on } \Gamma_N,
	\end{aligned}				
	\end{equation}
	where $\partial$ denotes the subdifferential operator. In order to solve \eqref{eq:pstrong}, recall from the extremality conditions \eqref{eq:duality_rel}, that if $u^*$ and $\bm p^*$ are solutions to \eqref{eq:Umod} and \eqref{eq:Dmod}, respectively, they satisfy
	\begin{equation}\label{eq:u}
	u^* := -\div\: \bm p + f \quad \mbox{a.e. in } \Omega  \, .
	\end{equation}
	Then, a primal-dual system arises from \eqref{eq:pstrong} and \eqref{eq:u}, which in the weak form becomes the following variational inequality of second kind: Find $(\bm p, u) \in V_{\Ga_N}(\Omega) \times L^2(\Omega)$ such that  
	\begin{align}
	\left(u,-\div\, (\mathbf{v} - \bm p)\right) + \int_\Omega \alpha(x) |\mathbf{v}(x)|_{2}\dif x - \int_\Omega \alpha(x) |\bm p(x)|_{2}\dif x 
	&\ge 0 & &\mbox{for all } \mathbf{v} \in V_{\Ga_N}(\Omega), \label{eq:psad} \\
	(u,w) + (\div\: \bm p , w) &= (f,w) & &\mbox{for all } w \in L^2(\Omega) \, .	\label{eq:usad}
	\end{align}
	Due to their nonlinear and nonsmooth nature, it is challenging to solve \eqref{eq:psad}--\eqref{eq:usad}. 
	
	We shall proceed by introducing the Huber-regularization for $\phi(\bm p) := |\bm p|_{2}$ in the last term under the integral in \eqref{eq:Dmod}. This regularization is $C^{1}$ with piecewise differentiable first order derivative. Therefore one can use Newton type methods to solve the resulting regularized system.  For a given parameter $\tau > 0$, the Huber regularization of $\phi$ is given by 
	\[
	\phi_{\tau}(\bm p) 
	:= \left\{ 
	\begin{array}{ll}
	|\bm p|_{2} - \frac12 \tau,  & |\bm p|_{2} > \tau,\\
	\frac{1}{2\tau} |\bm p|_{2}^2, & |\bm p|_{2} \le \tau \, .
	\end{array}	
	\right.	
	\]
	As $\tau \rightarrow 0$, $\phi_\tau(\bm p)  \rightarrow \phi(\bm p)$. Moreover, $\phi_\tau(\cdot)$ is continuously differentiable with derivative given by 	
	\[
	\phi_\tau'(\bm p)
	:= \left\{ 
	\begin{array}{ll}
	\frac{\bm p}{|\bm p|_{2}},  & |\bm p|_{2} > \tau \\
	\frac{1}{\tau} \bm p, & |\bm p|_{2} \le \tau \, .
	\end{array}	
	\right.	
	\]

	Replacing $\phi(\cdot) = |\cdot|_{2}$ in \eqref{eq:Dmod} by $\phi_\tau(\cdot)$, the regularized primal-dual system corresponding to \eqref{eq:psad}--\eqref{eq:usad} is given by
	\begin{align}
	(u,-\div\, \mathbf{v}) + \int_\Omega \alpha \phi_\tau'(\bm p) \cdot \mathbf{v}
	&= 0,   & & \mbox{for all } \mathbf{v} \in V_{\Ga_N}(\Omega), \label{eq:psadr} \\
	(u,w) + (\div\: \bm p , w) &= (f,w), & & \mbox{for all } w \in L^2(\Omega). 	\label{eq:usadr}
	\end{align}	
	Notice, that $\phi_\tau'(\cdot)$ is piecewise differentiable and the second order derivative is given by 
	\[
	\phi_\tau''(\bm p) = 
	:= \left\{ 
	\begin{array}{ll}
	\frac{1}{|\bm p|_{2}} \left( I_{d\times d} - \frac{\bm p \bm p^\top}{|\bm p|_{2}^2} \right),  & |\bm p|_{2} > \tau\,, \\
	\frac{1}{\tau} I_{d\times d},  & |\bm p|_{2} \le \tau \, ,
	\end{array}	
	\right.	
	\]
	where $I_{d \times d}$ is the $d\times d$ identity matrix.
	
	\subsection{Finite Element Discretization}		
	
	We discretize $\bm p$ and $u$ using the lowest order Raviart-Thomas ($\mathbb{RT}_0$) and piecewise constant $(\mathbb{P}_0)$ finite elements, respectively. Whenever needed, the integrals are computed using Gauss-quadrature which is exact for polynomials of degree less than equal to 4. For each fixed $\tau$, we solve the discrete saddle point system \eqref{eq:psadr}--\eqref{eq:usadr} using Newton's method with backtracking line-search strategy. We stop the Newton iteration when each residual in $L^2(\Omega)$-norm is smaller than $10^{-8}$.  Each linear solve during the Newton iteration is done using direct solve. Starting from $\tau = 10$, a continuation strategy is applied where in each step we reduce $\tau$ by a factor of 1.30 until $\tau$ is less than equal to $10^{-6}$. We initialize the Newton's method by zero. To compute solution for next $\tau$, we use the solution corresponding to previous $\tau$ as the initial iterate for the Newton's method.
	
	%%%%%%%%%%%%%%%%%%%%%%%
	\subsection{Numerical Examples}	
	%%%%%%%%%%%%%%%%%%%%%%%
	Next, we report results from various numerical experiments. In all examples we consider $\Omega=(0,1)\times(0,1)$ and we assume that $\Gamma_N = \emptyset$, $\Gamma_D=\partial\Omega$, and hence   
			pure Dirichlet boundary conditions on $u$ on the entire boundary are set. In the  {first example}, we construct exact solutions $(\bm p,u)$ when $f$ and $\alpha$ are constants. We compare these exact solutions with our finite element approximation. These experiments validate our finite element implementation for constant $\alpha$ and $f$ and provide optimal rate of convergence.  {Additionally, we} solve \eqref{eq:predual_problem_ini} and \eqref{eq:original_problem_ini} first for a fixed $\alpha$ and vary $f$ and next we fix $f$ and vary $\alpha$. In our second  experiment, we consider a more generic $f$ with different features such as cone, valley and flat regions. In our final experiment, we consider $\alpha$ to be a measure.

	\begin{example}
		{\rm			
			Note initially that if $\alpha$ and $f$ are constants, it is possible to calculate an exact solution. By setting
			\[
				m(x) := \min\left\{ f, \alpha(x-1), \alpha(x-0)\right\},
			\]
			the exact $u$ and $\bm p$ are given as:
			\[
				u(x, y) = \min\{ m(x), m(y) \},
			\]
			and
			\[
				\bm p(x, y) = \begin{cases}
					\left(
						\frac{1}{\alpha}(m(y) - m(x)) \mathrm{sgn}(0.5 - x)
						\left( f - \frac{1}{2}\left(m(x)+m(y)\right)\right),
						0
					\right)^\top,&\text{if }|x-0.5|>|y-0.5|,\\
					\left(
						0,
						\frac{1}{\alpha}(m(x) - m(y)) \mathrm{sgn}(0.5 - y)
						\left( f - \frac{1}{2}\left(m(x)+m(y)\right)\right)
					\right)^\top,&\text{otherwise}.
				\end{cases}
			\]

			Notice that in this example, $u$ is again Lipschitz continuous. In Figure~\ref{f:exact_ex2} (top panel), we have shown the 
			$\|\bm p-\bm p_h\|_{L^2(\Om)}$ and $\|u - u_h\|_{L^2(\Om)}$ when $\Om = (0,1)^2$, $f = 1$, and $\alpha =1$.
			We observe optimal rate of convergence in both cases. In the bottom row, the left panel shows $u_h$, the middle panel
			shows $|\nabla u_h|_{2}$, and the right panel shows $\bm p_h$. We observe that, in this example, the gradient 
			constraints are active in the entire region. Notice, that at the corners (which are sets of measure zero), gradient is 
			undefined. 
			
			\begin{figure}[htb]
				\centering
				\includegraphics[width=0.3\textwidth]{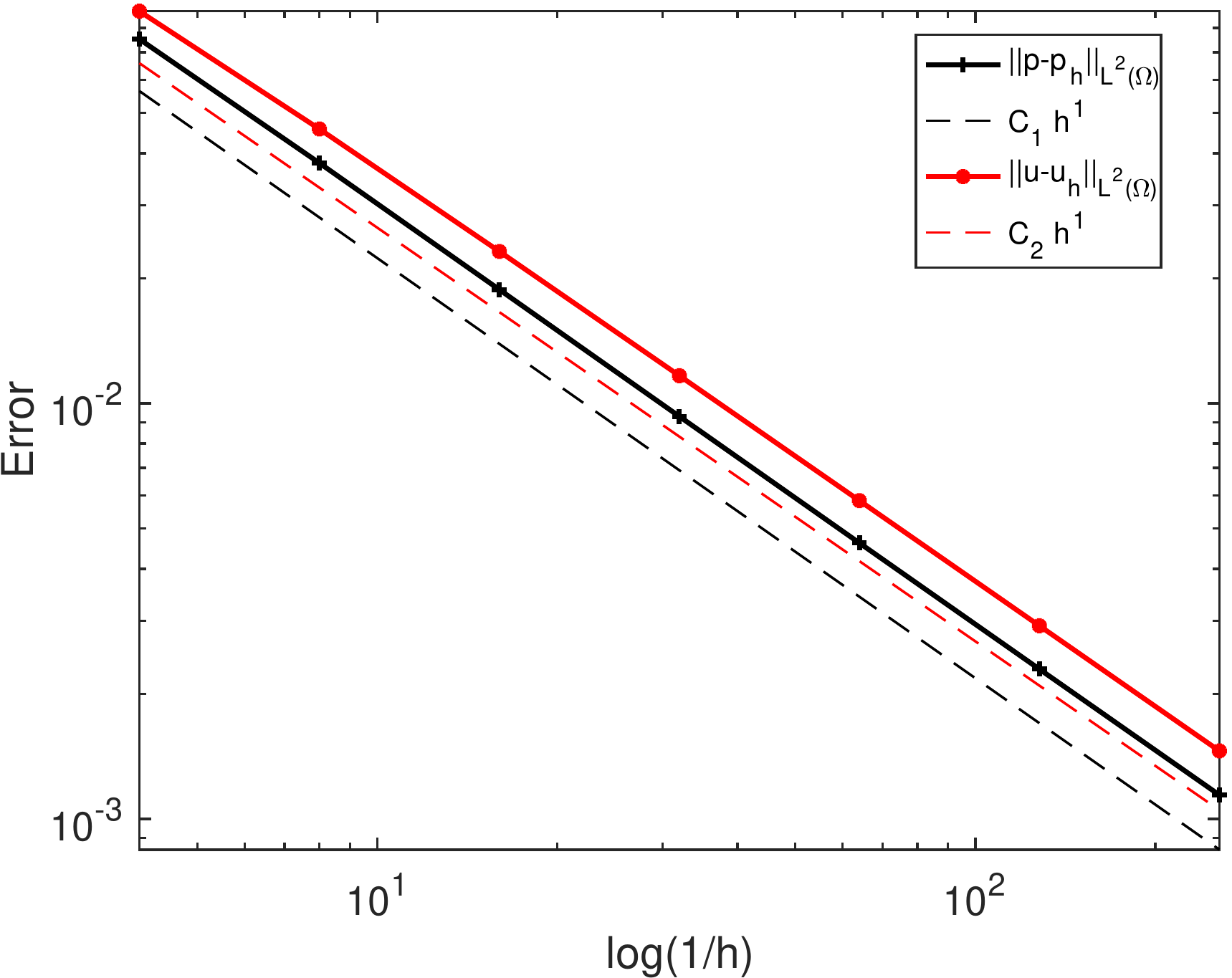} \\
				\includegraphics[width=0.32\textwidth]{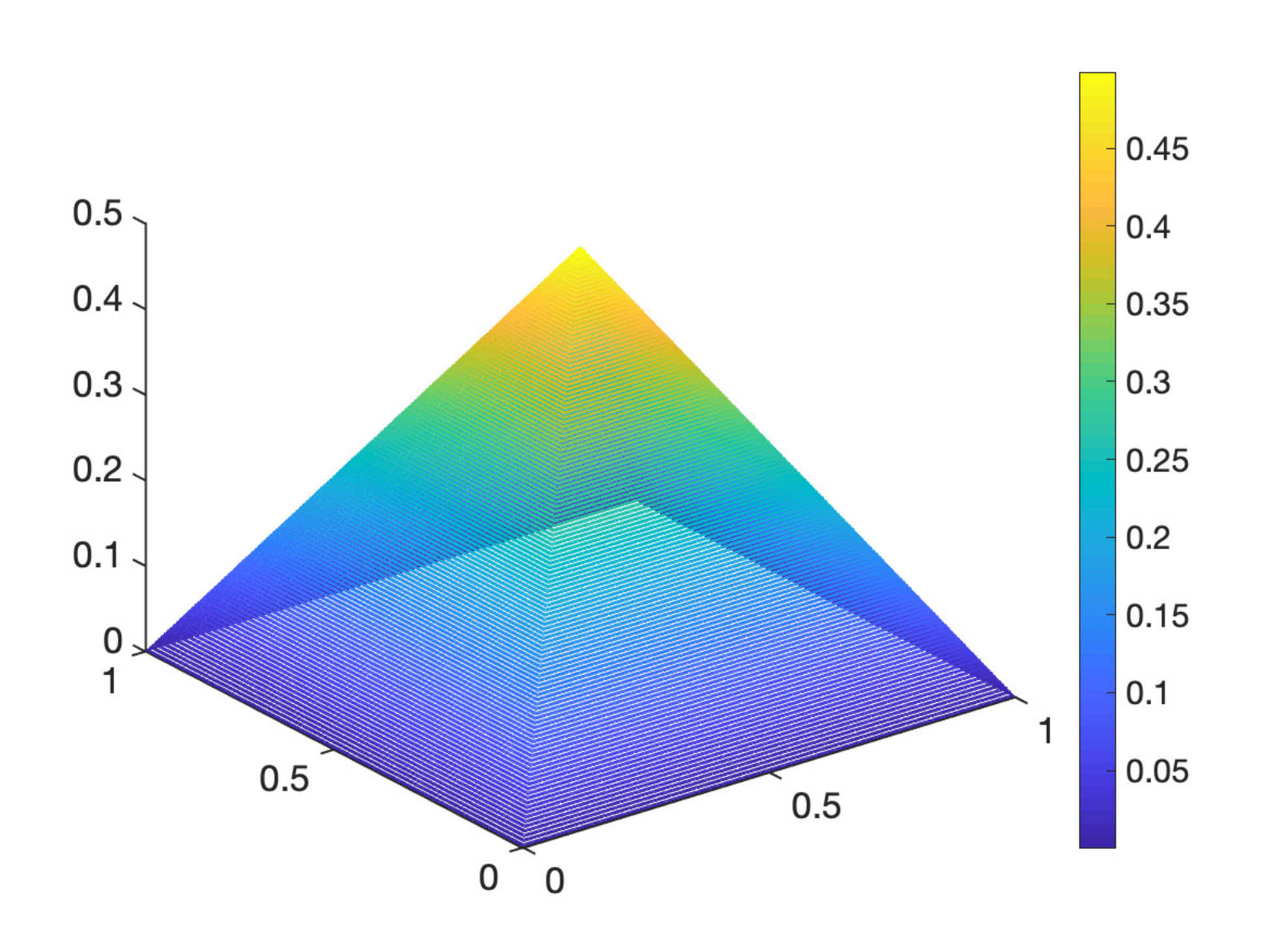}
				\includegraphics[width=0.32\textwidth]{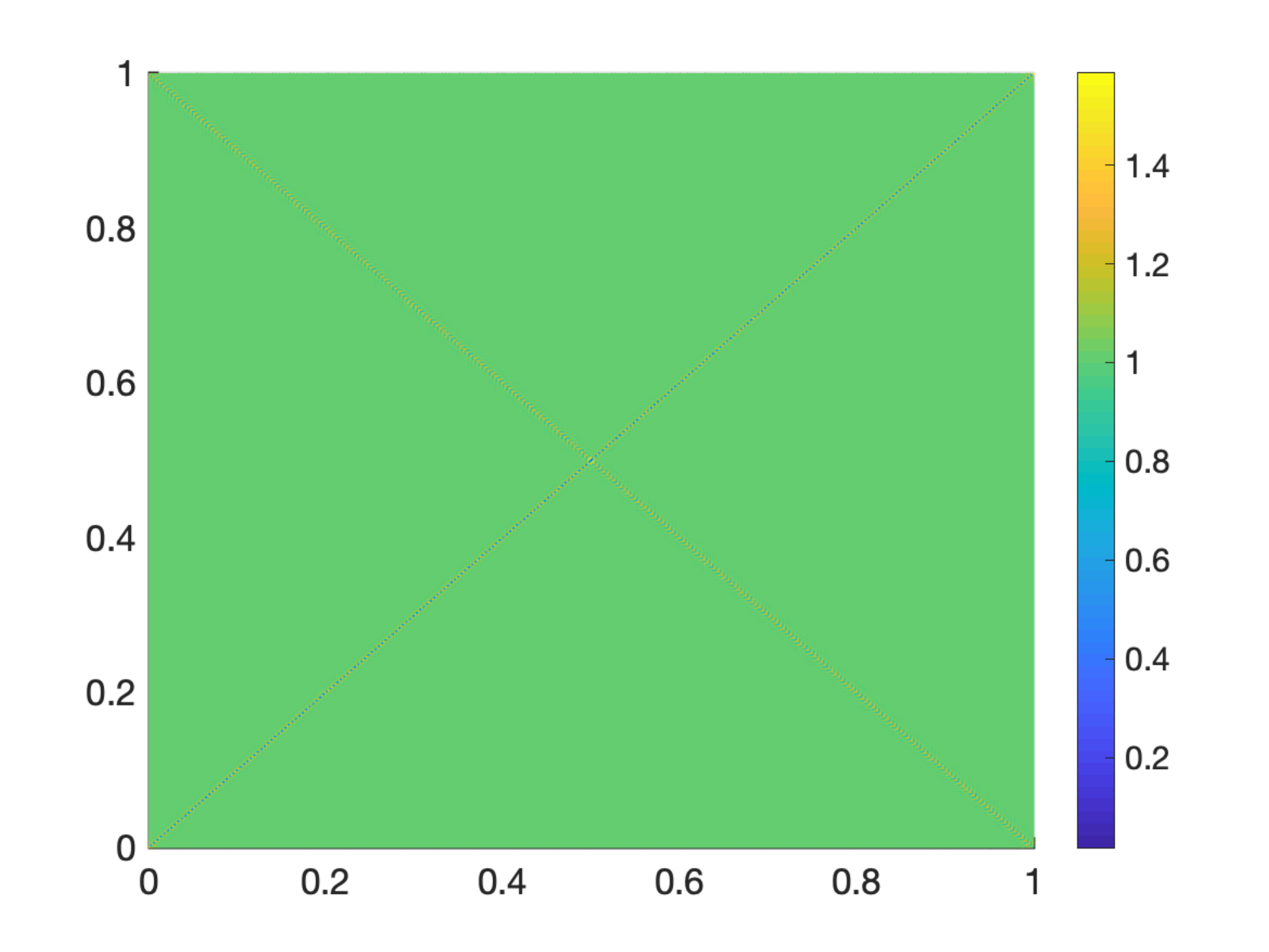}	
				\includegraphics[width=0.32\textwidth]{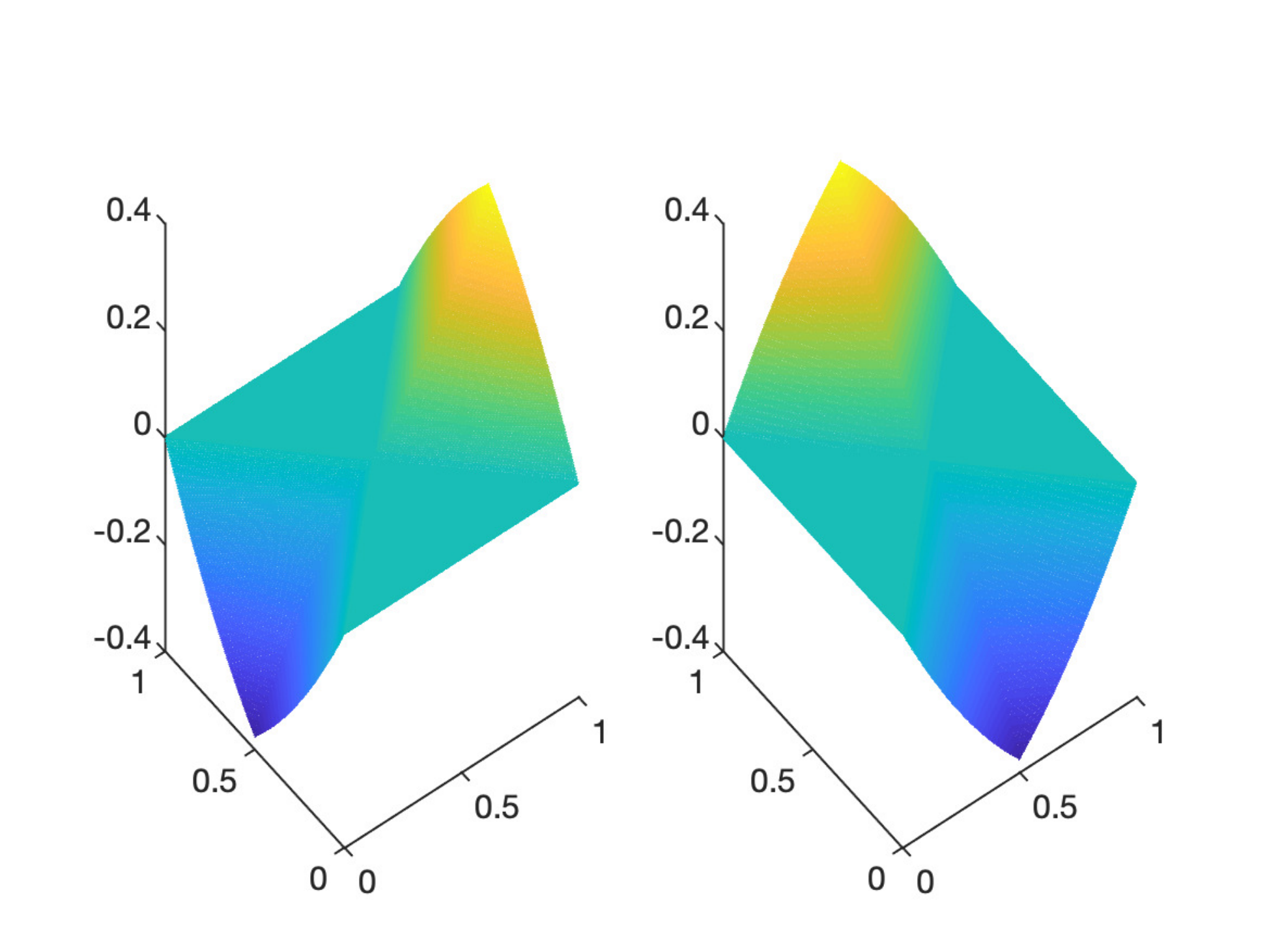}	
				\caption{{\bf Example 1:} \emph{Top panel} - We have shown the $L^2(\Om)$-error between the computed solution 
					$(u_h,\bm p_h)$ and the exact solution $(u,\bm p)$. The optimal linear rate of convergence is observed.
					Bottom panel: Computed $u_h$ (left), $|\nabla u_h|_{2}$ (middle), $\bm p_h$ (right). Notice that we are 
					touching the constraints in the entire region, except where the gradient is undefined. We have omitted the plots of
					the exact $(u,\bm p)$ as they look exactly same as $(u_h,\bm p_h)$.}
				\label{f:exact_ex2}
			\end{figure}

			Next, we fix the number of $\bm p_h$ and $u_h$ unknowns to be 197,120 and 131,072, respectively.  
			Figure~\ref{f:ex_2_1} shows  our results for 3 
			different experiments. In all cases, we have used a fixed $\alpha = 1$. The rows correspond to 
			$u_h, |\nabla u_h|_{2}$, and $\bm p_h$. Each column correspond to $f = 1$, $f = 0.25$, and 
			$f = 0.1$. As expected, for a large value of $f$, we observe steep slope, but for smaller values of
			$f$, plateau regions appear. We also observe that the active region shrinks as $f$ decreases since
			the gradient is zero at the top of the plateau. The dual variable $\bm p_h$ also changes significantly 
			with $f$. 
			
			\begin{figure}[htb]
				\centering
				\includegraphics[width=0.32\textwidth]{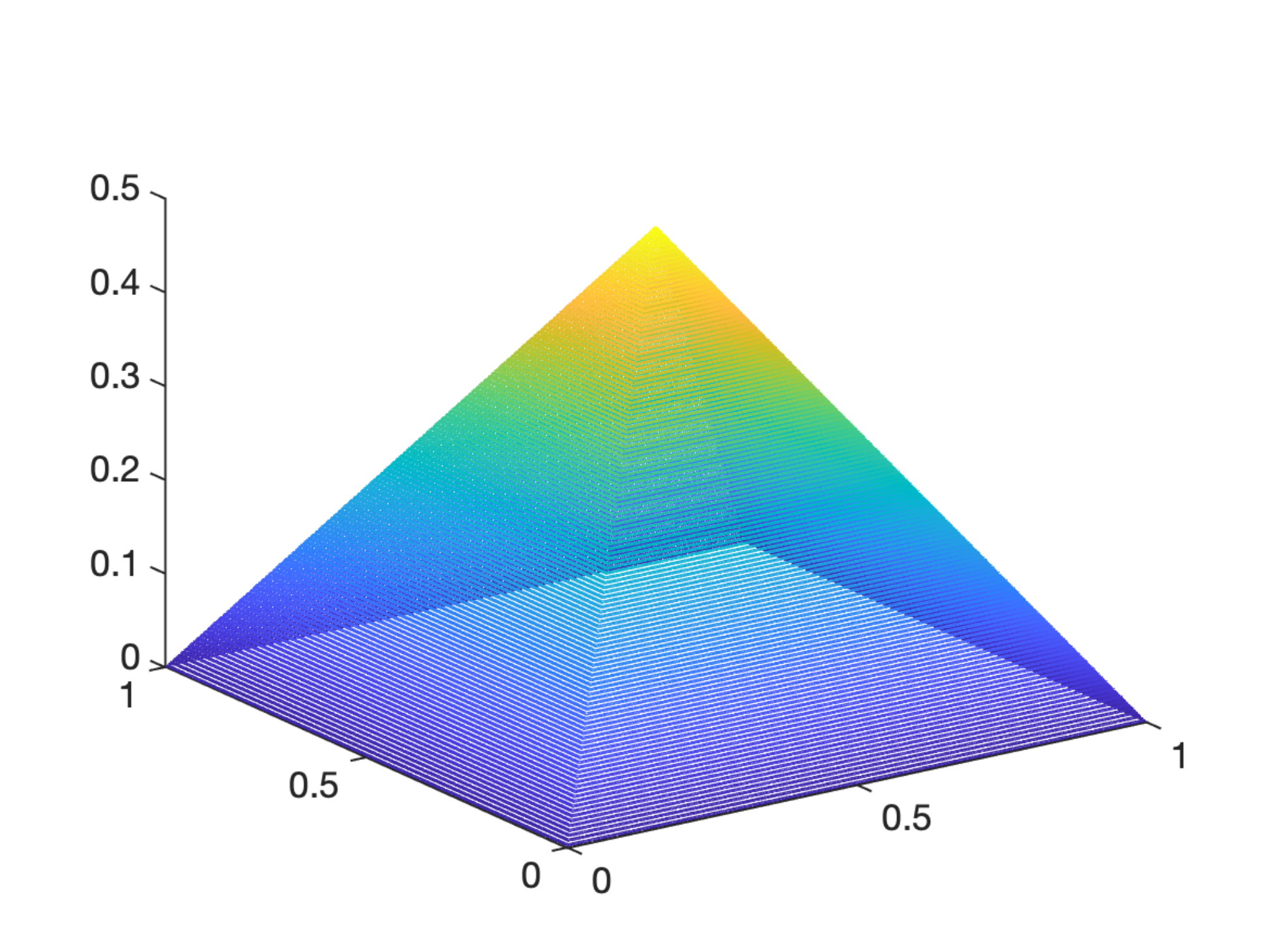}
				\includegraphics[width=0.32\textwidth]{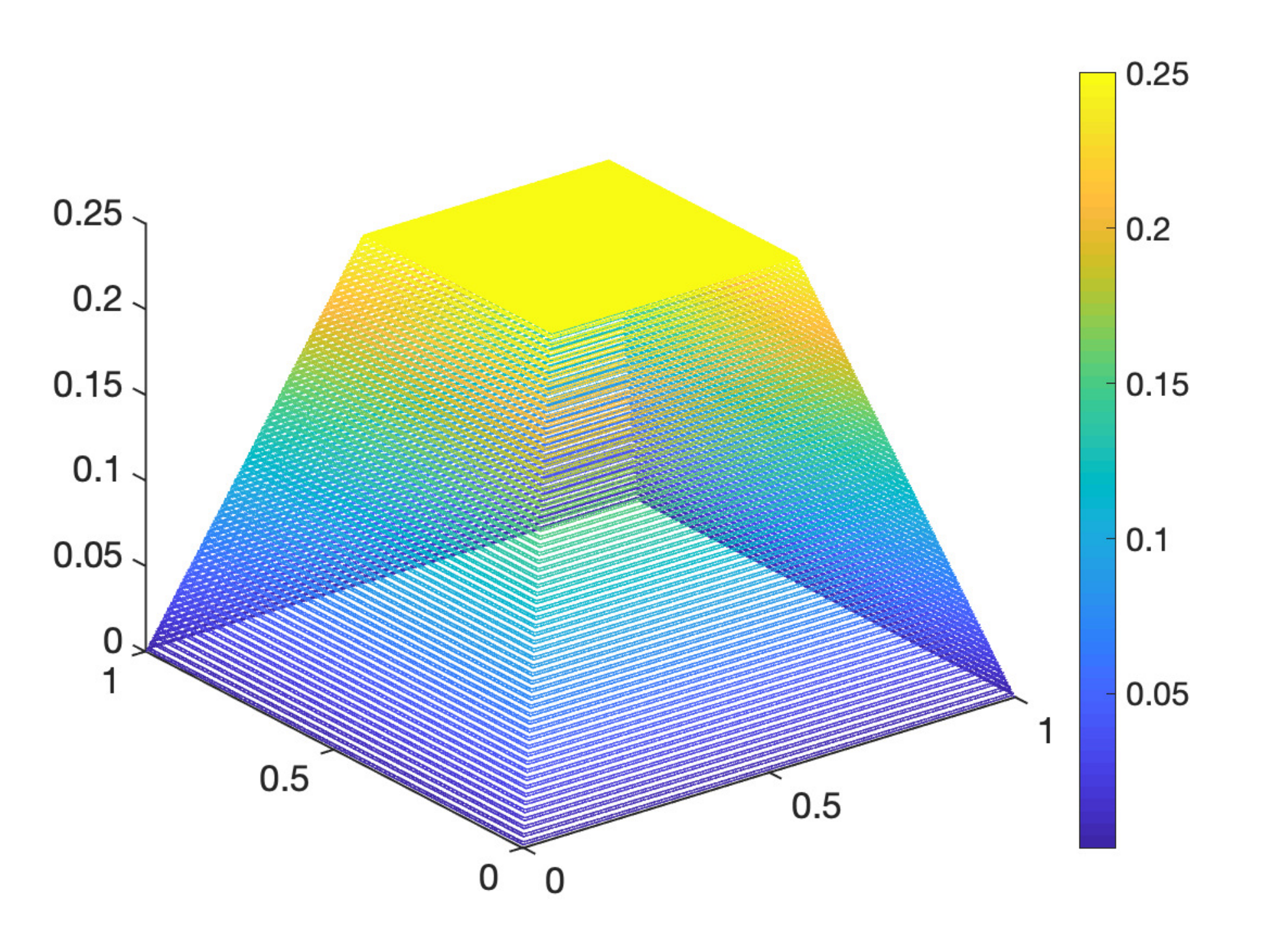}	
				\includegraphics[width=0.32\textwidth]{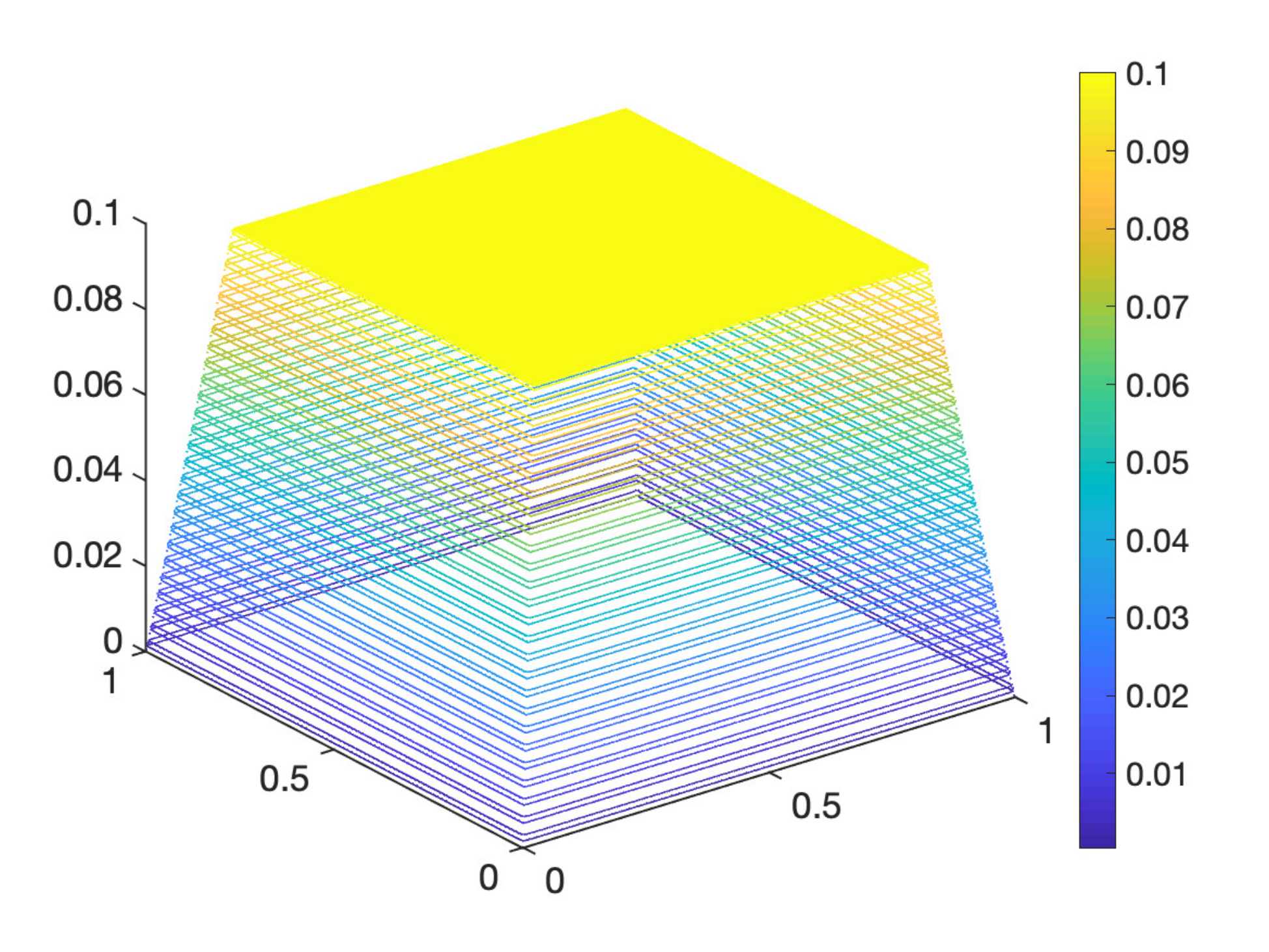}			
				\includegraphics[width=0.32\textwidth]{figures/ex_2/gu_mesh7_f_1_a_1_ex2}		
				\includegraphics[width=0.32\textwidth]{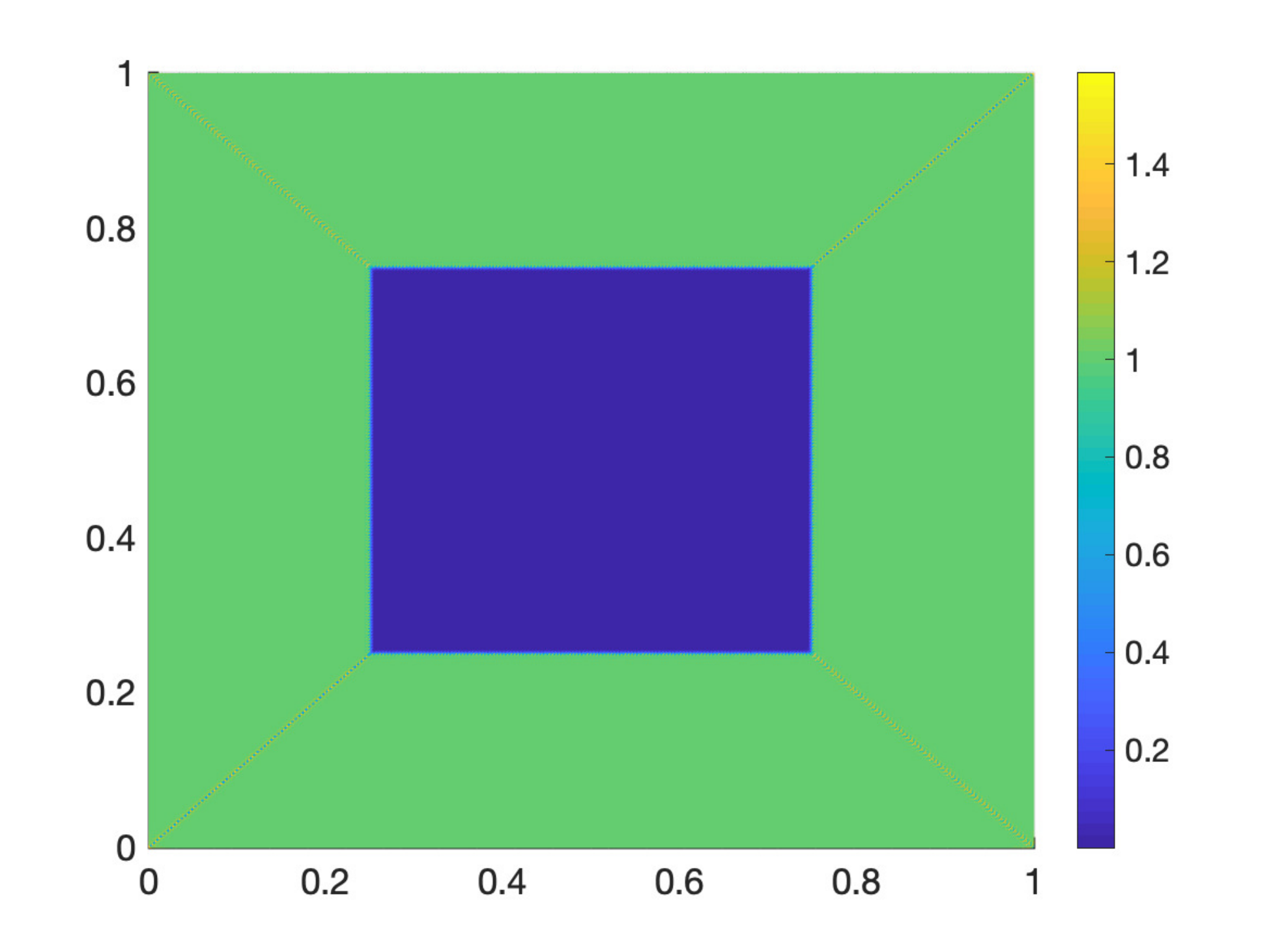}		
				\includegraphics[width=0.32\textwidth]{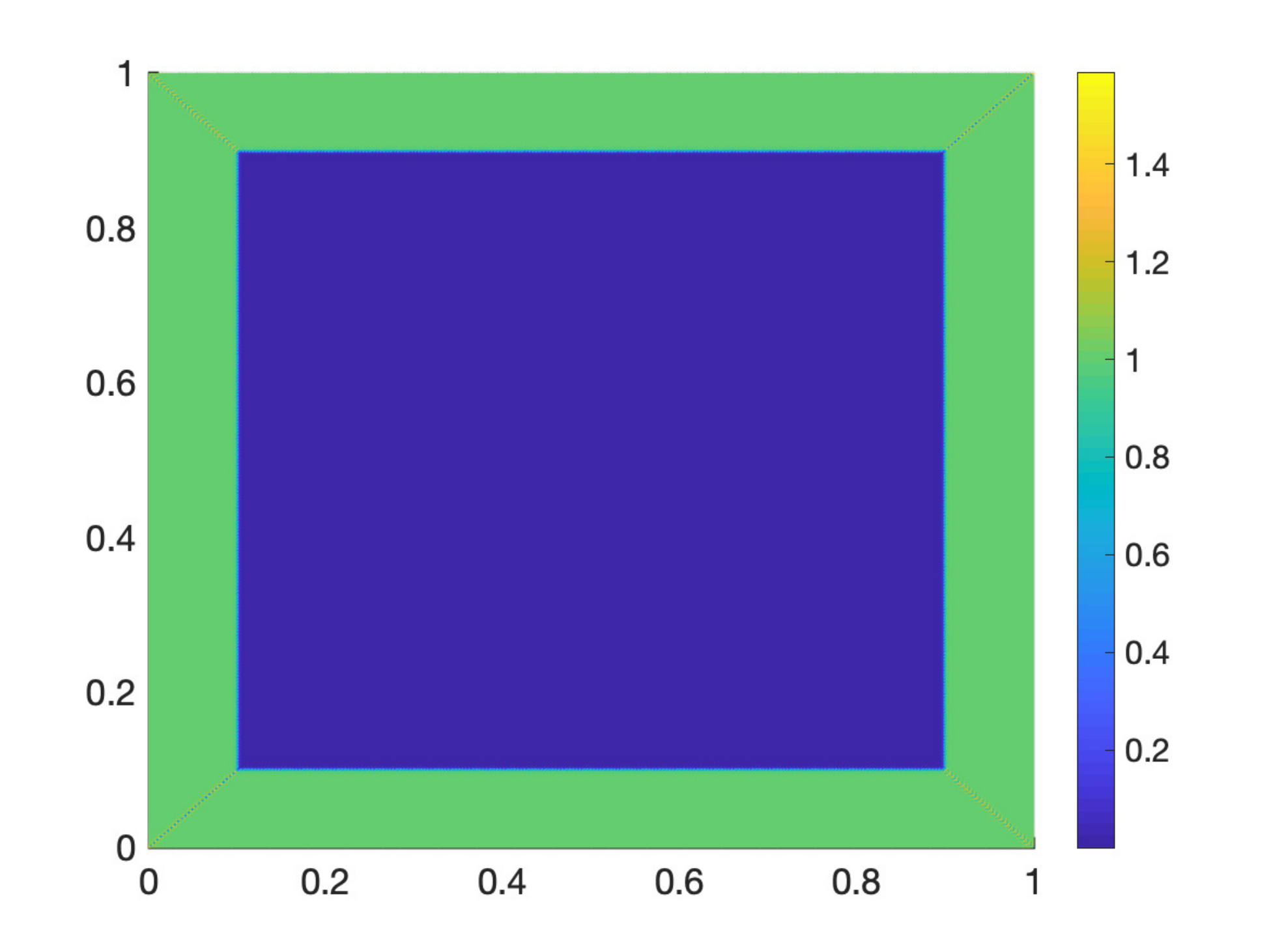}			
				\includegraphics[width=0.32\textwidth]{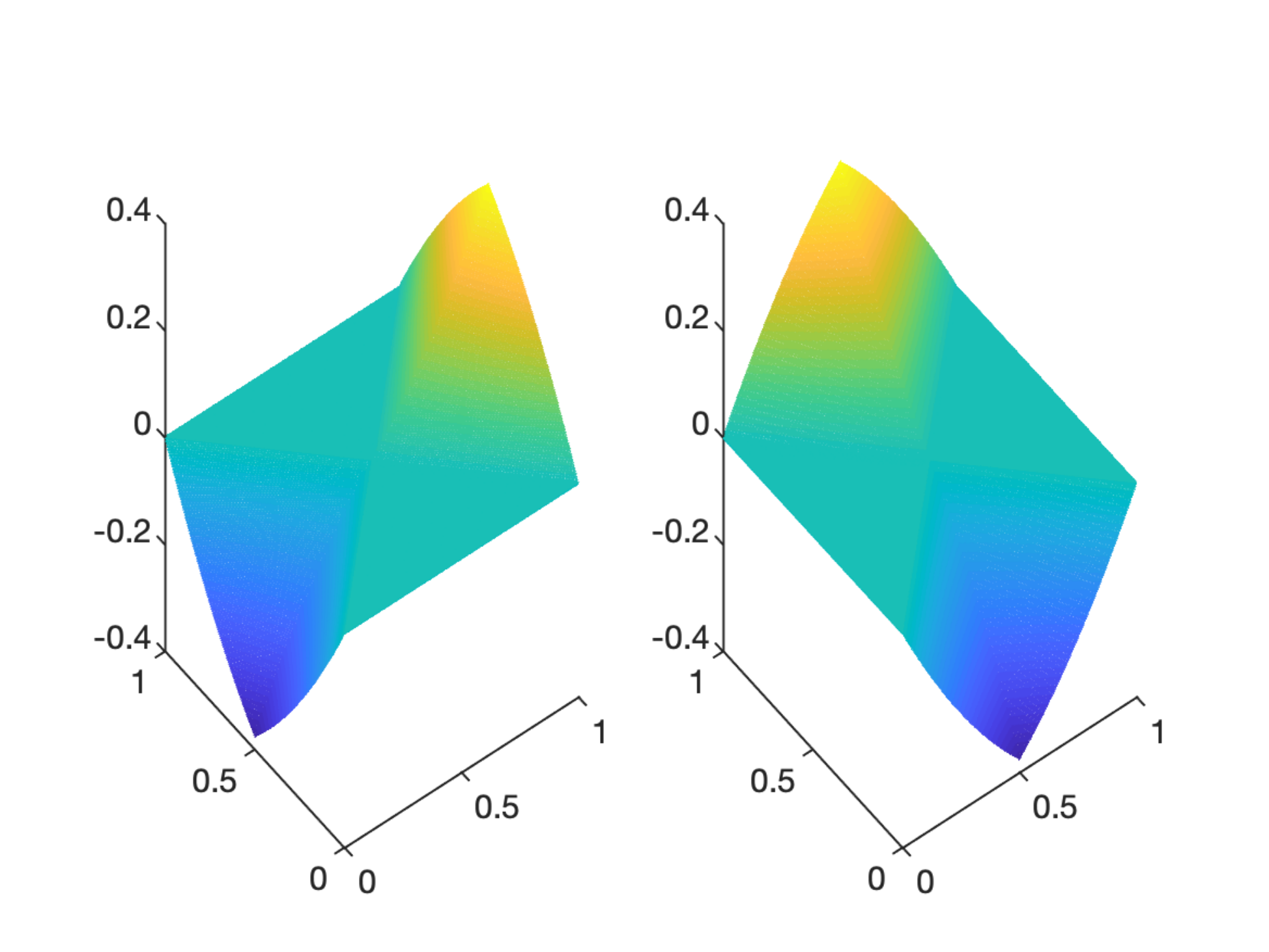}
				\includegraphics[width=0.32\textwidth]{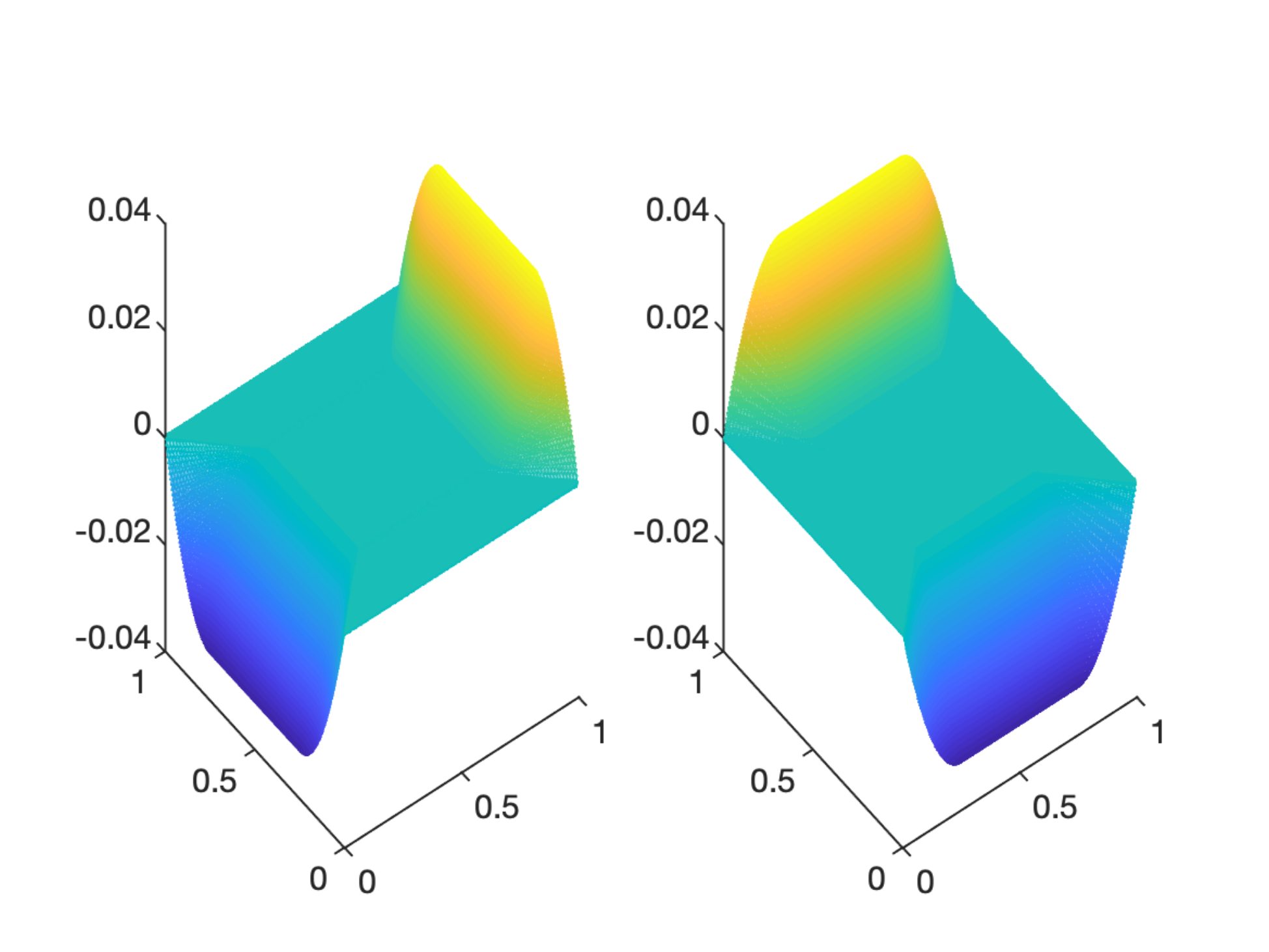}	
				\includegraphics[width=0.32\textwidth]{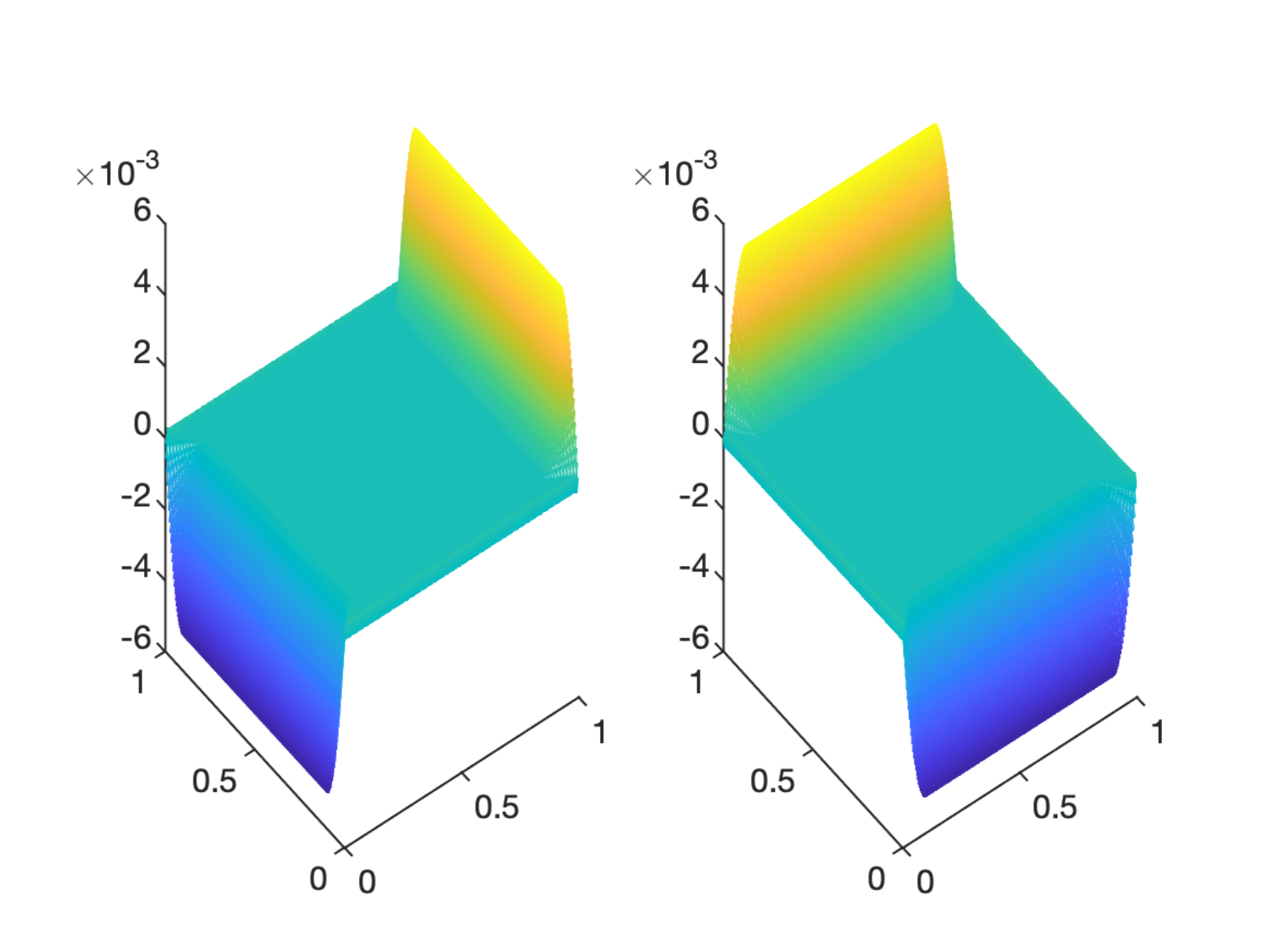}					
				\caption{{\bf Example~1 (fixed $\alpha$, varying $f$).} 
					The rows correspond to $u_h$, $|\nabla u_h|_{2}$, and $\bm p_h$. The columns 
					represent  $f = 1$, 0.25, and 0.1. In all cases, we observe that the gradient constraints are 
					active but the activity region shrinks as $f$ decreases, this is expected since the gradient on the
					plateau region is zero. The behavior of $\bm p$ also changes considerably with $f$.}
				\label{f:ex_2_1}
			\end{figure}

			In Figure~\ref{f:ex_2_2} we again show results from 3 
			different experiments. In all cases, we have used a fixed $f = 1$. The rows correspond to 
			$u_h, |\nabla u_h|_{2}$, and $\bm p_h$. Each column corresponds to 
			\[
			\alpha = 1, \alpha = 0.5, \mbox{ and } 
			\alpha = \left\{
			\begin{array}{ll}
			0.75 , & y \le 1-x, \\
			1.0,     & \mbox{otherwise} ,
			\end{array}	
			\right.
			\]
			respectively. In all cases, we observe that the gradient constraints are active in the entire domain 
			(except on a set of measure zero). For the case of piecewise constant $\alpha$, nonsmoothness
			in $\bm p_h$ is clearly visible.

			\begin{figure}[htb]
				\centering
				\includegraphics[width=0.32\textwidth]{figures/ex_2/u_mesh7_f_1_a_1_ex2}
				\includegraphics[width=0.32\textwidth]{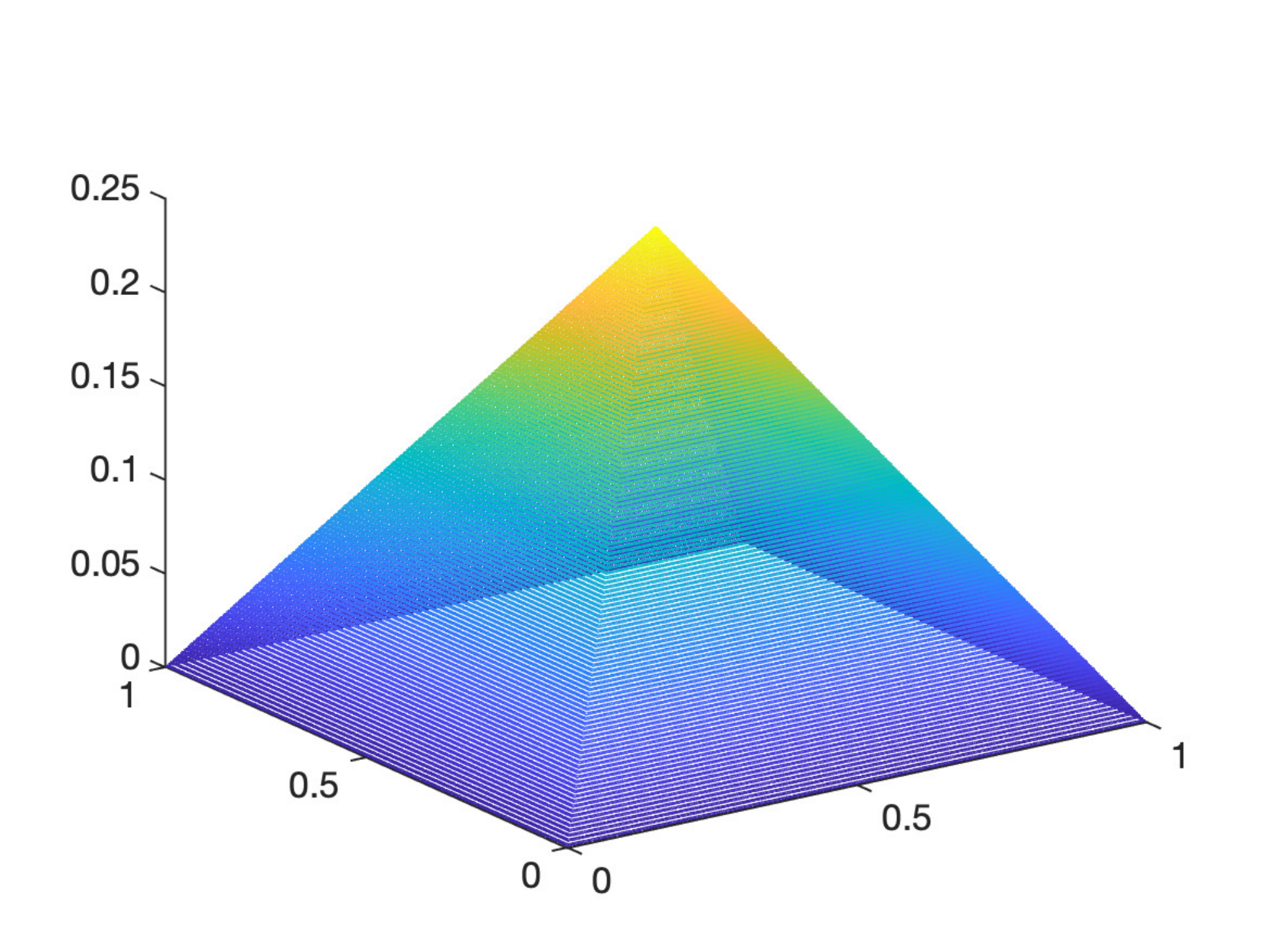}	
				\includegraphics[width=0.32\textwidth]{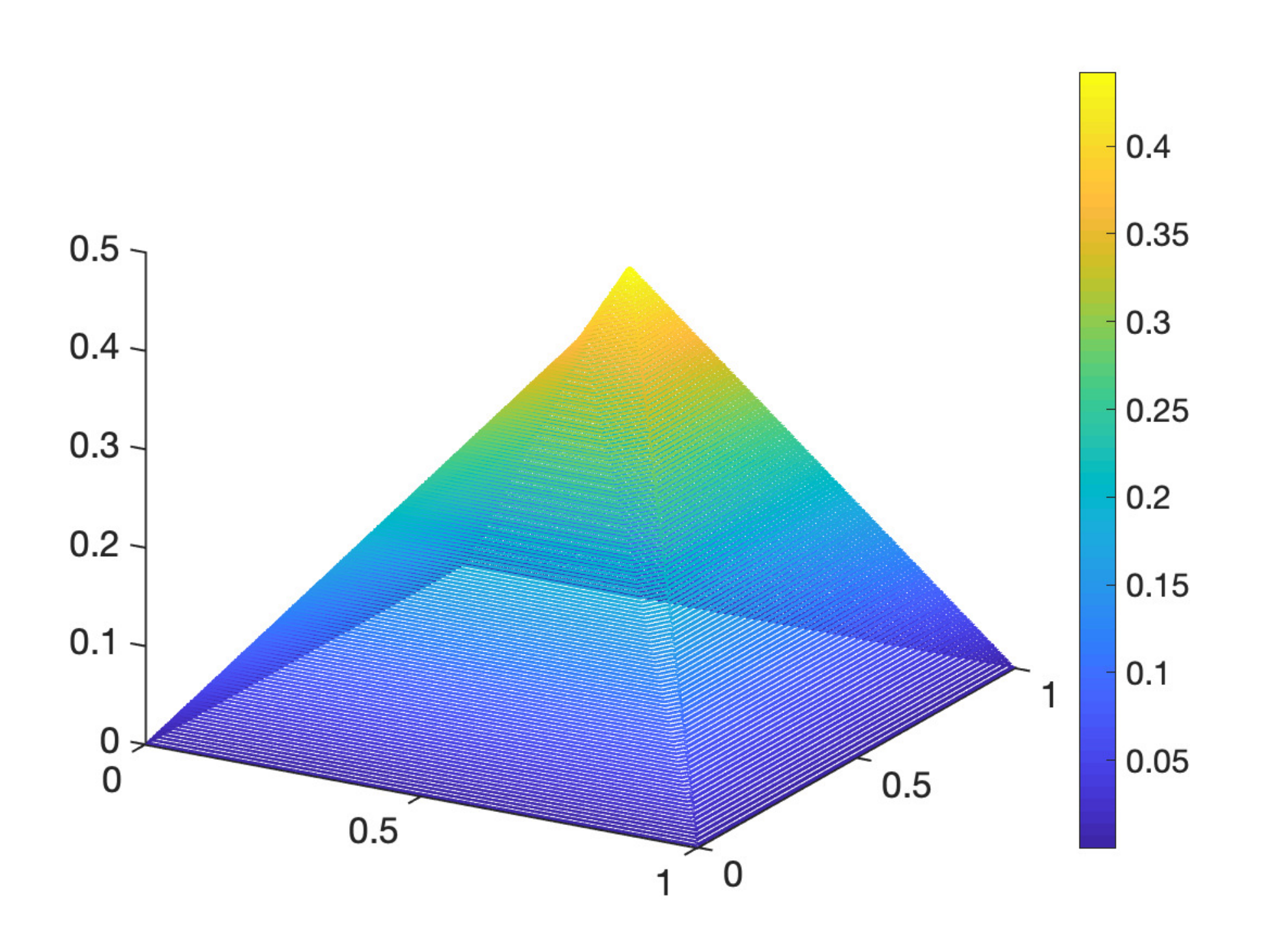}		
				\includegraphics[width=0.32\textwidth]{figures/ex_2/gu_mesh7_f_1_a_1_ex2}			
				\includegraphics[width=0.32\textwidth]{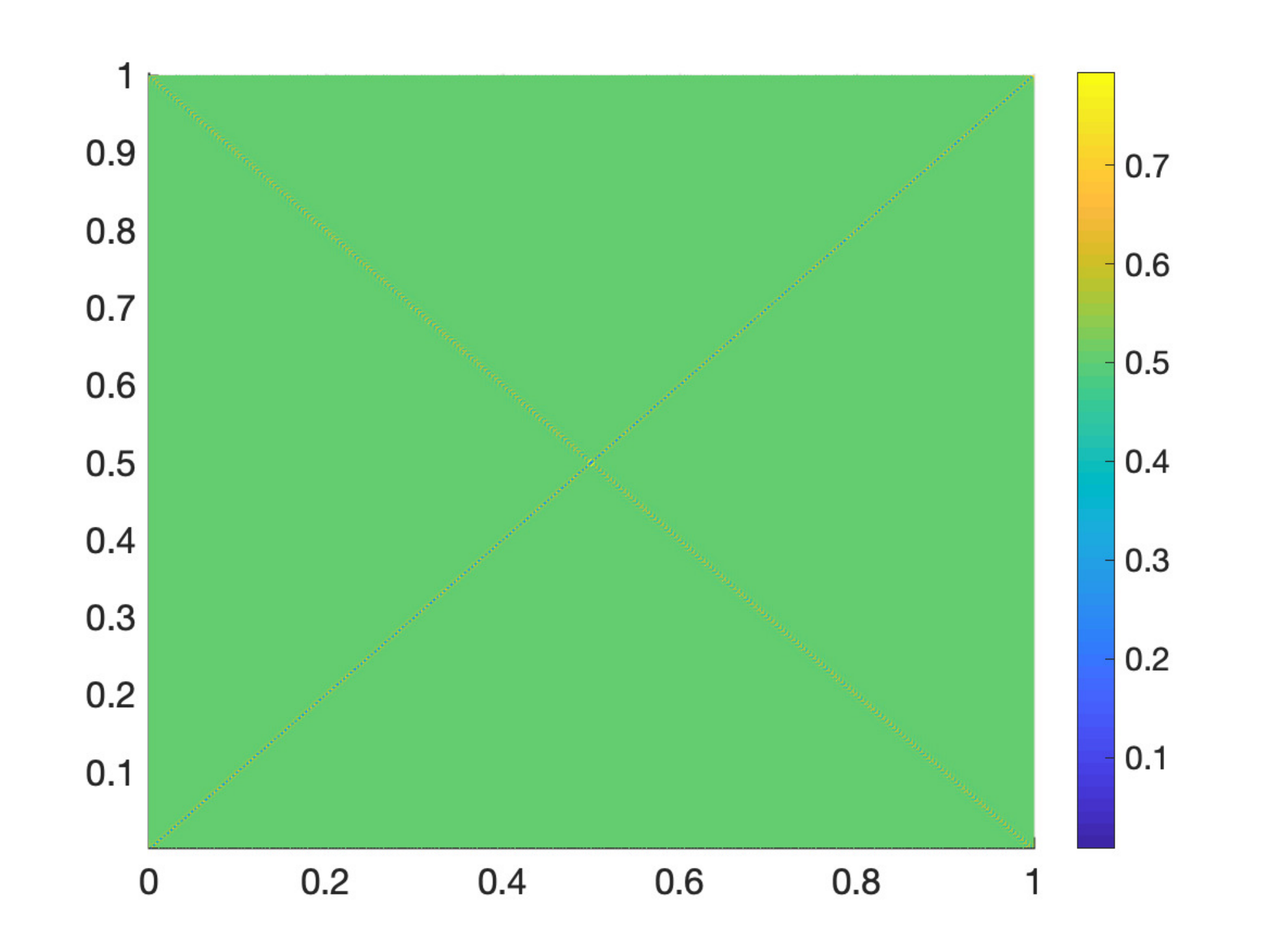}			
				\includegraphics[width=0.32\textwidth]{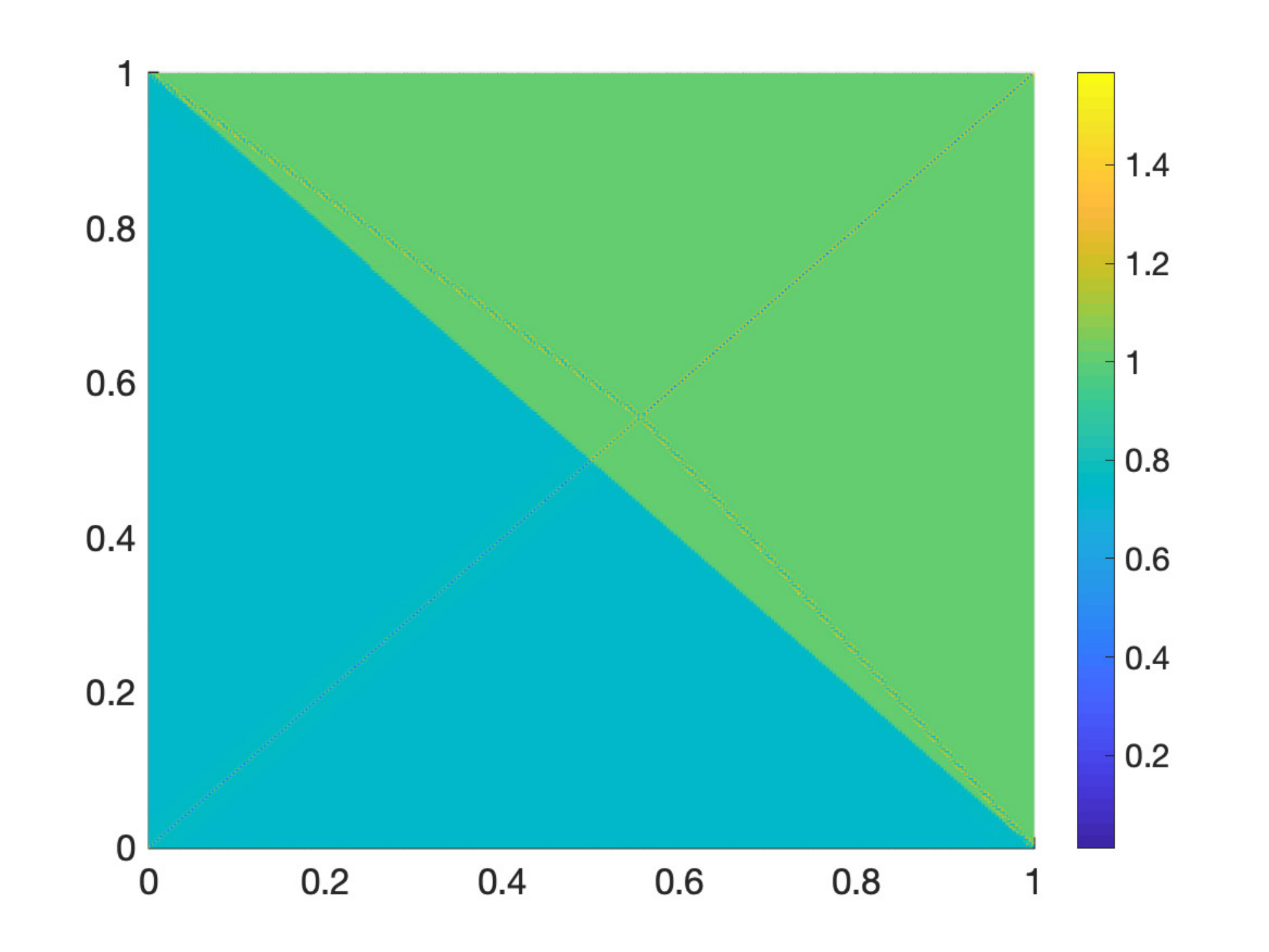}	
				\includegraphics[width=0.32\textwidth]{figures/ex_2/p_mesh7_f_1_a_1_ex2}
				\includegraphics[width=0.32\textwidth]{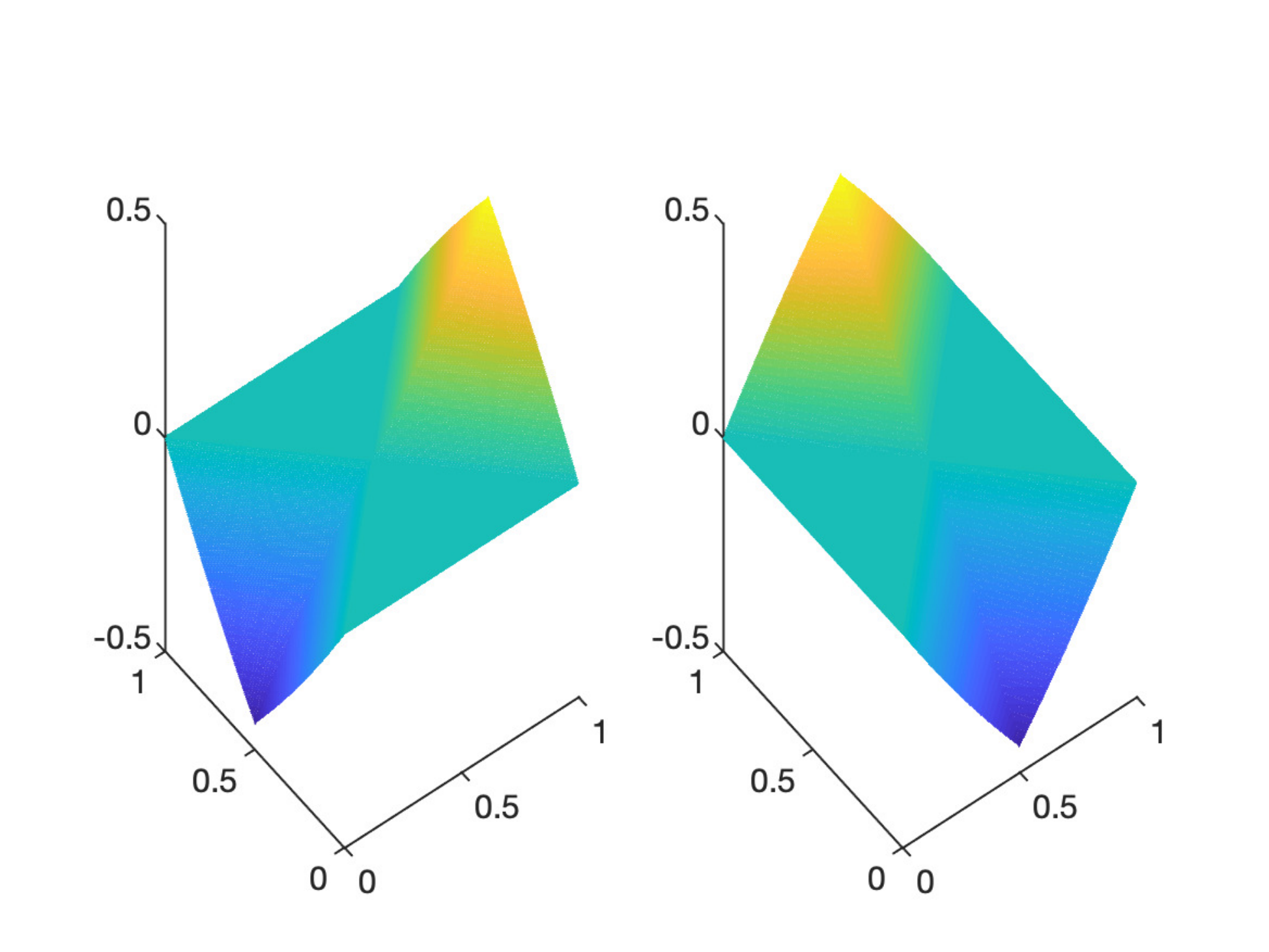}	
				\includegraphics[width=0.32\textwidth]{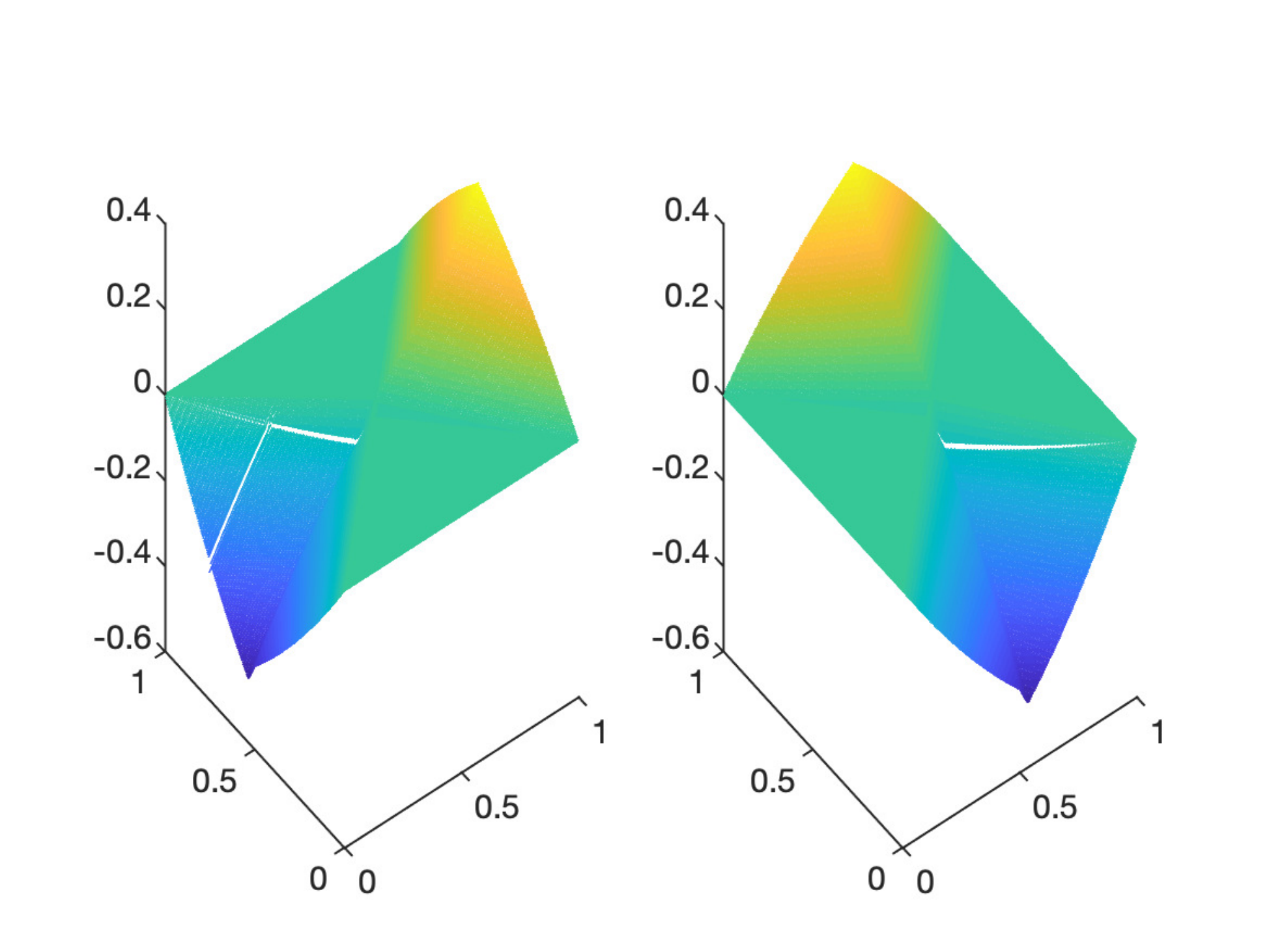}	
				\caption{{\bf Example~1 (fixed $f$, varying $\alpha$).} 
					The rows correspond to $u_h$, $|\nabla u_h|_{2}$, and $\bm p_h$. The first two columns 
					represent constant  $\alpha = 1$ and 0.5. The third column corresponds to $\alpha$ with jump
					discontinuity. In all cases, we observe that the gradient constraints are active in the entire region,
					except on a set of measure zero. Moreover, discontinuity in $\bm p$ in the last column is clearly 
					visible.}
				\label{f:ex_2_2}
			\end{figure}
			
		}
	\end{example}

	\newpage

	\begin{example}
		{\rm 			
			In this example, we set  
			\[
			f = 10^{-3} + u_0,
			\]
			where  
			\[
			u_0 := \left\{
			\begin{array}{ll}	
			\min\{0.2, 0.5 (x^2+y^2) \}, & y \le 1-x, \\
			\max\left\{ 1 - 5 \sqrt{ (x - 0.7)^2 + (y-0.7)^2 }, \min\left\{0.2, 0.5 (x^2+y^2) \right\} \right\}, & 1-x < y, \\
			0 & \mbox{otherwise} \, .	
			\end{array}
			\right. 
			\]
			Moreover, we set $\alpha = 2.5$. Figure~\ref{f:ex3_f} (left panel), shows a plot of $f$. 
			\begin{figure}
				\includegraphics[width=0.45\textwidth]{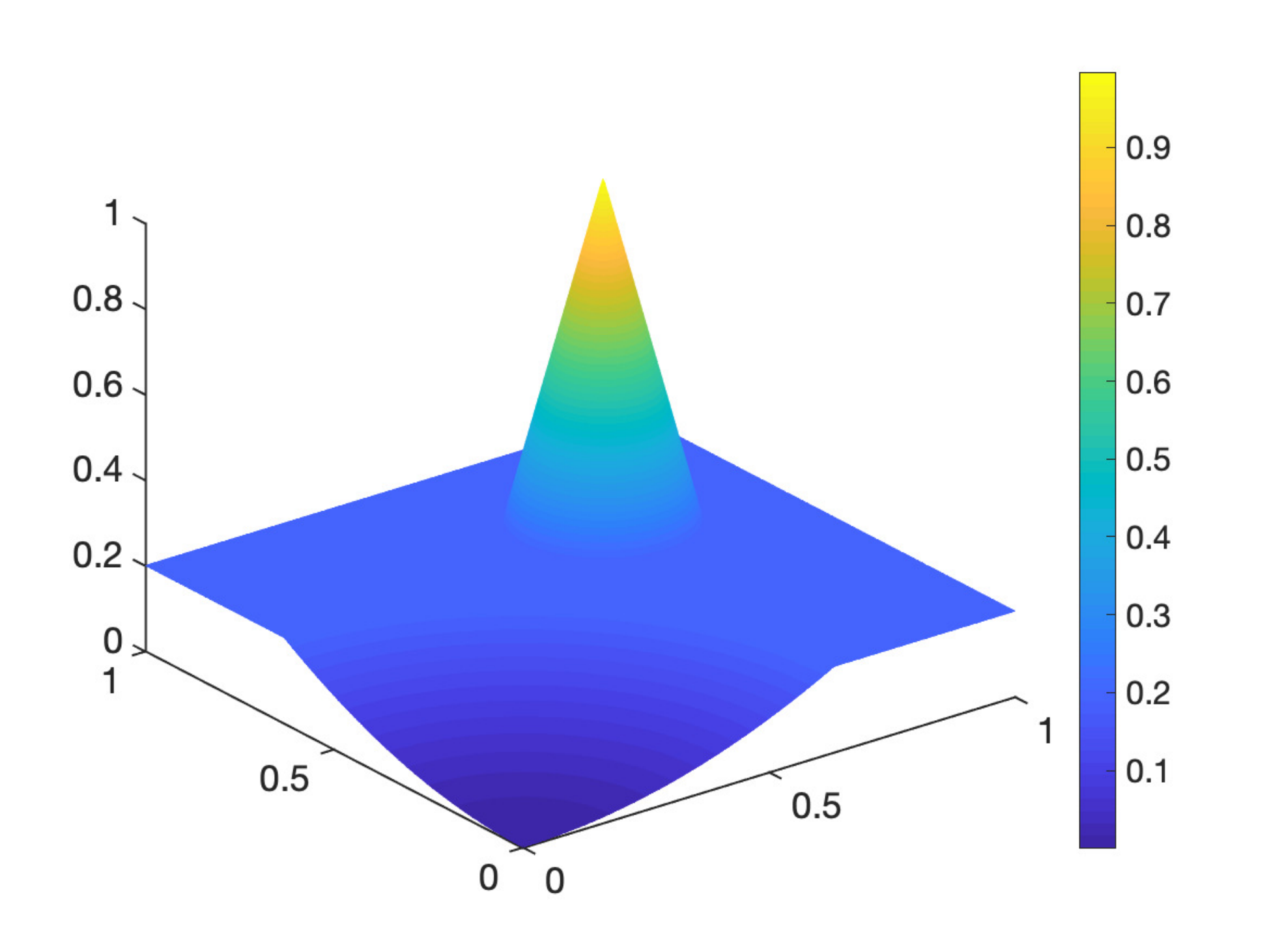}
				\includegraphics[width=0.45\textwidth]{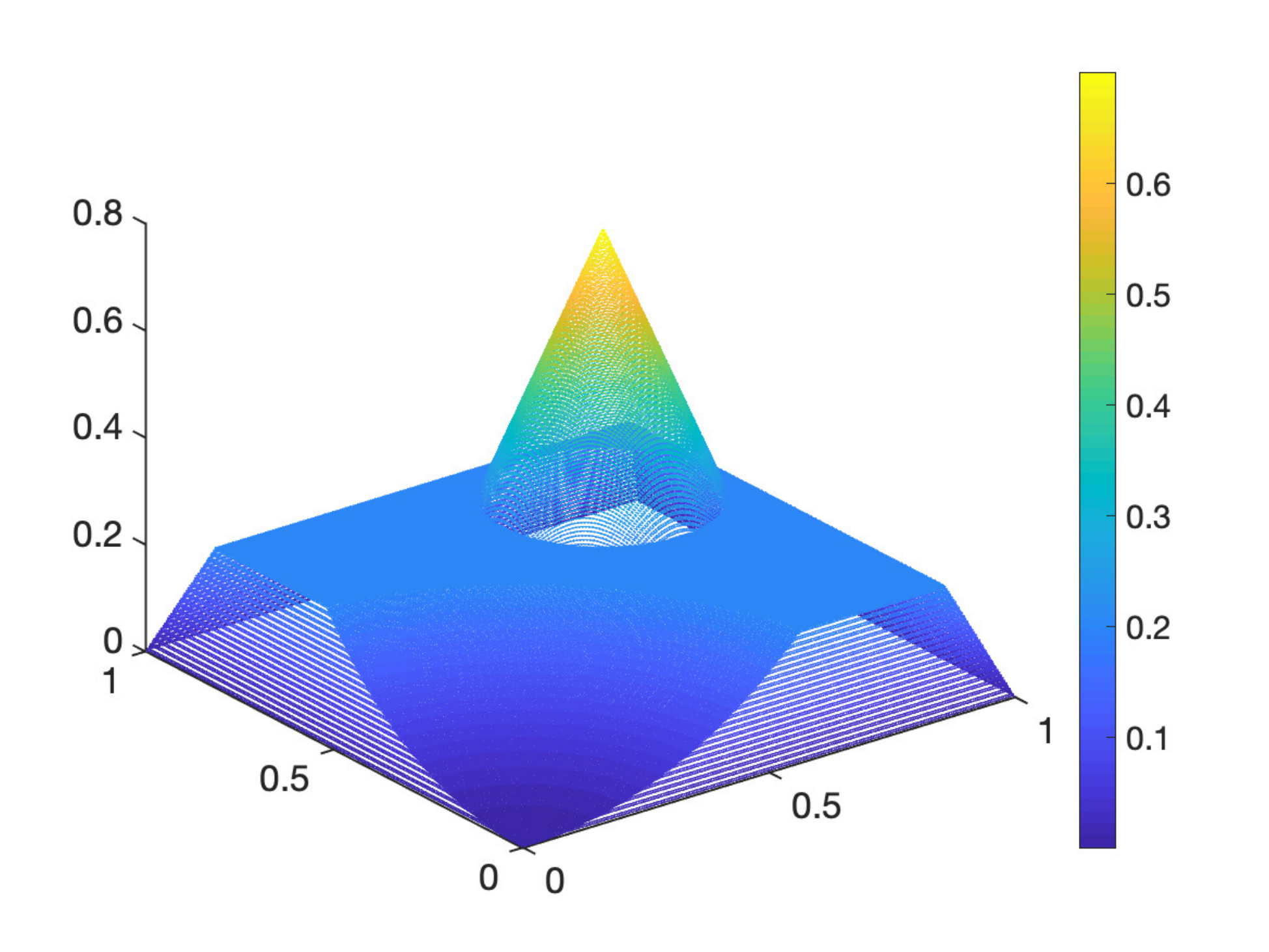}
				\caption{{\bf Example~2 ($\alpha = 2.5$).} Left panel: $f$. Right panel: the computed solution $u_h$.}
				\label{f:ex3_f}
			\end{figure}
			Figure~\ref{f:ex3_f} (right panel) shows the computed solution $u_h$. 
			In Figure~\ref{f:ex_3_1}, we have shown $|\nabla u_h|_{2}$ (left panel), and $\bm p_h$ (right
			panel). Notice that, the gradient constraints are active. Moreover, we also observe significant flat regions, 
			where the gradient is zero. 
			
			\begin{figure}[htb]
				\centering
				\includegraphics[width=0.45\textwidth]{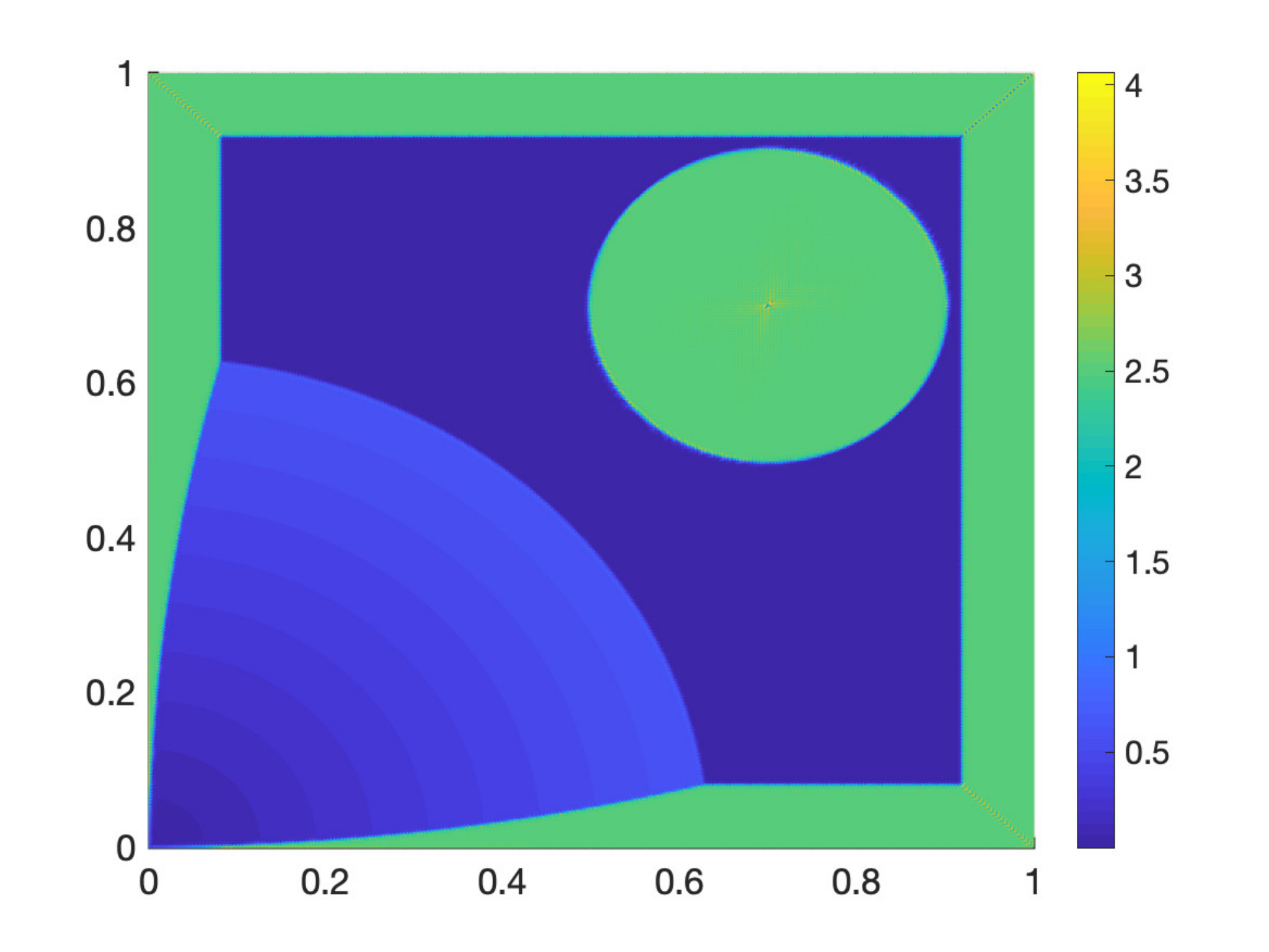}	
				\includegraphics[width=0.45\textwidth]{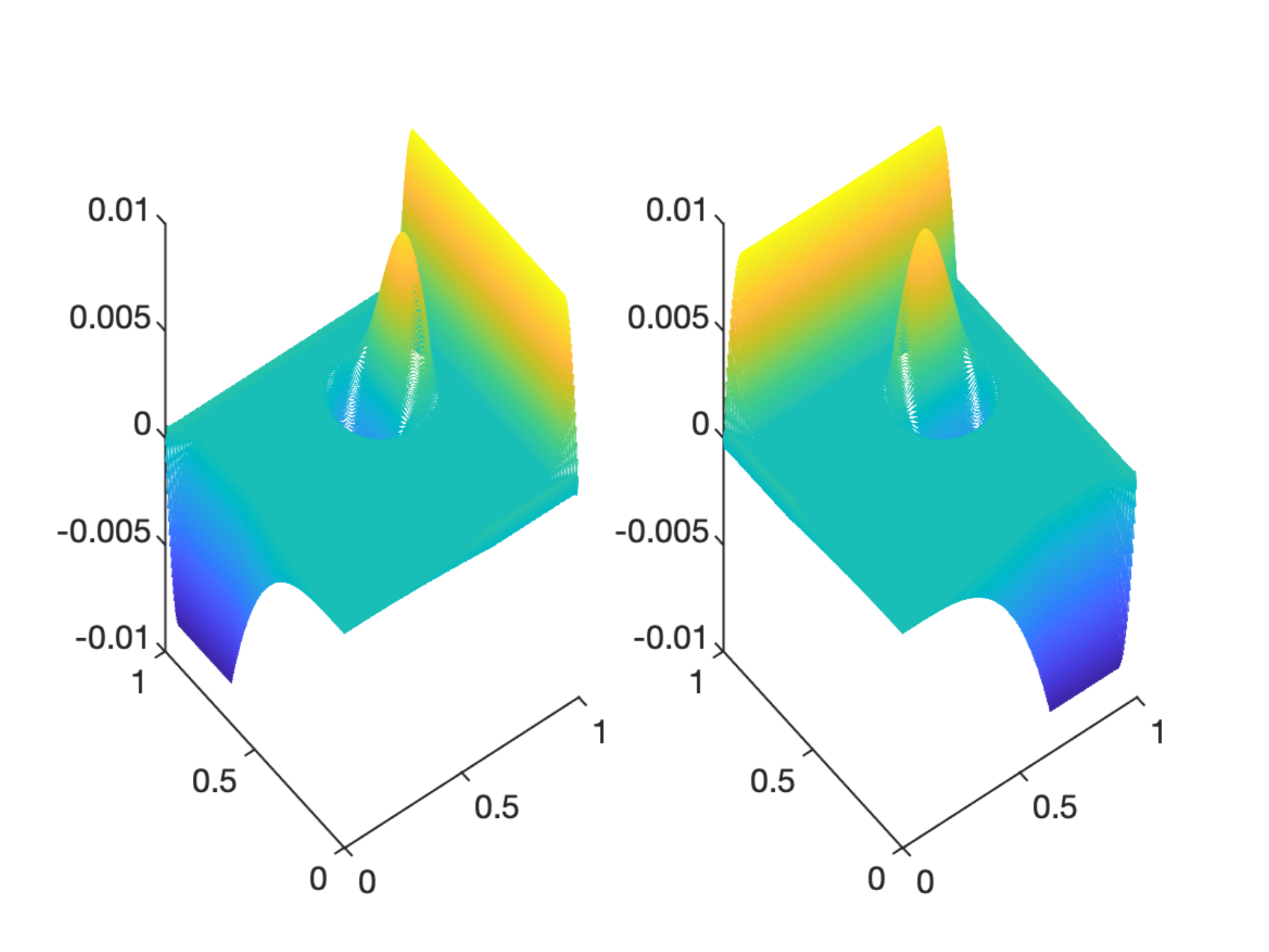}
				\caption{{\bf Example~3 ($\alpha = 2.5$).} Left panel: $|\nabla u_h|_{2}$. Right panel: the computed solution $\bm p_h$.}
				\label{f:ex_3_1}
			\end{figure}
			
			In Figure~\ref{f:ex_3_2} have also displayed $u_h$, $|\nabla u_h|_{2}$ and $\bm p_h$ 
			when $\alpha = 1.5$. 
			
			\begin{figure}[htb]
				\centering
				\includegraphics[width=0.45\textwidth]{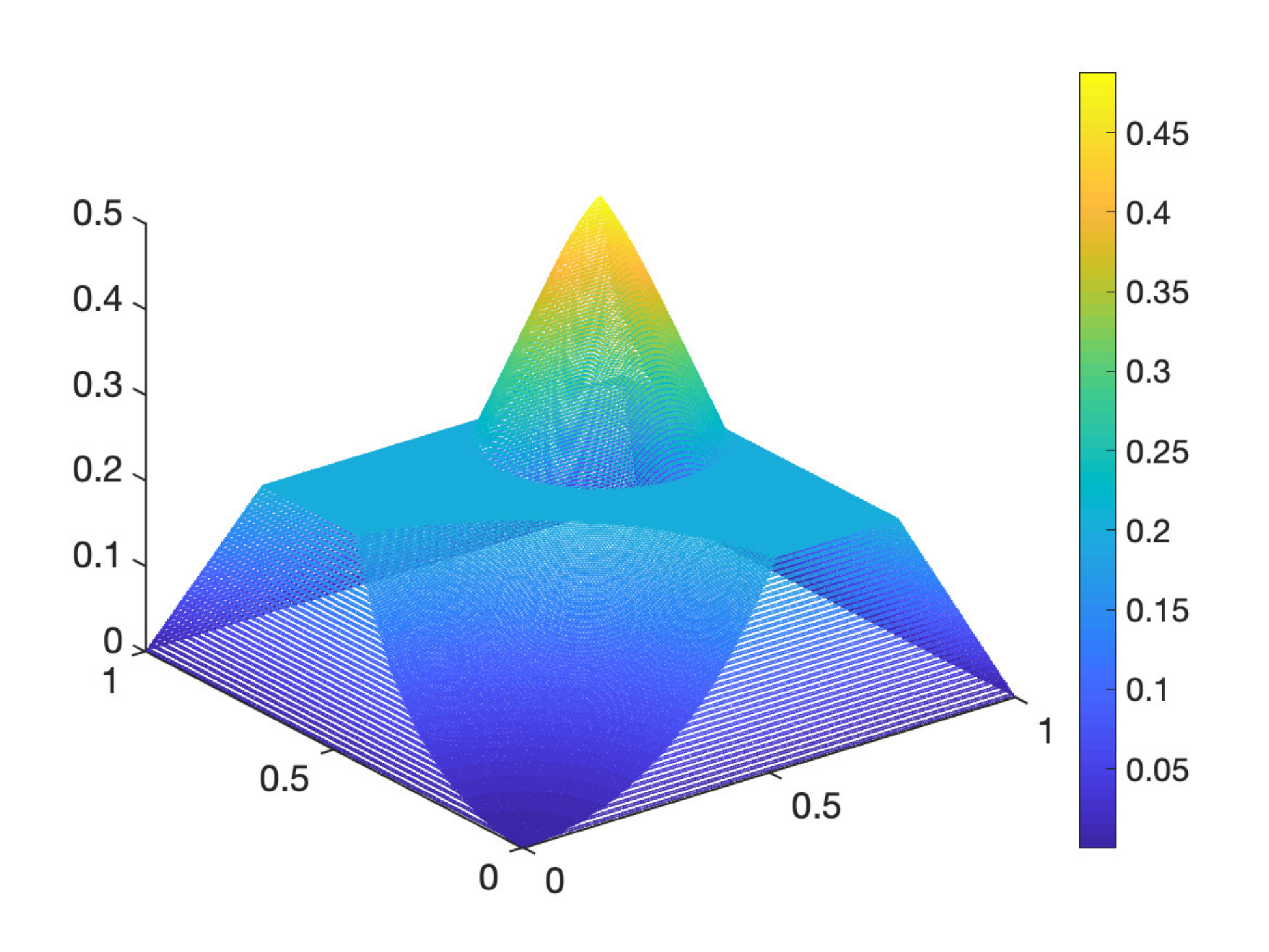}	
				\includegraphics[width=0.45\textwidth]{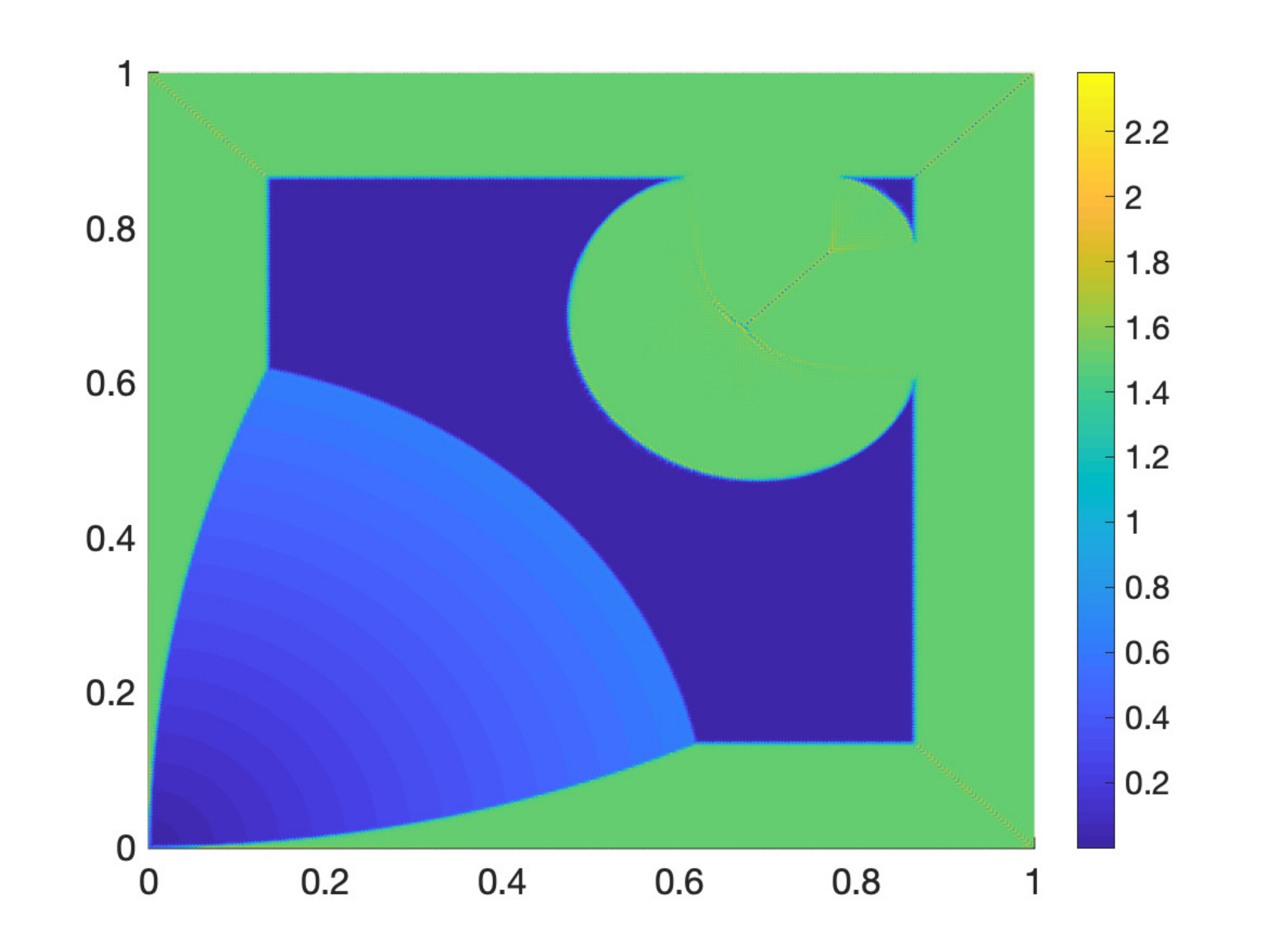}	
				\includegraphics[width=0.45\textwidth]{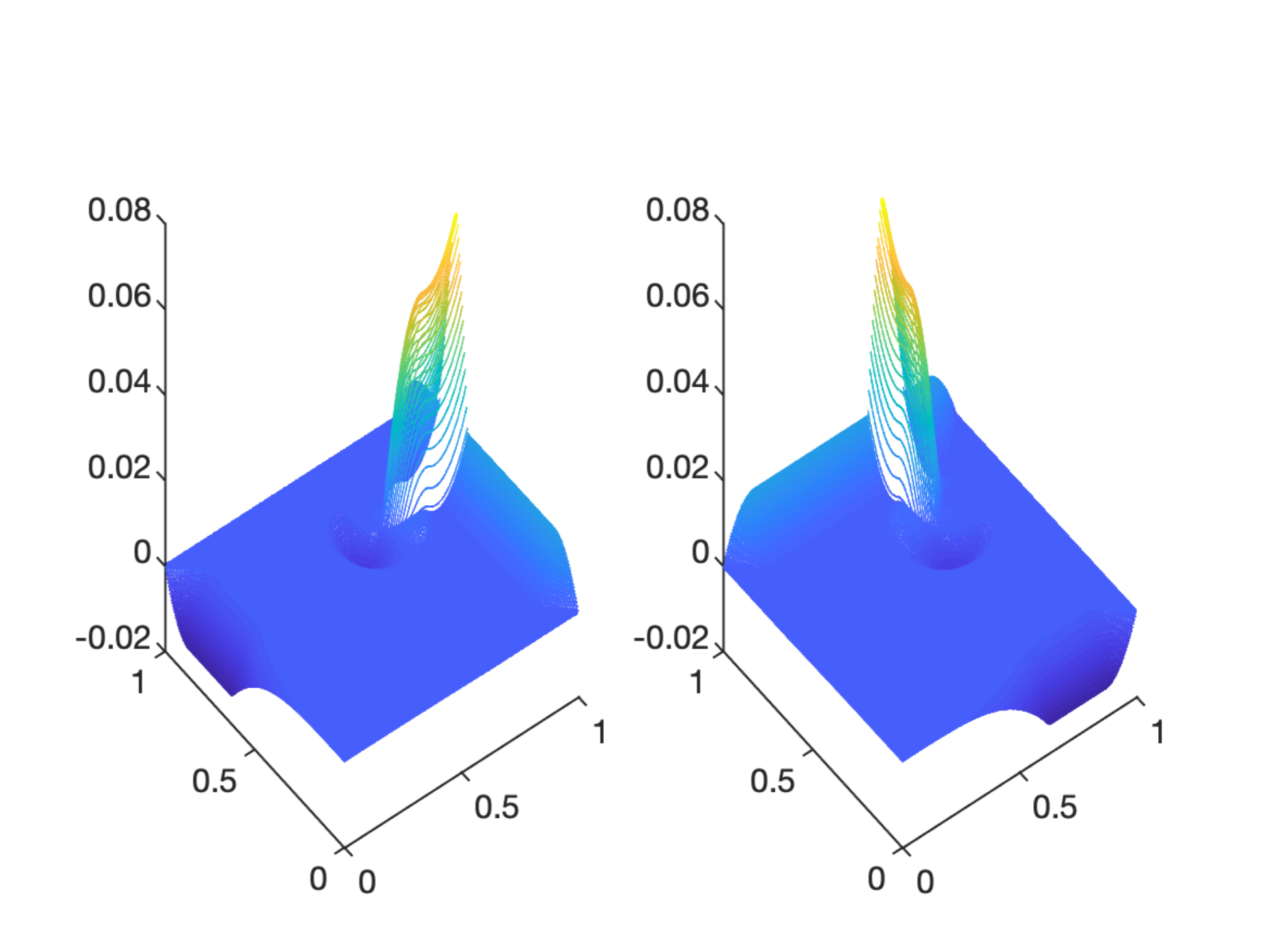}
				\caption{{\bf Example~3 ($\alpha = 1.5$).} Top row: $u_h$ and $|\nabla u_h|_{2}$. Bottom row: $\bm p_h$.}
				\label{f:ex_3_2}
			\end{figure}
			
		}
	\end{example}
	
	\begin{example}
		{\rm  
		In this example, we consider $f$ given by 
			\[
				f(x,y) = \left\{ \begin{array}{ll}
						0.25 &  (x,y) \in \Omega, \ 0.5 \le y, \\
						0      & \mbox{otherwise} .
					      \end{array}	 
					     \right. 
			\]
		The main novelty and challenge in this example is the fact that we let $\alpha$ to be a measure. Specifically
		\begin{equation*}
		\int_{\Omega} v \dif \alpha	=\int_{\Omega} v \dif x	+10^2 \int_{\omega} v \dif\mathcal{H}^{ 1}, 
		\end{equation*}
		for all $v\in C_c^\infty(\Omega)$ and
		where $\omega:=\{(x,y)\in\Omega:y=0.5\}$, i.e., $\alpha$ consists in the Lebesgue measure $\dif x$ and a weighted line measure on $\omega$.

		Let $h$ denotes the meshsize, then $\alpha^h$ is approximated as
		\begin{equation*}
			\dif\alpha^h=\left(1+\frac{\chi_{\omega^h}(x,y)}{h}\right)\dif x,
		\end{equation*} 
		where 
		\[
			\omega^h := \{ (x,y) \in \Omega \, : \, 0.5 - 10^{2} h \le y \le 0.5, \ x \in (0,1) \} .
		\] 		
				As $h \downarrow 0$, we approximate the measure in the sense that $\int_{\Omega} v \dif \alpha^h\to \int_{\Omega} v \dif \alpha$ for all $v\in C_c^\infty(\Omega)$.

		When $h = 8.4984 \times 10^{-5}$, the results are shown in Figure~\ref{f:ex4} (top row). Finally, when $h = 2.1412 \times 10^{-5}$ the results are provided in 
		Figure~\ref{f:ex4} (bottom row).  We notice that as $h \downarrow 0$, we indeed approximate the measure: In fact, we observe a clear discontinuity on the solution $u$, the size of the jump is below $100$ which is the upper bound on the distributional gradient on $\omega$. 
		
		\begin{figure}[htb]
			\includegraphics[width=0.45\textwidth]{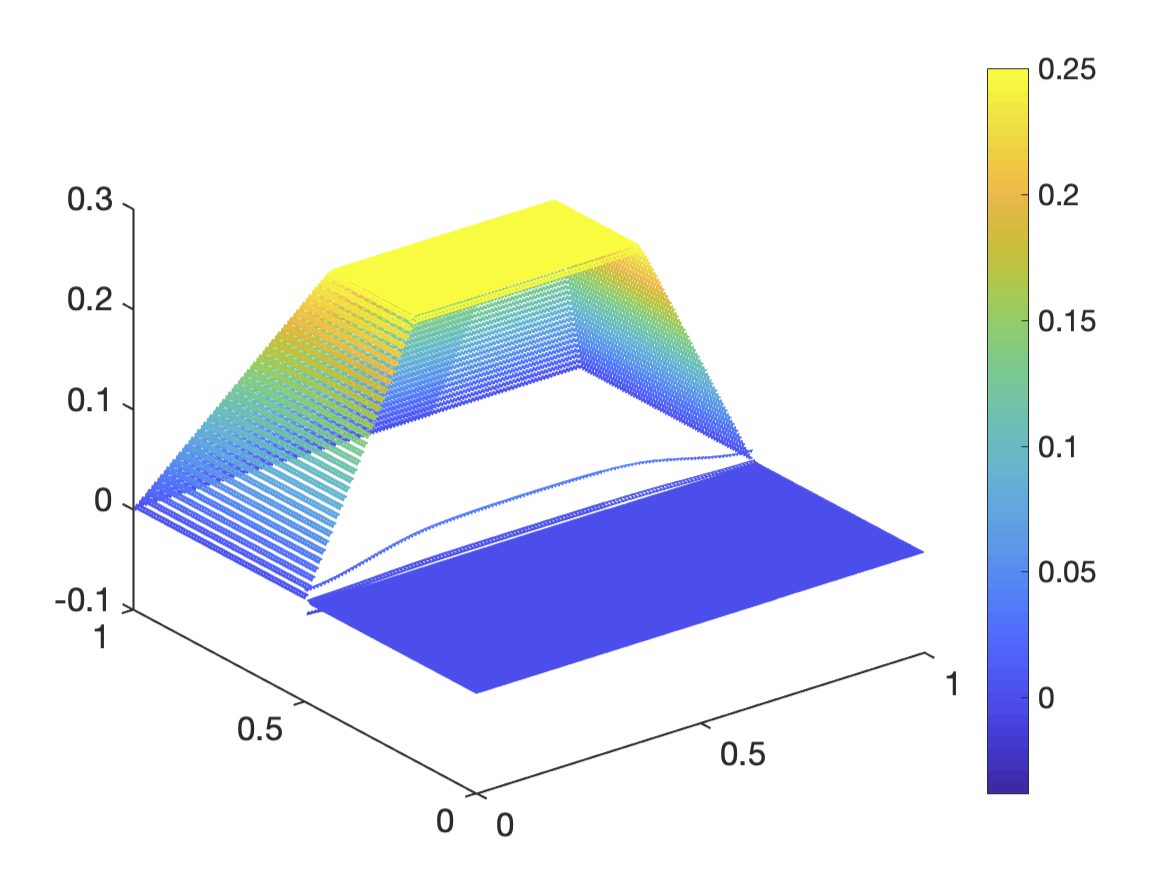}
			\includegraphics[width=0.45\textwidth]{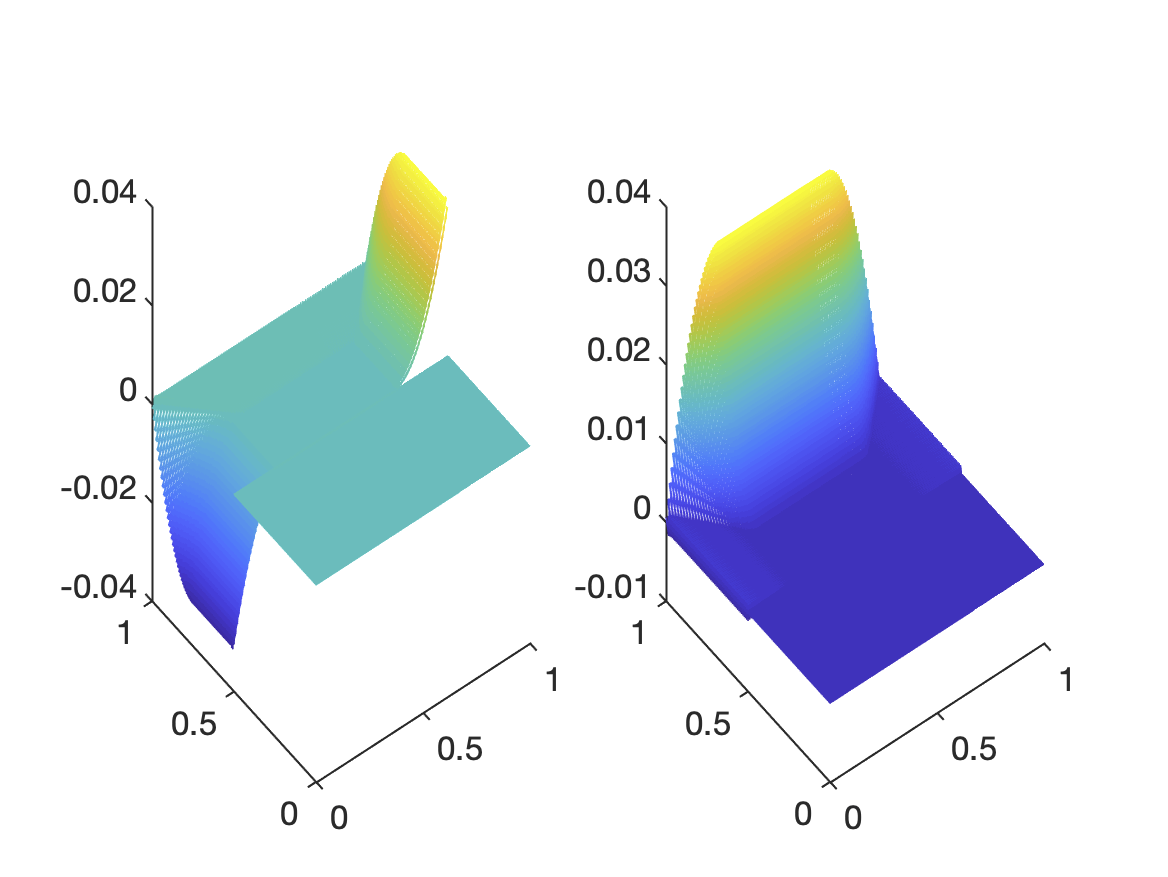}
			\includegraphics[width=0.45\textwidth]{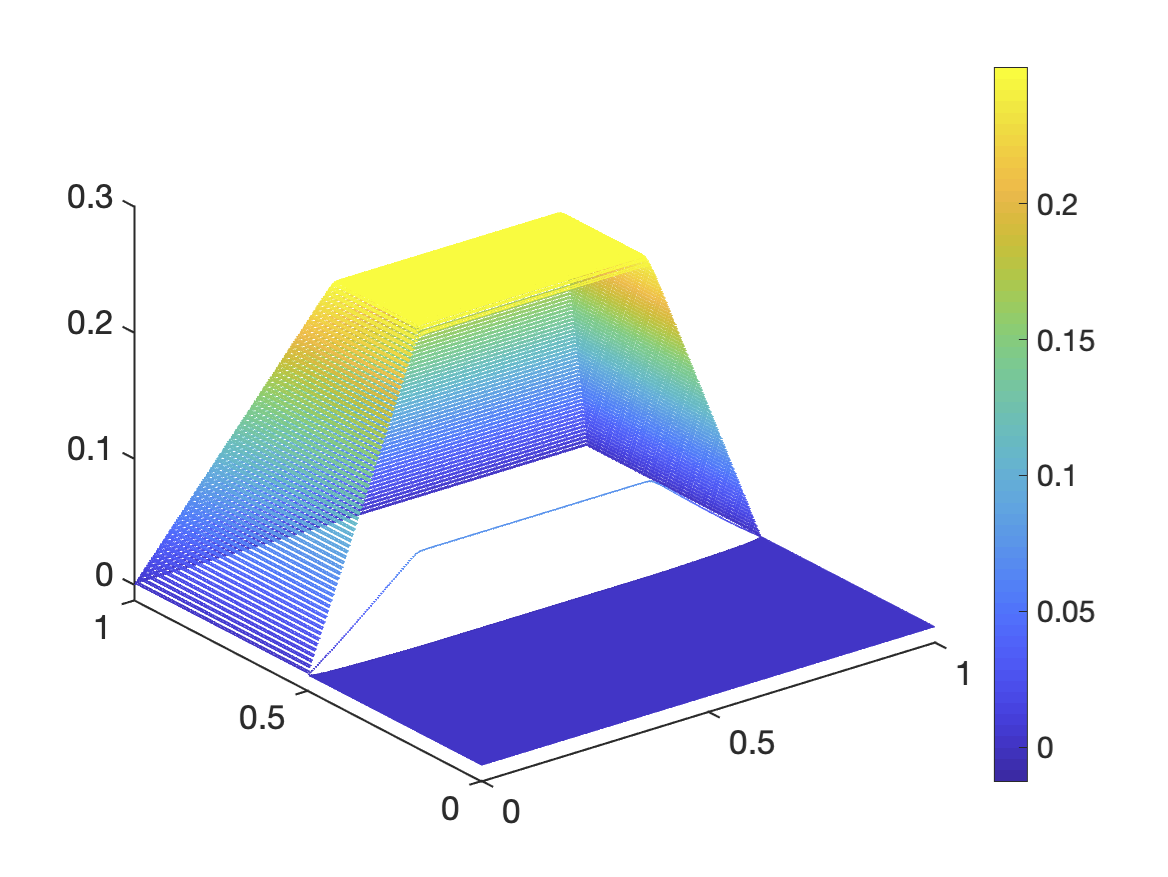}
			\includegraphics[width=0.45\textwidth]{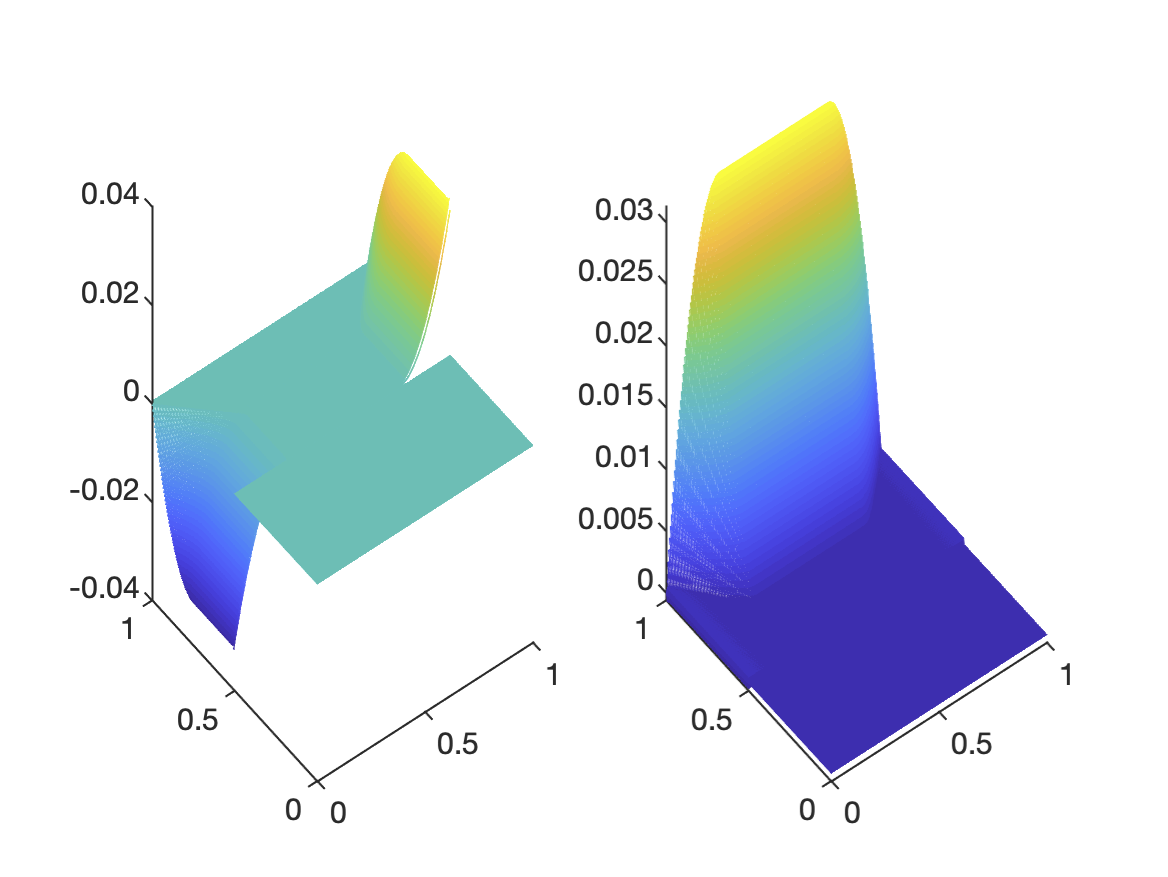}
			\caption{{\bf Example~4 ($\alpha$ measure).}  Top row: $u_h$ and $\bm p_h$ when the meshsize $h = 8.4984 \times 10^{-5}$.
			Bottom row: $u_h$ and $\bm p_h$ when the meshsize $h = 2.1412 \times 10^{-5}$. We notice that as $h \downarrow 0$,
			we accurately approximate the action of the measure  $\alpha$ as a distributional gradient constraint.}
			\label{f:ex4}
		\end{figure}
		
		}
	\end{example}
	
	\newpage

	%%%%%%%%%%%%%%%%%%%%%%%%%%%%%%%%%%%%%%%%%%%%%%%%%%%%%%%%
	\bibliographystyle{abbrv}
	\bibliography{references}
	%%%%%%%%%%%%%%%%%%%%%%%%%%%%%%%%%%%%%%%%%%%%%%%%%%%%%%%%

	%%%%%%%%%%%%%%%%%%%%%%%%%%%%%%%%%%%%%%%%%%%%%%%%%%%%%%%
\end{document}